\documentclass[10pt]{amsart}
\usepackage{amsfonts,amssymb,amscd,amsmath,enumerate,verbatim,calc}

\setbox0=\hbox{$+$}
\newdimen\plusheight
\plusheight=\ht0
\def\+{\;\lower\plusheight\hbox{$+$}\;}

\setbox0=\hbox{$-$}
\newdimen\minusheight
\minusheight=\ht0
\def\-{\;\lower\minusheight\hbox{$-$}\;}

\setbox0=\hbox{$\cdots$}
\newdimen\cdotsheight
\cdotsheight=\plusheight%\ht0
\def\cds{\lower\cdotsheight\hbox{$\cdots$}}

\setbox0=\hbox{$+$}
\numberwithin{equation}{section}
 \theoremstyle{plain}
\newtheorem{theorem}{Theorem}[section]

\newtheorem{lemma}[theorem]{Lemma}
\newtheorem{corollary}[theorem]{Corollary}
\newtheorem{property}[theorem]{Property}
\newtheorem{remark}[theorem]{Remark}
\newtheorem{definition}[theorem]{Definition}
\newtheorem{proposition}[theorem]{Proposition}
\newtheorem{example}[theorem]{Example}

\newcommand{\be}{\begin{equation}}
\newcommand{\ee}{\end{equation}}
\newcommand{\beba}{\begin{equation}\begin{array}{lcl}}
\newcommand{\eaee}{\end{array}\end{equation}}
\newcommand{\bea}{\begin{eqnarray}}
\newcommand{\eea}{\end{eqnarray}}
\newcommand{\ba}{\begin{array}}
\newcommand{\ea}{\end{array}}
\newcommand{\bal}{\begin{align}}
\newcommand{\eal}{\end{align}}

\begin{document}

\title[Fibonacci Numbers, Generalized Fibonacci Numbers and Generalized Fibonacci Polynomial] {Some Properties of Fibonacci Numbers, Generalized Fibonacci Numbers and Generalized Fibonacci Polynomial Sequences}
\author{Alexandre Laugier}
\address{Lyc{\'e}e professionnel Tristan Corbi{\`e}re, 16 rue de Kerveguen - BP 17149, 29671 Morlaix cedex, France}
\email{laugier.alexandre@orange.fr}
\author{Manjil P. Saikia}
\footnotetext[1]{Corresponding Author: manjil@gonitsora.com.}
\address{Mathematics Diploma Student, The Abdus Salam International Centre for Theoretical Physics, Strada Costiera, 11-34151, Trieste, Italy}
\email{manjil@gonitsora.com, msaikia@ictp.it}
\thanks{The work of the second author was supported by DST, Govt. of India, INSPIRE Scholarship 422/2009.}

\begin{abstract}
In this paper we study the Fibonacci numbers and derive some interesting properties and recurrence relations. We prove some charecterizations for $F_p$, where $p$ is a prime of a certain type. We also define period of a Fibonacci 
sequence modulo an integer, $m$ and derive certain interesting
properties related to them. Afterwards, we derive some new properties of a class of generalized Fibonacci numbers. In the last part of the paper we introduce some generalized Fibonacci polynomial sequences and we derive some results related to them.
\end{abstract}

\maketitle
\tableofcontents

\vskip 3mm

\noindent{\footnotesize Fibonacci numbers, congruences, period of Fibonacci sequence.}

\vskip 3mm

\noindent{\footnotesize 2010 Mathematical Reviews Classification
Numbers: 11A07, 11B37, 11B39, 11B50.}

\section{{Preliminaries}}

This paper is divided into six sections. This section is devoted to stating few results that will be used in the remainder of the paper. We also 
set the notations to be used and derive few simple results that will come in handy in our treatment. In Section 2, we charecterize numbers $5k+2$, which are primes 
with $k$ being an odd natural number. In Section 3, we prove more general results than given in Section 2. In Section 4, we define the 
period of a Fibonacci sequence modulo some number and derive many properties of this concept. In Section 5, we shall devote to the study of a class of generalized Fibonacci numbers and derive some interesting results related to them. Finally, in Section 6, we define some generalized Fibonacci polynomial sequences and we obtain some results related to them.

We begin with the following famous results without proof except for
some related properties.

\begin{lemma}[Euclid]\label{euclid}
If $ab\equiv 0\pmod p$ with $a,b$ two integers and $p$ a prime, then
either $p|a$ or $p|b$. 
\end{lemma}

\begin{remark}
In particular, if $\gcd(a,b)=1$, $p$ divides only one of the numbers
$a,b$.
\end{remark}

\begin{property}\label{peqc}
Let $a,b$ two positive integers, $m,n$ two integers
such that $(|m|,|n|)=1$ and $p$ a natural number. Then
$$
ma\equiv nb\pmod p$$
if and only if there exists $c\in\mathbb{Z}$ such that
$$a\equiv nc\pmod p$$ and 
$$b\equiv mc\pmod p.$$
\end{property}

\begin{proof}
Let $a,b$ two positive integers, $m,n$ two integers
such that $(|m|,|n|)=1$ and $p$ a natural
number.

If there exists an integer $c$ such that 
$a\equiv nc\pmod p$ and $b\equiv mc\pmod p$, then $ma\equiv
mnc\pmod p$ and $nb\equiv mnc\pmod p$. So, we have $ma\equiv
nb\pmod p$.

Conversely, if $ma\equiv nb\pmod p$ with $(|m|,|n|)=1$, then from Bezout's identity, there exist three integers $u,v,k$ such that
$$
um+vn=1
$$
and $$
ma-nb=kp.
$$
So, we have
$$
ukpm+vkpn=ma-nb
$$ and 
$$
m(a-ukp)=n(b+vkp).
$$
Since $|m|,|n|$ are relatively prime, from the Lemma \ref{euclid},
it implies that there exist two integers $c,d$ such that
$$
a-ukp=nc,
$$ and 
$$
b+vkp=md.
$$
It results that $mnc=mnd$ and so $c=d$. Therefore, we obtain
$$
a=nc+ukp\equiv nc\pmod p,
$$ and 
$$
b=mc-vkp\equiv mc\pmod p.
$$
\end{proof}

\begin{remark}
Using the notations given in the proof of Property \ref{peqc}, we
can see that if there exists an integer $c$ such that 
$a\equiv nc\pmod p$ and $b\equiv mc\pmod p$ with $(|m|,|n|)=1$ and $p$
a natural number, then we have
$$
ub+va\equiv (um+vn)c\equiv c\pmod p
$$
Moreover, denoting by $g$ the $gcd$ of $a$ and $b$, if $a=gn$ and
$b=gm$, then $ub+va=(um+vn)g=g$, then $g\equiv c\pmod p$.
\end{remark}

\begin{theorem}[Fermat's Little Theorem]\label{flt}
 If $p$ is a prime and $n\in \mathbb{N}$ relatively prime to $p$, then
 $n^{p-1}\equiv 1~(\textup{mod}~p)$. 
\end{theorem}

\begin{theorem}\label{t1.1}
If $x^2\equiv 1\pmod p$ with $p$ a prime, then either $x\equiv 1\pmod p$ or $x\equiv p-1\pmod p$.
\end{theorem}

\begin{proof}If $x^2\equiv 1\pmod p$ with $p$ a prime, then we have
$$
x^2-1\equiv 0\pmod p
$$
$$
(x-1)(x+1)\equiv 0\pmod p
$$

$x-1\equiv
0\pmod p$ or $x+1\equiv 0\pmod p$. It is equivalent to say that
$x\equiv 1\pmod p$ or $x\equiv -1\equiv p-1\pmod p$.
\end{proof}

\begin{definition}
Let $p$ be an odd prime and $gcd(a,p)=1$. If the congruence $x^2\equiv
a\pmod p$ has a solution, then $a$ is said to be a quadratic residue
of p. Otherwise, a is called a quadratic nonresidue of p.
\end{definition}

\begin{theorem}[Euler]\label{eul}
 Let $p$ be an odd prime and $\textup{gcd}(a,p)=1$. Then $a$ is a
 quadratic nonresidue of $p$ if and only if 
$a^{\frac{p-1}{2}}\equiv -1~(\textup{mod}~p)$.
\end{theorem}

\begin{definition}
 Let $p$ be an odd prime and let $\textup{gcd}(a,p)=1$. The \textit{Legendre symbol} $(a/p)$ is defined to be equal to $1$ if $a$ is a quadratic residue of $p$ and 
is equal to $-1$ is $a$ is a quadratic non residue of $p$.
\end{definition}

\begin{property}\label{lp}
Let $p$ an odd prime and $a$ and $b$ be integers which are relatively
prime to $p$. Then the Legendre symbol has the following properties:
\begin{enumerate}%[label=(\roman*),labelsep=*,align=left,leftmargin=0.3in]
\item If $a\equiv b\pmod p$, then $(a/p)=(b/p)$.
\item $(a/p)\equiv a^{(p-1)/2}\pmod p$
\item $(ab/p)=(a/p)(b/p)$
\end{enumerate}
\end{property}

\begin{remark}\label{rmpls}
Taking $a=b$ in (3) of Property \ref{lp}, we have
$$
(a^2/p)=(a/p)^2=1.
$$
\end{remark}

\begin{lemma}[Gauss]\label{GL}
Let $p$ be an odd prime and let $gcd(a,p)=1$. If $n$ denotes the
number of integers in the set $S=\left\{a,2a,3a,\ldots,
\left(\frac{p-1}{2}\right)a\right\}$, whose remainders upon division
by $p$ exceed $p/2$, then
$$
(a/p)=(-1)^n.
$$
\end{lemma}

\begin{corollary}\label{CofGL}
If $p$ is an odd prime, then
$$
(2/p)=\left\{\begin{array}{ccccc}
1 & if & p\equiv 1\pmod 8 & or & p\equiv 7\pmod 8 ,
\\
-1 & if & p\equiv 3\pmod 8 & or & p\equiv 5\pmod 8 .
\end{array}
\right.
$$
\end{corollary}

\begin{theorem}[Gauss' Quadratic Reciprocity Law]\label{GQRL}
If $p$ and $q$ are distinct odd primes, then
$$
(p/q)(q/p)=(-1)^{\frac{p-1}{2}\cdot\frac{q-1}{2}}.
$$
\end{theorem}

\begin{corollary}\label{CofGQRL}
If $p$ and $q$ are distinct odd primes, then
$$
(p/q)=\left\{\begin{array}{ccccc}
(q/p) & if & p\equiv 1\pmod 4 & or & q\equiv 1\pmod 4 ,
\\
-(q/p) & if & p\equiv q\equiv 3\pmod 4 .
\end{array}
\right.
$$
\end{corollary}

Throughout this paper, we assume $k\in \mathbb{N}$, unless otherwise stated.

From (1) of Property \ref{lp}, Theorem \ref{GQRL}, the
corollary \ref{CofGL} and the
corollary \ref{CofGQRL},
we deduce the following result.

\begin{theorem}\label{ttt}
 $$(5/5k+2)=-1.$$
\end{theorem}

\begin{proof}
 Clearly $(5/5k+2)=(5k+2/5)$ since $5 \equiv 1 \pmod {4}$.

Again $(5k+2/5)=(2/5)$ since $5k+2\equiv 2 \pmod {5}$.

Also it is a well known fact that $(2/5)=-1$ since $5\equiv 5 \pmod {8}$.
\end{proof}

For proofs of the above theorems the reader is suggested to see
\cite{dmb} or \cite{mol}.

Let $p$ a prime number such that $p=5k+2$ with $k$ an odd positive
integer. From Property \ref{lp} and Theorem \ref{ttt} we have
$$
\left(5^{\frac{5k+1}{2}}\right)^2\equiv 1\pmod {5k+2}.
$$
From Theorem \ref{t1.1}, we have either $$5^{\frac{5k+1}{2}}\equiv
-1\pmod {5k+2}$$ or $$5^{\frac{5k+1}{2}}\equiv 5k+1\pmod {5k+2}.$$ Moreover, we can observe that
$$
5(2k+1)\equiv 1\pmod {5k+2}.
$$

\begin{theorem}\label{t1.4}
 $$
5^{\frac{5k+1}{2}}\equiv 5k+1\pmod {5k+2}
$$ where $5k+2$ is a prime.
\end{theorem}

The proof of Theorem \ref{t1.4} follows very easily from Theorems
\ref{eul}, \ref{ttt} and Property \ref{lp}.

\begin{theorem}\label{t16}
Let $r$ be an integer in the set $\{1,2,3,4\}$. Then, we have
$$
(5/5k+r)=\left\{\begin{array}{ccccc}
1 & if & r=1 & or & r=4,
\\
-1 & if & r=2 & or & r=3.
\end{array}
\right.
$$
\end{theorem}

\begin{proof}
We have $(5/5k+r)=(5k+r/5)$ since $5\equiv 1\pmod 4$.

Moreover, $(5k+r/5)=(r/5)$ since $5k+r\equiv r\pmod 5$.

Or, we have

If $r=1$, then using Theorem \ref{GQRL}, $(r/5)=1$.

If $r=2$, then using Corollary \ref{CofGL}, $(r/5)=-1$.

If $r=3$, then using Theorem \ref{GQRL}, $(r/5)=-1$.

If $r=4$, then since $(4/5)=(2^2/5)=1$ (see also Remark \ref{rmpls}),
$(r/5)=1$. 

\end{proof}

\begin{theorem}\label{t17}
Let $r$ be an integer in the set $\{1,2,3,4\}$. Then, if $5k+r$ is a
prime, we have
$$
5^{\frac{5k+r-1}{2}}\equiv\left\{\begin{array}{ccccc}
1\pmod {5k+r} & if & r=1 & or & r=4,
\\
-1\pmod {5k+r} & if & r=2 & or & r=3.
\end{array}
\right.
$$
\end{theorem}

The proof of Theorem \ref{t17} follows very easily from Theorems
\ref{eul}, \ref{t16} and Property \ref{lp}.

We fix the notation $[[1,n]]=\{1,2,\ldots,n\}$ throughout the rest of
the paper. We now have the following properties. 

\begin{remark}
Let $5k+r$ with $r\in[[ 1,4]]$ be a prime
number. Then
$$
k\equiv\left\{\begin{array}{ccccc}
0\pmod 2 & if & r=1 & or & r=3,
\\
1\pmod 2 & if & r=2 & or & r=4,
\end{array}
\right.
$$
or equivalently
$$
k\equiv r+1\pmod 2.
$$
\end{remark}

\begin{property}\label{p1.5}
$$
\binom{5k+1}{2l+1}\equiv 5k+1\pmod {5k+2},
$$
with $l\in[[ 0,\lfloor\frac{5k}{2}\rfloor]]$ and $5k+2$ is a prime.
\end{property}

\begin{proof}
Notice that for $l=0$ the property is obviously true.

We also have
$$
\binom{5k+1}{2l+1}=\frac{(5k+1)5k(5k-1)\ldots (5k-2l+1)}{(2l+1)!}.
$$
Or,
$$
\begin{array}{c}
5k\equiv -2\pmod {5k+2},
\\
5k-1\equiv -3\pmod {5k+2},
\\
\vdots
\\
5k-2l+1\equiv -(2l+1)\pmod {5k+2}.
\end{array}
$$
Multiplying these congruences we get
$$
5k(5k-1)\ldots (5k-2l+1)\equiv (2l+1)!\pmod {5k+2}.
$$
Therefore
$$
(2l+1)!\binom{5k+1}{2l+1}\equiv (5k+1)(2l+1)!\pmod {5k+2}.
$$
Since $(2l+1)!$ and $5k+2$ are relatively prime, we obtain
$$
\binom{5k+1}{2l+1}\equiv 5k+1\pmod {5k+2}.
$$
\end{proof}

We have now the following generalization.

\begin{property}\label{p1.20}
$$
\binom{5k+r-1}{2l+1}\equiv -1\pmod {5k+r},
$$
with $l\in[[ 0,\lfloor\frac{5k+r-2}{2}\rfloor]]$ and $5k+r$ is a prime
such that $r\in[[ 1,4]]$ and $k\equiv r+1\pmod 2$.
\end{property}

The proof of Property \ref{p1.20} is very similar to the proof of Property \ref{p1.5}

\begin{property}\label{p1.6}

$$
\binom{5k}{2l+1}\equiv 5k-2l\equiv -2(l+1)\pmod {5k+2}
$$
with $l\in[[ 0,\lfloor\frac{5k}{2}\rfloor]]$ and $5k+2$ is a prime.
\end{property}

\begin{proof}

Notice that for $l=0$ the property is obviously true.

We have
$$
\binom{5k}{2l+1}=\frac{5k(5k-1)\ldots (5k-2l+1)(5k-2l)}{(2l+1)!}.
$$
Or
$$
\begin{array}{c}
5k\equiv -2\pmod {5k+2},
\\
5k-1\equiv -3\pmod {5k+2},
\\
\vdots
\\
5k-2l+1\equiv -(2l+1)\pmod {5k+2}.
\end{array}
$$
Multiplying these congruences we get
$$
5k(5k-1)\ldots (5k-2l+1)\equiv (2l+1)!\pmod {5k+2}.
$$
Therfore
$$
(2l+1)!\binom{5k}{2l+1}\equiv (2l+1)!(5k-2l)\pmod {5k+2}.
$$
Since $(2l+1)!$ and $5k+2$ are relatively prime, we obtain
$$
\binom{5k+1}{2l+1}\equiv 5k-2l\equiv 5k+2-2-2l\equiv -2(l+1)\pmod {5k+2}.
$$
\end{proof}

We can generalize the above as follows.

\begin{property}\label{p1.22}
$$
\binom{5k+r-2}{2l+1}\equiv -2(l+1)\pmod {5k+r},
$$
with $l\in[[ 0,\lfloor\frac{5k+r-3}{2}\rfloor]]$ and $5k+r$ is a prime
such that $r\in[[ 1,4]]$ and $k\equiv r+1\pmod 2$
\end{property}

The proof of Property \ref{p1.22} is very similar to the proof of Property \ref{p1.6}.

In the memainder of this section we derive or state a few results involving the Fibonacci numbers. The Fibonacci sequence $(F_n)$ is defined by $F_0=0, F_1=1, F_{n+2}=F_n+F_{n+1}$ for $n\geq 0$.

From the definition of the Fibonacci sequences we can establish
the formula for the $nth$ Fibonacci number,
$$
F_n=\frac{\varphi^n-(1-\varphi)^n}{\sqrt{5}},
$$
where $\varphi=\frac{1+\sqrt{5}}{2}$ is the golden ratio.

From binomial theorem, we have for $a\neq 0$ and $n\in\mathbb{N}$,
$$
(a+b)^n-(a-b)^n={\displaystyle\sum^n_{k=0}}
\binom{n}{k}a^{n-k}b^k(1-(-1)^k)=2{\displaystyle
\sum^{\lfloor\frac{n-1}{2}\rfloor}_{l=0}}
\binom{n}{2l+1}a^{n-(2l+1)}b^{2l+1},
$$
\begin{equation}\label{eq9}
(a+b)^n-(a-b)^n=2a^n{\displaystyle
\sum^{\lfloor\frac{n-1}{2}\rfloor}_{l=0}}\binom{n}{2l+1}\left(\frac{b}{a}\right)^{2l+1}.
\end{equation}
We set
$$
\begin{array}{c}
a+b=\varphi=\frac{1+\sqrt{5}}{2},
\\
a-b=1-\varphi=\frac{1-\sqrt{5}}{2}.
\end{array}
$$
So
$$
a=\frac{1}{2},\,\,b=\frac{2\varphi-1}{2}=\frac{\sqrt{5}}{2}.
$$
Thus
$$
\frac{b}{a}=\sqrt{5}.
$$
We get from \eqref{eq9}
$$
\varphi^n-(1-\varphi)^n=\frac{\sqrt{5}}{2^{n-1}}{\displaystyle
\sum^{\lfloor\frac{n-1}{2}\rfloor}_{l=0}}\binom{n}{2l+1}5^l.
$$

Thus we have,

\begin{theorem}\label{t2.1}
$$
F_n=\frac{1}{2^{n-1}}{\displaystyle
\sum^{\lfloor\frac{n-1}{2}\rfloor}_{l=0}}\binom{n}{2l+1}5^l.
$$
\end{theorem}
\begin{property}\label{p2.2}
$$
F_{k+2}=1+{\displaystyle\sum^k_{i=1}}F_i.
$$
\end{property}

\begin{theorem}\label{t2.3}

$$
F_{k+l}=F_lF_{k+1}+F_{l-1}F_k
$$
with $k\in\mathbb{N}$ and $l\geq 2$.
\end{theorem}

The proofs of the above two results can be found in \cite{mol}.

\begin{property}\label{p0.1}
Let $m$ be a positive integer which is greater than $2$. Then, we have
$$
F_{3m+2}=4F_{3m-1}+F_{3m-4}.
$$
\end{property}
The above can be generalized to 

\begin{property}
Let $m$ be a positive integer which is greater than $2$. Then, we have
$$
F_{3m+2}=4{\displaystyle\sum^m_{i=2}}F_{3i-1}.
$$
\end{property}

The above three results can be proved in a straighforward way using the recurrence relation of Fibonacci numbers.

We now state below a few congruence relationships of Fibonacci numbers.

\begin{property}
$F_n\equiv 0\pmod 2$ if and only if $n\equiv 0\pmod 3$.

\end{property}

\begin{corollary}\label{c32}
If $p=5k+2$ is a prime which is strictly greater than $5$
($k\in\mathbb{N}$ and $k$ odd), then $F_p=F_{5k+2}$ is an odd number.
\end{corollary}

In order to prove this assertion, it suffices to remark that $p$ is
not divisible by $3$.

\begin{property}\label{p33}
$$
F_{5k}\equiv 0\pmod 5
$$
with $k\in\mathbb{N}$.
\end{property}

\begin{property}\label{p34}

$$
F_n\geq n
$$
with $n\in \mathbb{N}$ and $n\geq 5$.
\end{property}

The proofs of the above results follows from the principle of mathematical induction and Theorem \ref{t2.3} and Proposition \ref{p2.2}. For brevity, we omit them here.

\section{{Congruences of Fibonacci numbers modulo a prime}}

In this section, we give some new congruence relations involving Fibonacci numbers modulo a prime. The study in this section and some parts of the 
subsequent sections are motivated by some similar results obtained by Bicknell-Johnson in \cite{bj} and by Hoggatt and Bicknell-Johnson in \cite{hbj}.

Let $p=5k+2$ be a prime number with $k$ a non-zero positive integer
which is odd. Notice that in this case, $5k\pm 1$ is an even number and so
$$
\left\lfloor\frac{5k\pm 1}{2}\right\rfloor=\frac{5k\pm 1}{2}.
$$

We now have the following properties.

\begin{property}\label{p4.1}
$$
F_{5k+2}\equiv 5k+1\pmod {5k+2}
$$
with $k\in\mathbb{N}$ and $k$ odd such that $5k+2$ is prime.
\end{property}

This result is also stated in \cite{hbj}, but we give a different proof of the result below.

\begin{proof}
From Theorems \ref{t1.4} and \ref{t2.1} we have
$$
2^{5k+1}F_{5k+2}={\displaystyle\sum^{\frac{5k+1}{2}}_{l=0}}
\binom{5k+2}{2l+1}5^l\equiv 5^{\frac{5k+1}{2}}\equiv 5k+1\pmod {5k+2}
$$
where we used the fact that $\binom{5k+2}{2l+1}$ is divisible by
$5k+2$ for $l=0,1,\ldots,\frac{5k-1}{2}$.

From Fermat's Little Theorem, we have
$$
2^{5k+1}\equiv 1\pmod {5k+2}.
$$

We get $F_{5k+2}\equiv 5k+1\pmod {5k+2}$.
\end{proof}

\begin{property}\label{p4.2}
$$
F_{5k+1}\equiv 1\pmod {5k+2}
$$
with $k\in\mathbb{N}$ and $k$ odd such that $5k+2$ is prime.
\end{property}

\begin{proof}
From Theorem \ref{t2.1} and Property \ref{p1.5} we have
$$
2^{5k}F_{5k+1}={\displaystyle\sum^{\lfloor\frac{5k}{2}\rfloor}_{l=0}}
\binom{5k+1}{2l+1}5^l\equiv (5k+1){\displaystyle\sum^{\lfloor\frac{5k}{2}\rfloor}_{l=0}}
5^l\pmod {5k+2}.
$$
We have
$$
{\displaystyle\sum^{\lfloor\frac{5k}{2}\rfloor}_{l=0}} 5^l
=\frac{5^{\lfloor\frac{5k}{2}\rfloor+1}-1}{4}.
$$
We get from the above
$$
2^{5k+2}F_{5k+1}\equiv (5k+1)\left\{
5^{\lfloor\frac{5k}{2}\rfloor+1}-1
\right\} \pmod {5k+2}.
$$
Since $k$ is an odd positive integer, there exists a positive integer
$m$ such that $k=2m+1$. It follows that
$$
\left\lfloor\frac{5k}{2}\right\rfloor=5m+2.
$$ 

Notice that $5k+2=10m+7$ is prime, implies that $k\neq
5$ and $k\neq 11$ or equivalently $m\neq 2$ and $m\neq 5$. Other
restrictions on $k$ and $m$ can be given.

From Theorem \ref{t1.4} we have
$$
5^{5m+3}\equiv 10m+6\pmod {10m+7}.
$$
We can rewrite 
$\sum^{\lfloor\frac{5k}{2}\rfloor}_{l=0} 5^l=\frac{5^{\lfloor\frac{5k}{2}\rfloor+1}-1}{4}$
as
$$
{\displaystyle\sum^{\lfloor\frac{5k}{2}\rfloor}_{l=0}} 5^l=\frac{5^{5m+3}-1}{4}.
$$
Moreover, we have
$$
(5k+1)\left\{
5^{\lfloor\frac{5k}{2}\rfloor+1}-1
\right\}=(10m+6)\left\{
5^{5m+3}-1
\right\}.
$$
Or,
$$
(10m+6)\left\{
5^{5m+3}-1
\right\}\equiv 5^{5m+3}\left\{
5^{5m+3}-1
\right\}\equiv 5^{10m+6}-10m-6\pmod {10m+7}.
$$
We have
$$
(10m+6)\left\{
5^{5m+3}-1
\right\}\equiv 5^{10m+6}+1\pmod {10m+7}.
$$
From Fermat's Little Theorem, we have $5^{10m+6}\equiv 1\pmod
{10m+7}$. Therefore
$$
(10m+6)\left\{
5^{5m+3}-1
\right\}\equiv 2\pmod {10m+7},
$$
or equivalently
$$
(5k+1)\left\{
5^{\lfloor\frac{5k}{2}\rfloor+1}-1
\right\}\equiv 2\pmod {5k+2}.
$$

It follows that
$$
2^{5k+2}F_{5k+1}\equiv 2\pmod {5k+2}.
$$
Since $2$ and $5k+2$ are relatively prime, so
$$
2^{5k+1}F_{5k+1}\equiv 1\pmod {5k+2}.
$$
From Fermat's Little Theorem, we have $2^{5k+1}\equiv 1\pmod
{5k+2}$. Therefore
$$
F_{5k+1}\equiv 1\pmod {5k+2}.
$$
\end{proof}

\begin{property}\label{p4.3}

$$
F_{5k}\equiv 5k\pmod {5k+2}
$$
with $k\in\mathbb{N}$ and $k$ odd such that $5k+2$ is prime.
\end{property}

\begin{proof}
From Theorem \ref{t2.1} and Property \ref{p1.6} we have
$$
2^{5k-1}F_{5k}={\displaystyle\sum^{\frac{5k-1}{2}}_{l=0}}
\binom{5k}{2l+1}5^l\equiv {\displaystyle\sum^{\frac{5k-1}{2}}_{l=0}}
(5k-2l)5^l\pmod {5k+2}.
$$
Also
$$
{\displaystyle\sum^{\frac{5k-1}{2}}_{l=0}}(5k-2l)5^l
=\frac{5\left[3\times 5^{\frac{5k-1}{2}}-(2k+1)\right]}{8},
$$ where we have used the fact that for $x\neq 1$ and $n\in \mathbb{N}$ we have $$\sum_{l=0}^{n}lx^l=\frac{(n+1)(x-1)x^{n+1}-x^{n+2}+x}{(x-1)^2}.$$
So
$$
2^{5k+2}F_{5k}\equiv
5\left(3\times 5^{\frac{5k-1}{2}}-(2k+1)\right)\pmod {5k+2}.
$$
Moreover since $k=2m+1$, we have
$$
3\times 5^{\frac{5k-1}{2}}-(2k+1)=3\times 5^{5m+2}-(4m+3).
$$
Since $5^{5m+3}\equiv 10m+6\pmod {10m+7}$, we have
$$
3\times 5^{5m+3}\equiv 30m+18 \equiv 40m+25 \pmod {10m+7}.
$$

Consequently
$$
3\times 5^{5m+2}\equiv 8m+5\pmod {10m+7},
$$
which implies
$$
3\times 5^{5m+2}-(4m+3)\equiv 4m+2\pmod {10m+7},
$$
or equivalently for $k=2m+1$
$$
3\times 5^{\frac{5k-1}{2}}-(2k+1)\equiv 2k\pmod {5k+2},
$$
$$
2^{5k+2}F_{5k}\equiv 2\times 5k\pmod {5k+2}.
$$
Since $2$ and $5k+2$ are relatively prime, so
$$
2^{5k+1}F_{5k}\equiv 5k\pmod {5k+2}.
$$
From Fermat's Little Theorem, we have $2^{5k+1}\equiv 1\pmod
{5k+2}$. Therefore
$$
F_{5k}\equiv 5k\pmod {5k+2}.
$$
\end{proof}

\begin{property}\label{p0.3}
Let $5k+2$ be a prime with $k$ odd and let $m$ be a positive
integer which is greater than $2$. Then, we have
$$
F_{3m}\equiv 2\left(3^{m-1}
+{\displaystyle\sum^{m-1}_{i=1}}3^{m-1-i}F_{3i-1}\right)\pmod {5k+2}.
$$
\end{property}

\begin{proof}

We prove the result by induction.

We know that $F_6=8$. Or, $2(3+F_2)=2(3+1)=2\times 4=8$. So
$$
F_6=F_{3\times 2}=2(3+F_2)\equiv 2(3+F_2)\pmod {5k+2}.
$$
Let us assume that 
$F_{3m}\equiv 2\left(3^{m-1}
+{\displaystyle\sum^{m-1}_{i=1}}3^{m-1-i}F_{3i-1}\right)\pmod {5k+2}$
with $m\geq 2$.

For $m$ a positive integer, we have by Theorem \ref{t2.3}

\begin{align}
F_{3(m+1)} &= F_{3m+3}=F_3F_{3m+1}+F_2F_{3m}=2F_{3m+1}+F_{3m} \nonumber \\
 &= 2(F_{3m}+F_{3m-1})+F_{3m}=2F_{3m-1}+3F_{3m}. \nonumber
\end{align}

From the assumption above, we get

\begin{align}
F_{3(m+1)} &\equiv 2F_{3m-1}+2\left(3^m
+{\displaystyle\sum^{m-1}_{i=1}}3^{m-i}F_{3i-1}\right)\pmod {5k+2}\nonumber \\
  &\equiv 2\left(3^m
+{\displaystyle\sum^{m}_{i=1}}3^{m-i}F_{3i-1}\right)\pmod {5k+2}.\nonumber 
\end{align}

Thus the proof is complete by induction.

\end{proof}

\begin{theorem}\label{t0.4}
Let $5k+2$ be a prime with $k$ an odd integer and let $m$ be a
positive integer which is greater than $2$. Then
$$
F_{5mk}\equiv 5k\left(3^{m-1}
+{\displaystyle\sum^{m-1}_{i=1}}3^{m-1-i}F_{3i-1}\right)\pmod {5k+2}
$$
and
$$
F_{5mk+1}\equiv F_{3m-1}\pmod {5k+2}.
$$
\end{theorem}

\begin{proof}

We prove the theorem by induction.

We have, using Theorem \ref{t2.3}
$$
F_{10k}=F_{5k+5k}=F_{5k}F_{5k+1}+F_{5k-1}F_{5k}=F_{5k}(F_{5k+1}+F_{5k-1}).
$$
Using Properties \ref{p4.1}, \ref{p4.2} and \ref{p4.3} we can see that
$$
F_{10k}\equiv 20k\pmod {5k+2}.
$$
Also
$$
5k\left(3
+{\displaystyle\sum^{1}_{i=1}}3^{1-i}F_{3i-1}\right)
=5k(3+F_2)=20k\equiv 20k\pmod {5k+2}.
$$
So
$$
F_{10k}=F_{5\times 2k}\equiv 5k\left(3
+{\displaystyle\sum^{1}_{i=1}}3^{1-i}F_{3i-1}\right)\pmod {5k+2}.
$$
Moreover, we have from Theorem \ref{t2.3}, Property \ref{p4.2} and Property \ref{p4.3},
$$
F_{10k+1}=F_{5k+5k+1}=F^2_{5k+1}+F^2_{5k}\equiv 1+25k^2\pmod {5k+2}.
$$
We have
$$
(5k+2)^2=25k^2+20k+4
\equiv  25k^2+10k\equiv 0\pmod {5k+2}.
$$
So
$$
25k^2\equiv -10k\equiv 2(5k+2)-10k\equiv 4\pmod {5k+2}.
$$
Therefore
$$
F_{10k+1}\equiv F_5\equiv 5\pmod {5k+2},
$$
or equivalently
$$
F_{5\times 2k+1}\equiv F_{3\times 2-1}\equiv 5\pmod {5k+2}.
$$
Let us assume that $$F_{5mk}\equiv 5k\left(3^{m-1}
+{\displaystyle\sum^{m-1}_{i=1}}3^{m-1-i}F_{3i-1}\right)\pmod {5k+2}$$ and
$$F_{5mk+1}\equiv F_{3m-1}\pmod {5k+2}.$$ Then, we have
$$
F_{5(m+1)k}=F_{5mk+5k}=F_{5k}F_{5mk+1}+F_{5k-1}F_{5mk}.
$$
Using Property \ref{p4.3} and $F_{5k-1}\equiv 3\pmod
{5k+2}$, from the assumptions above, we have
$$
F_{5(m+1)k}\equiv 5kF_{3m-1}+3\times 5k\left(3^{m-1}
+{\displaystyle\sum^{m-1}_{i=1}}3^{m-1-i}F_{3i-1}\right)\pmod {5k+2}.
$$
It gives
$$
F_{5(m+1)k}\equiv 5k\left(3^{m}
+{\displaystyle\sum^{m}_{i=1}}3^{m-i}F_{3i-1}\right)\pmod {5k+2}.
$$
Moreover, we have
$$
F_{5(m+1)k+1}=F_{5mk+5k+1}=F_{5k+1}F_{5mk+1}+F_{5k}F_{5mk}.
$$
Using Properties \ref{p4.2} and \ref{p4.3} and the assumptions above, and since $25k^2\equiv 4\pmod {5k+2}$, we have
\begin{align}
F_{5(m+1)k+1} &\equiv F_{3m-1}+25k^2\left(3^{m-1}
+{\displaystyle\sum^{m-1}_{i=1}}3^{m-i-1}F_{3i-1}\right)\pmod {5k+2}\nonumber \\
 &\equiv F_{3m-1}+4\left(3^{m-1}
+{\displaystyle\sum^{m-1}_{i=1}}3^{m-i-1}F_{3i-1}\right)\pmod {5k+2}\nonumber \\
 &\equiv F_{3m-1}+2F_{3m}\pmod {5k+2}.\nonumber
\end{align}

Or,
$$
F_{3m-1}+2F_{3m}=F_{3m-1}+F_{3m}+F_{3m}=F_{3m+1}+F_{3m}=F_{3m+2}.
$$
Therefore
$$
F_{5(m+1)k+1}\equiv F_{3m+2}\pmod {5k+2},
$$
or equivalently
$$
F_{5(m+1)k+1}\equiv F_{3(m+1)-1}\pmod {5k+2}.
$$
This completes the proof.
\end{proof}

\begin{corollary}Let $5k+2$ be a prime with $k$ an odd integer and $m$ be a
positive integer which is greater than $2$. Then
$$
F_{5mk+2}\equiv F_{3m-1}+5k\left(3^{m-1}
+{\displaystyle\sum^{m-1}_{i=1}}3^{m-1-i}F_{3i-1}\right)\pmod {5k+2},
$$
$$
F_{5mk+3}\equiv 2F_{3m-1}+5k\left(3^{m-1}
+{\displaystyle\sum^{m-1}_{i=1}}3^{m-1-i}F_{3i-1}\right)\pmod {5k+2},
$$
and
$$
F_{5mk+4}\equiv 3F_{3m-1}+10k\left(3^{m-1}
+{\displaystyle\sum^{m-1}_{i=1}}3^{m-1-i}F_{3i-1}\right)\pmod {5k+2}.
$$
\end{corollary}

\begin{theorem}\label{t0.6}
Let $5k+2$ be a prime with $k$ an odd positive integer, $m$ be a
positive integer which is greater than $2$ and $r\in\mathbb{N}$. Then
$$
F_{m(5k+r)}\equiv F_{mr}F_{3m-1}+5kF_{mr-1}\left(3^{m-1}
+{\displaystyle\sum^{m-1}_{i=1}}3^{m-1-i}F_{3i-1}\right)\pmod {5k+2}.
$$
\end{theorem}

\begin{proof}
For $m,r$ two non-zero positive integers, we have by Theorem \ref{t2.3}
$$
F_{m(5k+r)}=F_{5mk+mr}=F_{mr}F_{5mk+1}+F_{mr-1}F_{5mk}.
$$
From Theorem \ref{t0.4}, we have for $m\geq 2$ and $r\in\mathbb{N}$
$$
F_{m(5k+r)}\equiv F_{mr}F_{3m-1}+5kF_{mr-1}\left(3^{m-1}
+{\displaystyle\sum^{m-1}_{i=1}}3^{m-1-i}F_{3i-1}\right)\pmod {5k+2}.
$$
\end{proof}

\begin{remark}
In particular, if $r=3$, we know that 
$$F_{m(5k+3)}\equiv 0\pmod {5k+2}.$$ 

This congruence can be deduced from Property \ref{p0.3} and Theorem \ref{t0.6}. Indeed, using Theorem \ref{t0.6}, we have

\begin{align}
F_{m(5k+3)} &\equiv F_{3m}F_{3m-1} +5kF_{3m-1}\left(3^{m-1}
+{\displaystyle\sum^{m-1}_{i=1}}3^{m-1-i}F_{3i-1}\right)\pmod {5k+2}\nonumber \\
 &\equiv F_{3m}F_{3m-1}+5kF_{3m-1}\left(3^{m-1}
+{\displaystyle\sum^{m-1}_{i=1}}3^{m-1-i}F_{3i-1}\right)\nonumber \\
& \quad -(5k+2)F_{3m-1}\left(3^{m-1}
+{\displaystyle\sum^{m-1}_{i=1}}3^{m-1-i}F_{3i-1}\right)\pmod {5k+2}\nonumber \\
 &\equiv F_{3m}F_{3m-1}-2F_{3m-1}\left(3^{m-1}
+{\displaystyle\sum^{m-1}_{i=1}}3^{m-1-i}F_{3i-1}\right)\pmod {5k+2}.\nonumber
\end{align}

Using Property \ref{p0.3} we get
$$
F_{m(5k+3)}\equiv F_{3m}F_{3m-1}-F_{3m-1}F_{3m}\equiv 0\pmod {5k+2}.
$$
\end{remark}

\begin{corollary}\label{c0.8}
Let $5k+2$ be a prime with $k$ an odd positive integer
and $m,r\in\mathbb{N}$. Then
$$
F_{m(5k+r)}\equiv F_{mr}F_{3m-1}-F_{mr-1}F_{3m}\pmod {5k+2}.
$$
\end{corollary}

\begin{lemma}\label{l0.9}
Let $5k+2$ be a prime with $k$ an odd positive integer
and $r\in\mathbb{N}$. Then
$$
F_{5k+r}\equiv F_r-2F_{r-1}\pmod {5k+2}.
$$
\end{lemma}

\begin{proof}
We prove this lemma by induction. For $r=1$, we have $F_{5k+r}=F_{5k+1}\equiv 1\pmod {5k+2}$ and
$F_r-2F_{r-1}=F_1-2F_0=F_1\equiv 1\pmod {5k+2}$.

Let us assume that $F_{5k+s}\equiv F_s-2F_{s-1}\pmod {5k+2}$ for
$s\in[[ 2,r]]$ with $r\geq 2$. Using the assumption, we have for $r\geq 2$
\begin{align}
F_{5k+r+1} &\equiv F_{5k+r}+F_{5k+r-1}\pmod {5k+2}\nonumber \\
&\equiv F_r-2F_{r-1}+F_{r-1}-2F_{r-2}\pmod {5k+2}\nonumber \\
&\equiv F_r+F_{r-1}-2(F_{r-1}+F_{r-2})\pmod {5k+2}\nonumber \\
&\equiv F_{r+1}-2F_r\pmod {5k+2}.\nonumber
\end{align}
Thus the lemma is proved.
\end{proof}

We can prove Lemma \ref{l0.9} as a consequence of Corollary \ref{c0.8} by taking $m=1$.

\begin{corollary}\label{c0.10}
Let $5k+2$ be a prime with $k$ an odd positive integer, let
$m$ be a positive integer and $r\in\mathbb{N}$. Then
$$
F_{m(5k+r)}\equiv  F_{mr}F_{3m+1}-F_{mr+1}F_{3m}\pmod {5k+2}.
$$
\end{corollary}

The above corollary can be deduced from Corollary \ref{c0.8}.

\begin{lemma}\label{l0.11}
Let $5k+2$ be a prime with $k$ an odd positive integer and let
$m$ be a positive integer. Then
$$
F_{5mk}+F_{3m}\equiv 0\pmod {5k+2}.
$$
\end{lemma}

\begin{proof}

For $m=0$, we have $F_{5mk}+F_{3m}=2F_0=0\equiv 0\pmod {5k+2}$.

For $m=1$, we have $F_{5mk}+F_{3m}=F_{5k}+F_3\equiv 5k+2\equiv 0\pmod {5k+2}$.

So, it remains to prove that for $m\geq 2$, we have 
$F_{5mk}+F_{3m}\equiv 0\pmod {5k+2}$.

From Theorem \ref{t0.4}, we have
\begin{align}
F_{5mk} &\equiv 5k\left(3^{m-1}
+{\displaystyle\sum^{m-1}_{i=1}}3^{m-1-i}F_{3i-1}\right)\pmod {5k+2}\nonumber \\
&\equiv 5k\left(3^{m-1}
+{\displaystyle\sum^{m-1}_{i=1}}3^{m-1-i}F_{3i-1}\right)\nonumber\\
&\quad -(5k+2)\left(3^{m-1}
+{\displaystyle\sum^{m-1}_{i=1}}3^{m-1-i}F_{3i-1}\right)\pmod {5k+2}\nonumber\\
&\equiv -2\left(3^{m-1}
+{\displaystyle\sum^{m-1}_{i=1}}3^{m-1-i}F_{3i-1}\right)\pmod {5k+2}.\nonumber
\end{align}
From Property \ref{p0.3}, we have $
F_{5mk}\equiv -F_{3m}\pmod {5k+2}.
$
\end{proof}

We can prove Corollary \ref{c0.10} as a consequence of Lemma \ref{l0.11}.

\begin{remark}
We can observe that $$F_{1\times (5k+r)}=F_{5k+r}$$ and 
$$F_{1\times r}F_{3\times 1+1}-F_{1\times r+1}F_{3\times 1}=F_rF_4-F_{r+1}F_3
=3F_r-2F_{r+1}.$$ By Properties \ref{p4.1}, \ref{p4.2} and \ref{p4.3} we have
$$r=1:~F_{5k+r}=F_{5k+1}\equiv 1\pmod {5k+2}$$
$$3F_r-2F_{r+1}=3F_1-2F_2=1\equiv 1\pmod {5k+2}$$
$$r=2:~F_{5k+r}=F_{5k+2}\equiv 5k+1\pmod {5k+2}$$
$$3F_r-2F_{r+1}=3F_2-2F_3=-1\equiv 5k+1\pmod {5k+2}$$
$$r=3:~F_{5k+r}=F_{5k+3}\equiv 0\pmod {5k+2}$$
$$3F_r-2F_{r+1}=3F_3-2F_4=0\equiv 0\pmod {5k+2}$$
$$r=4:~F_{5k+r}=F_{5k+4}\equiv 5k+1\pmod {5k+2}$$
$$3F_r-2F_{r+1}=3F_4-2F_5=-1\equiv 5k+1\pmod {5k+2}$$

So, we have $$F_{5k+r}\equiv 3F_r-2F_{r+1} \pmod {5k+2}$$ or
equivalently $$F_{1\times (5k+r)}\equiv 
F_{1\times r}F_{3\times 1+1}-F_{1\times r+1}F_{3\times 1}\pmod {5k+2}$$ 
with $r\in[[ 1,4]]$.
\end{remark}

Thus we have

\begin{property}
Let $5k+2$ be a prime with $k$ an odd positive integer, let
$m$ be a positive integer and $r\in\mathbb{N}$. Then
$$
F_{5k+r}\equiv 3F_r-2F_{r+1}\pmod {5k+2}.
$$
\end{property}

\begin{proof}
We have $$F_{5k+0}=F_{5k}\equiv 5k\pmod {5k+2}$$ and
$$3F_0-2F_1=-2\equiv 5k\pmod {5k+2}.$$ So, $F_{5k}\equiv 3F_0-2F_1\pmod {5k+2}$.

Moreover, we know that $$F_{5k+1}\equiv 3F_1-2F_2\equiv 1\pmod {5k+2}.$$

Let us assume that $$F_{5k+s}\equiv 3F_{s}-2F_{s+1}\pmod {5k+2}$$ for
$s\in[[ 1,r]]$. We have for $r\in\mathbb{N}$,

\begin{align}
F_{5k+r+1} &=F_{5k+r}+F_{5k+r-1}\equiv 3F_r-2F_{r+1}+3F_{r-1}-2F_r\pmod {5k+2}\nonumber \\
&\equiv 3(F_r+F_{r-1})-2(F_{r+1}+F_r)\pmod {5k+2}\nonumber \\
&\equiv 3F_{r+1}-2F_{r+2}\pmod {5k+2}.\nonumber 
\end{align}

Thus the proof is complete by induction.
\end{proof}

\begin{property}\label{p0.14}
Let $5k+2$ be a prime with $k$ odd and let $m$ be a positive
integer which is greater than $2$. Then, we have
$$
F_{3m+1}\equiv 3^{m}
+2{\displaystyle\sum^{m-1}_{i=1}}3^{m-1-i}F_{3i}\pmod {5k+2}.
$$
\end{property}

\begin{proof}
We prove the result by induction.

We have for $m=2$, $F_{3m+1}=F_{3\times 2+1}=F_7=13$ and 
$3^{m}+2{\displaystyle\sum^{m-1}_{i=1}}3^{m-1-i}F_{3i}
=3^2+2F_3=9+2\times 2=9+4=13$. So, $F_7\equiv 3^2+2F_3\equiv 13\pmod {5k+2}$.

Let us assume that for $m\geq 2$ the result holds. Using 
this assumption, we have for $m\geq 2$
\begin{align}
F_{3(m+1)+1} &= F_{3m+4}=F_4F_{3m+1}+F_3F_{3m}=3F_{3m+1}+2F_{3m}\nonumber \\
&\equiv 3^{m+1}+2{\displaystyle\sum^{m-1}_{i=1}}3^{m-i}F_{3i}+2F_{3m}\pmod {5k+2}\nonumber \\
&\equiv 3^{m+1}+2{\displaystyle\sum^{m}_{i=1}}3^{m-i}F_{3i}
\pmod {5k+2}.\nonumber
\end{align}
Thus the induction hypothesis holds.
\end{proof}

\begin{corollary}
Let $5k+2$ be a prime with $k$ odd and let $m$ be a positive
integer which is greater than $2$. Then, we have
$$
F_{3m+2}=3^{m}
+2\left(3^{m-1}+{\displaystyle\sum^{m-1}_{i=1}}3^{m-1-i}F_{3i+1}\right)\pmod {5k+2}.
$$
\end{corollary}

\begin{proof}
It stems from the recurrence relation of the Fibonacci sequence which
implies that $F_{3m+2}=F_{3m}+F_{3m+1}$ and $F_{3k+1}=F_{3k}+F_{3k-1}$
and Properties \ref{p0.3} and \ref{p0.14}.
\end{proof}

\section{{Some further congruences of Fibonacci numbers modulo a prime}}

In this section we state and prove some more results of the type that were proved in the previous section. These results generalizes some of the 
results in the previous section and in \cite{bj} and \cite{hbj}.

Let $p=5k+r$ with $r\in[[ 1,4]]$ be a prime number
with $k$ a non-zero positive integer such that $k\equiv r+1\pmod 2$.
Notice that $5k+r\pm 1$ is an even number and so
$$
\left\lfloor\frac{5k+r\pm 1}{2}\right\rfloor=\frac{5k+r\pm 1}{2}.
$$

We have the following properties.

\begin{property}\label{p51}
$$
F_{5k+r}\equiv 5^{\frac{5k+r-1}{2}}\equiv\left\{\begin{array}{ccccc}
1\pmod {5k+r} & if & r=1 & or & r=4,
\\
-1\pmod {5k+r} & if & r=2 & or & r=3,
\end{array}
\right.
$$
with $r\in[[ 1,4]]$
$k\in\mathbb{N}$ and $k\equiv r+1\pmod 2$ such that $5k+r$ is prime.
\end{property}

This result is also stated in \cite{hbj}, here we give a different proof below.

\begin{proof}
From Theorems \ref{t1.4} and \ref{t2.1} we have
$$
2^{5k+r-1}F_{5k+r}={\displaystyle\sum^{\frac{5k+r-1}{2}}_{l=0}}
\binom{5k+r}{2l+1}5^l\equiv 5^{\frac{5k+r-1}{2}}\pmod {5k+r},
$$
where we used the fact that $\binom{5k+r}{2l+1}$ is divisible by
$5k+r$ for $l=0,1,\ldots,\frac{5k+r-3}{2}$.

From Theorem \ref{flt}, we have
$$
2^{5k+r-1}\equiv 1\pmod {5k+r}.
$$

We get $F_{5k+r}\equiv 5^{\frac{5k+r-1}{2}}\pmod {5k+r}$. The rest of
the theorem stems from Theorem \ref{t17}.
\end{proof}

\begin{corollary}\label{c52}
Let $p$ be a prime number which is not equal to $5$. Then, we have
$$
F_p\equiv\left\{\begin{array}{ccccc}
1\pmod p & if & p\equiv 1\pmod 5 & or & p\equiv 4\pmod 5,
\\
p-1\pmod p & if & p\equiv 2\pmod 5 & or & p\equiv 3\pmod 5.
\end{array}
\right.
$$
\end{corollary}

\begin{proof}
We can notice that $F_2=1\equiv 1\pmod 2$ and $2\equiv 2\pmod 5$.
Moreover, we can notice that $F_3=2\equiv 2\pmod 3$ and $3\equiv
3\pmod 5$. So, Corollary \ref{c52} is true for $p=2,3$.

We can observe that the result of Corollary \ref{c52} doesn't work for $p=5$
since $F_5=5\equiv 0\pmod 5$.

The Euclid division of a prime number $p>5$ by $5$ allows to write $p$
like $5k+r$ with $0\leq r<5$ and $k\equiv r+1\pmod 2$. Then, applying
Property \ref{p51}, we verify that 
the result of Corollary \ref{c52} is also true for $p>5$.

It completes the proof of this corollary.
\end{proof}

\begin{property}\label{p52}
$$
F_{5k+r-1}\equiv\left\{\begin{array}{ccccc}
0\pmod {5k+r} & if & r=1 & or & r=4,
\\
1\pmod {5k+r} & if & r=2 & or & r=3
\end{array}
\right.
$$
with $r\in[[ 1,4]]$
$k\in\mathbb{N}$ and $k\equiv r+1\pmod 2$ such that $5k+r$ is prime.
\end{property}

Some parts of this result is stated in \cite{bj} in a different form. We give an alternate proof of the result below.

\begin{proof}
From Theorem \ref{t2.1} and Property \ref{p1.20} we have
$$
2^{5k+r}F_{5k+r-1}=4{\displaystyle\sum^{\lfloor\frac{5k+r-2}{2}\rfloor}_{l=0}}
\binom{5k+r-1}{2l+1}5^l\equiv 4(5k+r-1)
{\displaystyle\sum^{\lfloor\frac{5k+r-2}{2}\rfloor}_{l=0}}5^l\pmod {5k+r}.
$$
It comes that
$$
2^{5k+r}F_{5k+r-1}\equiv (5k+r-1)
\left(5^{\lfloor\frac{5k+r-2}{2}\rfloor+1}-1\right)
\pmod {5k+r}.
$$
From the Theorem \ref{flt}, we have
$$
2^{5k+r}\equiv 2\pmod {5k+r}.
$$
So, since $5k+r-1$ is even and since $2$ and $5k+r$ are relatively
prime when $5k+r$ prime, we obtain
$$
F_{5k+r-1}\equiv\frac{5k+r-1}{2}
\left(5^{\lfloor\frac{5k+r-2}{2}\rfloor+1}-1\right)\pmod {5k+r}.
$$
Since $5k+r-1$ is even and so $\frac{5k+r-1}{2}$ is an integer, 
we can notice that
\begin{align}
\displaystyle\left\lfloor\frac{5k+r-2}{2}\displaystyle\right\rfloor+1 &= \displaystyle\left\lfloor\frac{5k+r-1}{2}-\frac{1}{2}\displaystyle\right\rfloor+1
=\frac{5k+r-1}{2}+\displaystyle\left\lfloor-\frac{1}{2}\displaystyle\right\rfloor+1\nonumber \\
&= \frac{5k+r-1}{2}+\displaystyle\left\lfloor1-\frac{1}{2}\displaystyle\right\rfloor
=\frac{5k+r-1}{2}+\displaystyle\left\lfloor\frac{1}{2}\displaystyle\right\rfloor=\frac{5k+r-1}{2}\nonumber
\end{align}
where we used the property that $\lfloor n+x\rfloor=n+\lfloor
x\rfloor$ for all $n\in\mathbb{N}$ and for all $x\in\mathbb{R}$.

It follows that
\begin{equation}\label{a2}
F_{5k+r-1}\equiv\frac{5k+r-1}{2}
\left(5^{\frac{5k+r-1}{2}}-1\right)\pmod {5k+r}.
\end{equation}
The case $r=2$ was done above. We found (see Property
\ref{p4.2}) and we can verify from the congruence above that
$$
F_{5k+1}\equiv 1\pmod {5k+2}.
$$
From Theorem \ref{t17}, if $r=1$, we have 
$5^{\frac{5k+r-1}{2}}=5^{\frac{5k}{2}}\equiv 1\pmod
{5k+1}$. So, using \eqref{a2}, we
deduce that
$$
F_{5k}\equiv 0\pmod {5k+1}.
$$
From Theorem \ref{t17}, if $r=3$, we have 
$5^{\frac{5k+r-1}{2}}=5^{\frac{5k+2}{2}}\equiv -1\equiv 5k+2\pmod
{5k+3}$. So, using \eqref{a2}, we
deduce that
$$
F_{5k+2}\equiv -(5k+2)\equiv 1\pmod {5k+3}.
$$
From Theorem \ref{t17}, if $r=4$, we have 
$5^{\frac{5k+r-1}{2}}=5^{\frac{5k+3}{2}}\equiv 1\pmod
{5k+4}$. So, using \eqref{a2}, we
deduce that
$$
F_{5k+3}\equiv 0\pmod {5k+4}.
$$
\end{proof}

The following two results are easy consequences of Properties \ref{p51} and \ref{p52}.

\begin{property}\label{p53}
$$
F_{5k+r-2}\equiv\left\{\begin{array}{ccccc}
1\pmod {5k+r} & if & r=1 & or & r=4,
\\
-2\pmod {5k+r} & if & r=2 & or & r=3,
\end{array}
\right.
$$
with $r\in[[ 1,4]]$
$k\in\mathbb{N}$ and $k\equiv r+1\pmod 2$ such that $5k+r$ is prime.
\end{property}

\begin{property}\label{p54}
$$
F_{5k+r+1}\equiv\left\{\begin{array}{ccccc}
1\pmod {5k+r} & if & r=1 & or & r=4,
\\
0\pmod {5k+r} & if & r=2 & or & r=3,
\end{array}
\right.
$$
with $r\in[[ 1,4]]$
$k\in\mathbb{N}$ and $k\equiv r+1\pmod 2$ such that $5k+r$ is prime.
\end{property}

The following is a consequence of Properties \ref{p51} and \ref{p54}.

\begin{property}\label{p55}
$$
F_{5k+r+2}\equiv\left\{\begin{array}{ccccc}
2\pmod {5k+r} & if & r=1 & or & r=4,
\\
-1\pmod {5k+r} & if & r=2 & or & r=3,
\end{array}
\right.
$$
with $r\in[[ 1,4]]$
$k\in\mathbb{N}$ and $k\equiv r+1\pmod 2$ such that $5k+r$ is prime.
\end{property}

Some of the stated properties above are given in \cite{bj} and \cite{hbj} also, but the methods used here are different.

\section{{Periods of the Fibonacci sequence modulo a positive 
integer}}

Notice that $F_1=F_2\equiv 1\pmod m$ with $m$ an integer which is
greater than 2.

\begin{definition}
The Fibonacci sequence $(F_n)$ is periodic modulo a positive integer
$m$ which is greater than $2$ ($m\geq 2$), if there exists at least a non-zero
integer $\ell_m$ such that
$$
F_{1+\ell_m}\equiv F_{2+\ell_m}\equiv 1 \pmod m.
$$ 
The number $\ell_m$ is called a period of the Fibonacci sequence $(F_n)$ modulo
$m$.
\end{definition}

\begin{remark}
For $m\geq 2$ we have $l_m\geq 2$.
Indeed, $\ell_m$ cannot be equal to $1$ since $F_3=2$.
\end{remark}

From Theorem \ref{t2.3} we have
$$
F_{2+\ell_m}=F_{\ell_m}F_3+F_{\ell_m-1}F_2\equiv 2F_{\ell_m}+F_{\ell_m-1}\pmod m.
$$
Since $F_{\ell_m}+F_{\ell_m-1}=F_{1+\ell_m}$, we get
$$
F_{2+\ell_m}\equiv 2F_{\ell_m}+F_{\ell_m-1}\equiv
F_{\ell_m}+F_{1+\ell_m}\equiv F_{\ell_m}+F_{2+\ell_m}\pmod m.
$$
Therefore we have the following
\begin{property}\label{p5.3}

$$
F_{\ell_m}\equiv 0\pmod m.
$$
\end{property}
Moreover, from Theorem \ref{t2.3} we have
$$
F_{1+\ell_m}=F_{\ell_m}F_2+F_{\ell_m-1}F_1\equiv F_{\ell_m}+F_{\ell_m-1}\equiv 
F_{\ell_m-1}\pmod m.
$$
Since $F_{1+\ell_m}\equiv 1\pmod m$, we obtain the following

\begin{property}\label{p5.4}
$$
F_{\ell_m-1}\equiv 1\pmod m.
$$
\end{property}
Besides, using the recurrence relation of the Fibonacci sequence, from
Property \ref{p5.3} we get
$$
F_{\ell_m-2}+F_{\ell_m-1}=F_{\ell_m}\equiv 0\pmod m.
$$
Using Property \ref{p5.4} we obtain

\begin{property}
$$
F_{\ell_m-2}\equiv m-1\pmod m.
$$
\end{property}

\begin{remark}\label{rmtm}
From Theorem \ref{t2.3} we have for $m\geq 2$
$$
F_{2m}=F_{m+m}=F_mF_{m+1}+F_{m-1}F_m=F_m(F_{m+1}+F_{m-1}),
$$
and
\begin{align}
F_{2m+1} &=F_{(m+1)+m}=F_mF_{m+2}+F_{m-1}F_{m+1}\nonumber \\
&= F_m(F_m+F_{m+1})+F_{m-1}(F_{m-1}+F_m)\nonumber \\
&= F_m(2F_m+F_{m-1})+F_{m-1}(F_{m-1}+F_m)\nonumber \\
&= 2F^2_m+2F_mF_{m-1}+F^2_{m-1}=F^2_m+F^2_{m+1}.\nonumber
\end{align}
From this we get
$$
F_{2m+2}=F_{2m+1}+F_{2m}=F^2_m+F^2_{m+1}+F_m(F_{m+1}+F_{m-1}),
$$
$$
F_{2m+3}=F_3F_{2m+1}+F_2F_{2m}=2(F^2_m+F^2_{m+1})+F_m(F_{m+1}+F_{m-1}),
$$
and
$$
F_{2m+4}=F_{2m+3}+F_{2m+2}=3(F^2_m+F^2_{m+1})+2F_m(F_{m+1}+F_{m-1}).
$$

\end{remark}

\begin{theorem}\label{t57}

A period of the Fibonacci sequence modulo $5k+2$ with $5k+2$ a prime
 and $k$ odd is given by
$$
\ell_{5k+2}=2(5k+3).
$$
\end{theorem}

\begin{proof}

Using the recurrence relation of the Fibonacci sequence, and from Properties \ref{p4.1}, \ref{p4.2} and \ref{p4.3} we have 
$$
F_{5k+3}=F_{5k+2}+F_{5k+1}\equiv 5k+2\equiv 0\pmod {5k+2}.
$$
Taking $m=5k+2$ prime ($k$ odd) in the formulas of $F_{2m+3}$ and
$F_{2m+4}$, we have 
\begin{align}
F_{10k+7} &= 2(F^2_{5k+2}+F^2_{5k+3})+F_{5k+2}(F_{5k+3}+F_{5k+1})\nonumber \\
&\equiv 2(5k+1)^2+5k+1\pmod {5k+2}\nonumber \\
&\equiv 50k^2+20k+2+5k+1\equiv 10k(5k+2)+(5k+2)+1\pmod {5k+2}\nonumber \\
&\equiv 1+(10k+1)(5k+2)\equiv 1\pmod {5k+2},\nonumber
\end{align}
and
\begin{align}
F_{10k+8} &= 3(F^2_{5k+2}+F^2_{5k+3})+2F_{5k+2}(F_{5k+3}+F_{5k+1})\nonumber \\
&\equiv 3(5k+1)^2+2(5k+1)\pmod {5k+2}\nonumber \\
&\equiv 75k^2+30k+3+10k+2\pmod {5k+2}\nonumber \\
&\equiv 15k(5k+2)+2(5k+2)+1\pmod {5k+2}\nonumber \\
&\equiv 1+(15k+2)(5k+2)\equiv 1\pmod {5k+2}.\nonumber
\end{align}
Thus
$$
F_{10k+7}\equiv F_{10k+8}\equiv 1\pmod {5k+2},
$$
or equivalently
$$
F_{1+2(5k+3)}\equiv F_{2+2(5k+3)}\equiv 1\pmod {5k+2}.
$$
We deduce that a period of the Fibonacci sequence modulo $5k+2$ with
$5k+2$ a prime is $\ell_{5k+2}=2(5k+3)$.
\end{proof}

We can generalize the above result as follows.

\begin{theorem}\label{t68}
A period of the Fibonacci sequence modulo $5k+r$ with $5k+r$ a prime
such that $r=2,3$ and $k\equiv r+1\pmod 2$ is given by
$$
\ell_{5k+r}=2(5k+r+1).
$$
\end{theorem}

\begin{proof}
Using the formula for $F_{2m}$ given in Remark \ref{rmtm}, taking
$m=5k+r+1$, we have
$$
F_{2(5k+r+1)}=F_{5k+r+1}(F_{5k+r}+F_{5k+r+2}).
$$
From Properties \ref{p51}, \ref{p54} and \ref{p55}, we obtain
$$
F_{2(5k+r+1)}\equiv\left\{\begin{array}{ccccc}
3\pmod {5k+r} & if & r=1 & or & r=4,
\\
0\pmod {5k+r} & if & r=2 & or & r=3.
\end{array}
\right.
$$
Using the formula for $F_{2m+1}$ given in Remark \ref{rmtm}, taking
$m=5k+r+1$, we have
$$
F_{2(5k+r+1)+1}=F^2_{5k+r+1}+F^2_{5k+r+2}.
$$
From Properties \ref{p51} and \ref{p54}, we obtain
$$
F_{2(5k+r+1)+1}\equiv\left\{\begin{array}{ccccc}
5\pmod {5k+r} & if & r=1 & or & r=4,
\\
1\pmod {5k+r} & if & r=2 & or & r=3.
\end{array}
\right.
$$
Using the recurrence relation of the Fibonacci sequence, we have
$F_{2(5k+r+1)+2}=F_{2(5k+r+1)}+F_{2(5k+r+1)+1}$. So
$$
F_{2(5k+r+1)+2}\equiv\left\{\begin{array}{ccccc}
8\pmod {5k+r} & if & r=1 & or & r=4,
\\
1\pmod {5k+r} & if & r=2 & or & r=3.
\end{array}
\right.
$$
Therefore, when $5k+r$ is prime such that $r=2,3$ and $k\equiv
r+1\pmod 2$, we have $F_{2(5k+r+1)}\equiv 0\pmod {5k+r}$ and 
$F_{2(5k+r+1)+1}\equiv F_{2(5k+r+1)+2}\equiv 1\pmod {5k+r}$. It
results that if $5k+r$ is prime such that $r=2,3$ and $k\equiv
r+1\pmod 2$, then $2(5k+r+1)$ is a period of the Fibonacci sequence
modulo $5k+r$. 
\end{proof}

\begin{theorem}\label{t69}
A period of the Fibonacci sequence modulo $5k+r$ with $5k+r$ a prime
such that $r=1,4$ and $k\equiv r+1\pmod 2$ is given by
$$
\ell_{5k+r}=2(5k+r-1).
$$
\end{theorem}

\begin{proof}
Using the formula for $F_{2m}$ given in Remark \ref{rmtm}, taking
$m=5k+r-1$, we have
$$
F_{2(5k+r-1)}=F_{5k+r-1}(F_{5k+r}+F_{5k+r-2}).
$$
From Properties \ref{p51}, \ref{p52} and \ref{p53}, we obtain
$$
F_{2(5k+r-1)}\equiv\left\{\begin{array}{ccccc}
0\pmod {5k+r} & if & r=1 & or & r=4,
\\
-3\pmod {5k+r} & if & r=2 & or & r=3.
\end{array}
\right.
$$
Using the formula for $F_{2m+1}$ given in Remark \ref{rmtm}, taking
$m=5k+r-1$, we have
$$
F_{2(5k+r-1)+1}=F^2_{5k+r-1}+F^2_{5k+r}.
$$
From Properties \ref{p51} and \ref{p52}, we obtain
$$
F_{2(5k+r-1)+1}\equiv\left\{\begin{array}{ccccc}
1\pmod {5k+r} & if & r=1 & or & r=4,
\\
2\pmod {5k+r} & if & r=2 & or & r=3.
\end{array}
\right.
$$
Using the recurrence relation of the Fibonacci sequence, we have
$F_{2(5k+r-1)+2}=F_{2(5k+r-1)}+F_{2(5k+r-1)+1}$. So
$$
F_{2(5k+r-1)+2}\equiv\left\{\begin{array}{ccccc}
1\pmod {5k+r} & if & r=1 & or & r=4,
\\
-1\pmod {5k+r} & if & r=2 & or & r=3.
\end{array}
\right.
$$
Therefore, when $5k+r$ is prime such that $r=1,4$ and $k\equiv
r+1\pmod 2$, we have $F_{2(5k+r-1)}\equiv 0\pmod {5k+r}$ and 
$F_{2(5k+r-1)+1}\equiv F_{2(5k+r-1)+2}\equiv 1\pmod {5k+r}$. It
results that if $5k+r$ is prime such that $r=1,4$ and $k\equiv
r+1\pmod 2$, then $2(5k+r-1)$ is a period of the Fibonacci sequence
modulo $5k+r$. 
\end{proof}

\begin{corollary}\label{c610}
A period of the Fibonacci sequence modulo $5k+r$ with $5k+r$ a prime
such that $r=1,2,3,4$ and $k\equiv r+1\pmod 2$ is given by
$$
\ell_{5k+r}=\left\{\begin{array}{ccccc}
10k & if & r=1, & &
\\
2(5k+3) & if & r=2 & or & r=4,
\\
2(5k+4) & if & r=3, & & 
\end{array}
\right.
$$
or more compactly
$$
\ell_{5k+r}=10k+3(1+(-1)^r)+2(r-1)(1-(-1)^r).
$$
\end{corollary}

Corollary \ref{c610} follows from Theorems \ref{t68} and \ref{t69}.

\begin{corollary}\label{c611}
A period of the Fibonacci sequence modulo $p$ with $p$ a prime
which is not equal to $5$ is given by
$$
\ell_p=\left\{\begin{array}{ccccc}
2(p-1) & if & p\equiv 1\pmod 5 & or & p\equiv 4\pmod 5,
\\
2(p+1) & if & p\equiv 2\pmod 5 & or & p\equiv 3\pmod 5.
\end{array}
\right.
$$
\end{corollary}

\begin{proof}
The Euclid division of $p$ by $5$ is written $p=5k+r$ with $0\leq r<5$
and $k\equiv r+1\pmod 2$. Then, applying Corollary \ref{c610}, it
gives:
$$
\begin{array}{ccccc}
r=1 & & p=5k+1 & & \ell_p=\ell_{5k+1}=10k=2p-2=2(p-1)
\\
r=2 & & p=5k+2 & & \ell_p=\ell_{5k+2}=10k+6=2p+2=2(p+1)
\\
r=3 & & p=5k+3 & & \ell_p=\ell_{5k+3}=10k+8=2p+2=2(p+1)
\\
r=4 & & p=5k+4 & & \ell_p=\ell_{5k+4}=10k+6=2p-2=2(p-1).
\end{array}
$$
\end{proof}

\begin{property}\label{p612}
A period of the Fibonacci sequence modulo $5$ is $20$.
\end{property}

\begin{proof}
From Property \ref{p33}, we know that $F_{5k}\equiv 0\pmod 5$ with
$k\in\mathbb{N}$. Using the recurrence relation of the Fibonacci
sequence, we have $F_{5k+1}\equiv F_{5k+2}\pmod 5$.
So, it is relevant to search a period as an integer multiple of
$5$. Trying the first non-zero values of $k$, it gives:
$$
\begin{array}{ccccc}
k=1 & & F_{5k+1}=F_6\equiv 3\pmod 5 & & F_{5k+2}=F_7\equiv 3\pmod 5
\\
k=2 & & F_{5k+1}=F_{11}\equiv 4\pmod 5 & & F_{5k+2}=F_{12}\equiv 4\pmod 5
\\
k=3 & & F_{5k+1}=F_{16}\equiv 2\pmod 5 & & F_{5k+2}=F_{17}\equiv 2\pmod 5
\\
k=4 & & F_{5k+1}=F_{21}\equiv 1\pmod 5 & & F_{5k+2}=F_{22}\equiv 1\pmod 5.
\end{array}
$$
\end{proof}

\begin{property}\label{p613}
Let $k$ be a positive integer. Then, we have
$$
\begin{array}{ccccc}
F_{5k+1}\equiv F_{5k+2}\equiv 1\pmod 5 & & if & & k\equiv 0\pmod 4,
\\
F_{5k+1}\equiv F_{5k+2}\equiv 3\pmod 5 & & if & & k\equiv 1\pmod 4,
\\
F_{5k+1}\equiv F_{5k+2}\equiv 4\pmod 5 & & if & & k\equiv 2\pmod 4,
\\
F_{5k+1}\equiv F_{5k+2}\equiv 2\pmod 5 & & if & & k\equiv 3\pmod 4.
\end{array}
$$
\end{property}

\begin{proof}
Since $F_{5k}\equiv 0\pmod 5$, using the recurrence relation of the
Fibonacci sequence, we have $F_{5k+2}=F_{5k+1}+F_{5k}\equiv
F_{5k+1}\pmod 5$. 

If $k\equiv 0\pmod 4$ and $k\geq 0$, then there exists a positive
integer $m$ such that $k=4m$. So, if $k\equiv 0\pmod 4$ and $k\geq 0$,
since $20$ is a period of the
Fibonacci sequence modulo $5$ (see Property \ref{p612}), then we have
$F_{5k+1}=F_{20m+1}\equiv F_1\equiv 1\pmod 5$.

If $k\equiv 1\pmod 4$ and $k\geq 0$, then there exists a positive
integer $m$ such that $k=4m+1$. Using Theorem \ref{t2.3}, it comes that
$$F_{5k+1}=F_{20m+6}=F_6F_{20m+1}+F_5F_{20m}.$$ So, if $k\equiv 1\pmod
4$ and $k\geq 0$, since $F_6=8\equiv
3\pmod 5$, $F_5=5\equiv 0\pmod 5$ and since $20$ is a period of the
Fibonacci sequence modulo $5$ (see Property \ref{p612}), then we
have $F_{5k+1}\equiv F_6F_1\equiv 3\pmod 5$. 

If $k\equiv 2\pmod 4$ and $k\geq 0$, then there exists a positive
integer $m$ such that $k=4m+2$. Using Theorem \ref{t2.3}, it comes that
$$F_{5k+1}=F_{20m+11}=F_{11}F_{20m+1}+F_{10}F_{20m}.$$ So, if $k\equiv
2\pmod 4$ and $k\geq 0$, since $F_{11}=89\equiv
4\pmod 5$, $F_{10}=55\equiv 0\pmod 5$ and since $20$ is a period of the
Fibonacci sequence modulo $5$ (see Property \ref{p612}), then we
have $F_{5k+1}\equiv F_{11}F_1\equiv 4\pmod 5$.

If $k\equiv 3\pmod 4$ and $k\geq 0$, then there exists a positive
integer $m$ such that $k=4m+3$. Using Theorem \ref{t2.3}, it comes that
$$F_{5k+1}=F_{20m+16}=F_{16}F_{20m+1}+F_{15}F_{20m}.$$ So, if $k\equiv
3\pmod 4$ and $k\geq 0$, since $F_{16}=987\equiv
2\pmod 5$, $F_{15}=610\equiv 0\pmod 5$ and since $20$ is a period of the
Fibonacci sequence modulo $5$ (see Property \ref{p612}), we have
$$F_{5k+1}\equiv F_{16}F_1\equiv 2\pmod 5.$$

\end{proof}

\begin{property}\label{p614}
Let $k$ be a positive integer. Then, we have
$$
\begin{array}{ccccccc}
F_{5k+3}\equiv 2\pmod 5 & & F_{5k+4}\equiv 3\pmod 5 & & if & & k\equiv 0\pmod 4,
\\
F_{5k+3}\equiv 1\pmod 5 & & F_{5k+4}\equiv 4\pmod 5 & & if & & k\equiv 1\pmod 4,
\\
F_{5k+3}\equiv 3\pmod 5 & & F_{5k+4}\equiv 2\pmod 5 & & if & & k\equiv 2\pmod 4,
\\
F_{5k+3}\equiv 4\pmod 5 & & F_{5k+4}\equiv 1\pmod 5 & & if & & k\equiv 3\pmod 4.
\end{array}
$$
\end{property}

Property \ref{p614} stems from the recurrence relation of the
Fibonacci sequence and Property \ref{p613}.

\begin{corollary}\label{c615}
The minimal period of the Fibonacci sequence modulo $5$ is $20$.
\end{corollary}

Corollary \ref{c615} stems from the Euclid division, 
Properties \ref{p612},
\ref{p613} and \ref{p614}.

\begin{property}\label{p615}
Let $5k+1$ be a prime with $k$ a non-zero even positive integer. Then, we have
($m\in\mathbb{N}$)
$$
F_{5mk}\equiv 0\pmod {5k+1},
$$
and
$$
F_{5mk+1}\equiv 1\pmod {5k+1}.
$$
\end{property}

\begin{proof}
Let prove the property by induction on the integer $m$.

We have $F_0=0\equiv 0\pmod {5k+1}$ and $F_1=1\equiv 1\pmod {5k+1}$.

Moreover, from Properties \ref{p51} and \ref{p52}, we can notice
that $$F_{5k}\equiv 0\pmod {5k+1}$$ and $$F_{5k+1}\equiv 1\pmod {5k+1}.$$

Let assume that for a positive integer $m$, we have $$F_{5mk}\equiv 0\pmod {5k+1}$$
and $$F_{5mk+1}\equiv 1\pmod {5k+1}.$$ Then, using the assumption, Theorem
\ref{t2.3} and Properties \ref{p51} and \ref{p52}, we have
$$
F_{5(m+1)k}=F_{5mk+5k}=F_{5mk}F_{5k+1}+F_{5mk-1}F_{5k}\equiv 0\pmod {5k+1}
$$
and
$$
F_{5(m+1)k+1}=F_{5mk+1+5k}=F_{5mk+1}F_{5k+1}+F_{5mk}F_{5k}\equiv 1\pmod {5k+1}.
$$
This completes the proof by induction on the integer $m$.
\end{proof}

\begin{property}\label{p617}
A period of the Fibonacci sequence modulo $5k+1$ with $5k+1$ prime is $5k$.
\end{property}

This is a direct consequence of Property
\ref{p615}.

\begin{property}\label{p618}
A period of the Fibonacci sequence modulo $5k+4$ with $5k+4$ prime and
$k$ a non-zero odd positive integer is $5k+3$.
\end{property}

\begin{proof}
From Properties \ref{p51}, \ref{p52} and \ref{p54}, we have
$$
F_{5k+3}\equiv 0\pmod {5k+4},
$$
and
$$
F_{5k+4}\equiv F_{5k+5}\equiv 1\pmod {5k+4}.
$$
So $$F_{1+5k+3}\equiv F_{2+5k+3}\equiv 1\pmod {5k+4}.$$ It results that
$5k+3$ is a period of the Fibonacci sequence modulo $5k+4$.
\end{proof}

\begin{corollary}\label{c619}
A period of the Fibonacci sequence modulo $p$ with $p$ a prime
which is not equal to $5$ is given by
$$
\ell_p=\left\{\begin{array}{ccccc}
p-1 & if & p\equiv 1\pmod 5 & or & p\equiv 4\pmod 5,
\\
2(p+1) & if & p\equiv 2\pmod 5 & or & p\equiv 3\pmod 5.
\end{array}
\right.
$$
\end{corollary}

Corollary \ref{c619} stems from Corollary \ref{c611} and
Properties \ref{p617} and \ref{p618}.

\begin{property}\label{p620}
Let $5k+1$ be a prime with $k$ a non-zero even positive integer. Then,
for all $m\in[[ 0,5]]$
\begin{equation}\label{n}
F_{5k-m}\equiv (-1)^{m+1}F_m\pmod {5k+1}.
\end{equation}
\end{property}

\begin{proof}
We prove this result by induction on the integer $m$.

From Properties \ref{p52} and \ref{p53}, we have
$$
F_{5k}\equiv 0\pmod {5k+1},
$$
and
$$
F_{5k-1}\equiv 1\pmod {5k+1}.
$$
So, we verify that \eqref{n} is true when $m=0$ and $m=1$. Notice that \eqref{n} is
verified when $m=5k$ since $F_0=0\equiv F_{5k}\pmod {5k+1}$.

Let us assume for an integer $m\in[ 0,5k-1]$, we have
$F_{5k-i}\equiv (-1)^{i+1}F_i\pmod {5k+1}$ with $i=0,1,\ldots,m$. 
Then, using the recurrence relation of the Fibonacci sequence, 
we have ($0\leq m\leq 5k-1$),
\begin{align}
F_{5k-m-1} &= F_{5k-m+1}-F_{5k-m}\equiv (-1)^mF_{m-1}-(-1)^{m+1}F_m\pmod {5k+1}\nonumber \\
&\equiv (-1)^m(F_{m-1}+F_m)\equiv (-1)^mF_{m+1}\pmod {5k+1}\nonumber\\
&\equiv (-1)^{m+2}F_{m+1}\pmod {5k+1}\nonumber
\end{align}
since $(-1)^2=1$. It achieves the proof of Property \ref{p620} by
induction on the integer $m$. 
\end{proof}

\begin{remark}
Property \ref{p620} implies that we can limit ourself to the
integer interval $[ 1,\frac{5k}{2}]$ (knowing that
the case $m=0$ is a trivial case) in order to search or to rule out a
value for a possible
period of the Fibonacci sequence modulo $5k+1$ with $5k+1$ prime 
(such that $k$ is a non-zero even positive integer) which is less than
$5k$. Notice that $5k$ is not in general the minimal period of the
Fibonacci sequence modulo $5k+1$ with $5k+1$ prime 
(such that $k$ is a non-zero even positive integer). Indeed, for
instance, if $5k+1=101$ (and so for $k=20$), then it can be shown by calculating
the residue of $F_m$ with $m\in[ 1,50]$ modulo
$5k+1=101$, that the minimal period is $\frac{5k}{2}=50$. Notice that
in some cases as for instance $k=56,84$, the number $k$ is the minimal period
of the Fibonacci sequence modulo $5k+1$ with $5k+1$ prime.
\end{remark}

\begin{theorem}
Let $5k+1$ be a prime with $k$ a non-zero even positive integer. If
$k\equiv 0\pmod 4$, then $F_{\frac{5k}{2}}\equiv 0\pmod {5k+1}$.
\end{theorem}

\begin{proof}

If $k$ is a non-zero positive integer such that $k\equiv 0\pmod 4$, then the
integer $\frac{5k}{2}$ is a non-zero even positive integer. Using
Property \ref{p620} and taking $m=\frac{5k}{2}$, we have
$$
F_{\frac{5k}{2}}\equiv -F_{\frac{5k}{2}}\pmod {5k+1},
$$
and
$$
2F_{\frac{5k}{2}}\equiv 0\pmod {5k+1}.
$$
Since $2$ and $5k+1$ with $5k+1$ prime are relatively prime, we get
$$
F_{\frac{5k}{2}}\equiv 0\pmod {5k+1}.
$$
\end{proof}

\begin{remark}
We can observe that
$$
F_{5k-1}=F_{5k+1}-F_{5k}\equiv 1-5k\equiv 3\equiv F_4\pmod {5k+2},
$$
$$
F_{5k-2}=F_{5k}-F_{5k-1}\equiv 5k-3\equiv 5k-F_4\pmod {5k+2},
$$
and
$$
F_{5k-3}=F_{5k-1}-F_{5k-2}\equiv 6-5k\equiv 8\equiv F_6\pmod {5k+2},
$$
$$
F_{5k-4}=F_{5k-2}-F_{5k-3}\equiv 5k-11\equiv 5k-(F_4+F_6)\pmod {5k+2}.
$$

\end{remark}

Using induction we can show the following two properties.

\begin{property}
Let $5k+2$ be a prime with $k$ odd. Then, we have
$$
F_{5k-(2l+1)}\equiv F_{2(l+2)}\pmod {5k+2}
$$
with $l$ a positive integer such that 
$l\leq\lfloor\frac{5k-1}{2}\rfloor$.
\end{property}

\begin{property}

Let $5k+2$ be a prime with $k$ odd. Then, we have
$$
F_{5k-2l}\equiv 5k-{\displaystyle\sum^{l-1}_{i=0}}F_{2(i+2)}\pmod {5k+2}
$$
with $l\geq 1$ such that 
$l\leq \lfloor\frac{5k}{2}\rfloor$.
\end{property}

\begin{remark}
We can notice that
$$
F_{5k+4}=F_{5k+3}+F_{5k+2}\equiv F_{5k+2}\equiv 5k+1\pmod {5k+2},
$$
$$
F_{5k+5}=F_{5k+4}+F_{5k+3}\equiv F_{5k+4}\equiv 5k+1\pmod {5k+2},
$$
and
$$
F_{5k+6}=F_{5k+5}+F_{5k+4}\equiv 10k+2\equiv 5k\pmod {5k+2}.
$$
And for $l\geq 1$ we have
\begin{align}
F_{5k+3l+2} &= F_{3l+2}F_{5k+1}+F_{3l+1}F_{5k}\equiv
F_{3l+2}+5kF_{3l+1}\pmod {5k+2}\nonumber\\
&\equiv F_{3l}+(5k+1)F_{3l+1}\equiv F_{3l}-F_{3l+1}
+(5k+2)F_{3l+1}\pmod {5k+2}\nonumber\\
&\equiv F_{3l}-F_{3l+1}\equiv -F_{3l-1}\pmod {5k+2}.\nonumber
\end{align}
Furthermore, we have for $l\geq 1$
\begin{align}
F_{5k+3l+1} &= F_{3l+1}F_{5k+1}+F_{3l}F_{5k}\equiv
F_{3l+1}+5kF_{3l}\pmod {5k+2}\nonumber\\
&\equiv F_{3l-1}+(5k+1)F_{3l}\equiv F_{3l-1}-F_{3l}
+(5k+2)F_{3l}\pmod {5k+2}\nonumber \\
&\equiv F_{3l-1}-F_{3l}\equiv -F_{3l-2}\pmod {5k+2}.\nonumber
\end{align}

Besides, we have for $l\geq 1$
\begin{align}
F_{5k+3l} &= F_{3l}F_{5k+1}+F_{3l-1}F_{5k}\equiv
F_{3l}+5kF_{3l-1}\pmod {5k+2}\nonumber \\
&\equiv F_{3l-2}+(5k+1)F_{3l-1}\equiv F_{3l-2}-F_{3l-1}
+(5k+2)F_{3l-1}\pmod {5k+2}\nonumber \\
&\equiv F_{3l-2}-F_{3l-1}\equiv -F_{3l-3}\pmod {5k+2}.\nonumber
\end{align}
\end{remark}

We can state the following property, the proof of which follows from
the above remark and by using induction. 
\begin{property}
 $F_{5k+n}\equiv -F_{n-3} \pmod {5k+2}$.
\end{property}

\begin{theorem}
Let $5k+2$ be a prime with $k$ an odd positive number and let $n$ a
positive integer. Then, we have
$$
F_{n(5k+3)}\equiv 0\pmod {5k+2}.
$$
\end{theorem}

\begin{proof}
The proof of the theorem will be done by induction. We have $F_0\equiv
0\pmod {5k+2}$. Moreover, we know that $$F_{5k+3}\equiv 0\pmod {5k+2}.$$
Let us assume that 
\begin{equation}\label{n1} 
F_{n(5k+3)}\equiv 0\pmod {5k+2}.
\end{equation}
We have
$$
F_{(n+1)(5k+3)}=F_{n(5k+3)+5k+3}=F_{5k+3}F_{n(5k+3)+1}+F_{5k+2}F_{n(5k+3)}.
$$
Since $F_{5k+3}\equiv 0\pmod {5k+2}$, using \eqref{n1}, we deduce that
$$
F_{(n+1)(5k+3)}\equiv 0\pmod {5k+2}.
$$
\end{proof}

The following follows very easily from the above theorem.

\begin{corollary}
If $5k+3|m$, then $F_m\equiv 0\pmod {5k+2}$.
\end{corollary}

\begin{property}\label{p629}
Let $5k+2$ be a prime with $k$ an odd positive integer. Then,
for all $m\in[[ 0,5k]]$
\begin{equation}\label{n2}
F_{5k-m}\equiv (-1)^{m+1}F_{m+3}\pmod {5k+2}.
\end{equation}
\end{property}

\begin{proof}
Let us prove Property \ref{p629} by induction on the integer $m$.

From Properties \ref{p51} and \ref{p52}, we have
$$
F_{5k+2}\equiv -1\pmod {5k+2},
$$
and
$$
F_{5k+1}\equiv 1\pmod {5k+2}.
$$
Using the recurrence relation of the Fibonacci sequence, it comes
that
$$
F_{5k}\equiv -2\pmod {5k+2},
$$
and
$$
F_{5k-1}\equiv 3\pmod {5k+2}.
$$
So, we verify \eqref{n2} is true when $m=0$ and $m=1$. 

Notice that \eqref{n2} is
verified when $m=5k$ since $F_0=0\equiv 0\pmod {5k+2}$ and
$F_{5k+3}\equiv 0\pmod {5k+2}$.

Let assume for an integer $m\in[[ 0,5k-1]]$, we have
$F_{5k-i}\equiv (-1)^{i+1}F_{i+3}\pmod {5k+2}$ with $i=0,1,\ldots,m$. 
Then, using the recurrence relation of the Fibonacci sequence, 
we have ($0\leq m\leq 5k-1$)
\begin{align}
F_{5k-m-1} &= F_{5k-m+1}-F_{5k-m}\equiv (-1)^mF_{m+2}-(-1)^{m+1}F_{m+3}\pmod {5k+2}\nonumber\\
&\equiv (-1)^m(F_{m+2}+F_{m+3})\equiv (-1)^{m+2}F_{m+4}\pmod {5k+2}\nonumber\\
&\equiv (-1)^{m+2}F_{m+4}\pmod {5k+2}\nonumber
\end{align}
since $(-1)^2=1$. It achieves the proof of Property \ref{p629} by
induction on the integer $m$. 
\end{proof}

Notice that Property \ref{p629} is also true for $m=-2,-1$.

\begin{remark}
In general, the number $2(5k+3)$ is not the minimal period of the
Fibonacci sequence modulo $5k+2$ with $5k+2$ prime such that 
$k$ an odd positive
integer. Indeed, if $k\equiv 0\pmod 3$, then in some cases as for
instance k=9,21,69,111,135,195,219, it can be verified that the
numbers $\frac{2(5k+3)}{3}$ and $\frac{4(5k+3)}{3}$ are periods of the
Fibonacci sequence modulo $5k+2$ with $5k+2$ prime. 
\end{remark}

\begin{theorem}
Let $5k+2$ be a prime number with $k$ an odd positive number. If
$k\equiv 3\pmod 4$, then $F_{\frac{5k+3}{2}}\equiv 0\pmod {5k+2}$.
\end{theorem}

\begin{proof}

Since $5k+2$ with $k$ an odd positive number, is prime, the numbers 
$5k\pm 3$ are non-zero even positive integers. So, the numbers
$\frac{5k\pm 3}{2}$ are
non-zero positive integers. Moreover, if $k\equiv 3\pmod 4$, then
$5k-3\equiv 12\equiv 0\pmod 4$. So, the integer $\frac{5k-3}{2}$ is
even.

Using Property \ref{p629} and taking $m=\frac{5k-3}{2}$, we have
$$
F_{\frac{5k+3}{2}}\equiv -F_{\frac{5k+3}{2}}\pmod {5k+2},
$$
or,
$$
2F_{\frac{5k+3}{2}}\equiv 0\pmod {5k+2}.
$$
Finally,
$$
F_{\frac{5k+3}{2}}\equiv 0\pmod {5k+2},
$$
since $2$ and $5k+2$ with $5k+2$ prime are relatively prime.
\end{proof}

\begin{theorem}
Let $5k+2$ be a prime number with $k$ an odd positive integer. If
$k\equiv 0\pmod 3$ and if the number $\frac{2(5k+3)}{3}$ is a period of
the Fibonacci sequence modulo $5k+2$, then the congruence
$$F_{5k+3}\equiv 0\pmod {5k+2}$$ is 
equivalent to the congruence $$F_{\frac{5k+3}{3}}\equiv 0\pmod {5k+2}$$
which is equivalent to the congruence $$F_{\frac{2(5k+3)}{3}}\equiv
0\pmod {5k+2}.$$ Moreover, if $k\equiv 0\pmod 3$ and if
$$F_{\frac{5k+3}{3}}\equiv 0\pmod {5k+2}$$, 
the number $\frac{2(5k+3)}{3}$ is a period of the Fibonacci
sequence modulo $5k+2$ if and only if $$F_{\frac{5k}{3}}\equiv -1\pmod {5k+2}.$$
\end{theorem}

\begin{proof}

If $k\equiv 0\pmod 3$ and $k$ an
odd positive integer, then there exists a non-zero positive integer
$m$ such that $k=3m$. Notice that $m$ is odd since $k$ is odd.
Since $F_{5k+3}\equiv 0\pmod {5k+2}$ with $5k+2$ prime ($k$
positive odd), we have 
also $F_{15m+3}\equiv 0\pmod {15m+2}$ with $15m+2$ prime ($m$ positive
odd). Using Theorem \ref{t2.3}, we have
$$ 
F_{15m+3}=F_{10m+2+5m+1}=F_{5m+1}F_{10m+3}+F_{5m}F_{10m+2}.
$$
Or, from Remark \ref{rmtm}, we have
$$
F_{10m+2}=F_{2(5m+1)}=F_{5m+1}(F_{5m}+F_{5m+2})=F^2_{5m+2}-F^2_{5m},
$$
and
$$
F_{10m+3}=F_{2(5m+1)+1}=F^2_{5m+1}+F^2_{5m+2}.
$$
We have also
$$
F_{10m+1}=F_{5m+5m+1}=F^2_{5m+1}+F^2_{5m}.
$$
So
\begin{align}
F_{15m+3} &= F_{5m+1}(F^2_{5m+1}+F^2_{5m+2})+F_{5m}F_{5m+1}(F_{5m}+F_{5m+2})\nonumber \\
&= F_{5m+1}(F^2_{5m+1}+F^2_{5m+2}+F^2_{5m}+F_{5m}F_{5m+2})\nonumber \\
&= F_{5m+1}(3F^2_{5m}+3F_{5m}F_{5m+1}+2F^2_{5m+1})\nonumber \\
&= F_{5m+1}(3F_{5m}F_{5m+2}+2F^2_{5m+1}).\nonumber
\end{align}

So, the congruence 
$F_{15m+3}\equiv 0\pmod {15m+2}$ with $m$ an odd positive integer such
that $15m+2$ prime is satisfied if and only if either
$$
F_{5m+1}\equiv 0\pmod {15m+2},
$$ 
or
$$
3F_{5m}F_{5m+2}\equiv -2F^2_{5m+1}\pmod {15m+2}.
$$
If $F_{5m+1}\equiv 0\pmod {15m+2}$, then from above, we have necessarily
$$F_{10m+2}\equiv 0\pmod {15m+2}.$$ Using the 
recurrence relation of the Fibonacci sequence, it implies also that
$F_{5m}\equiv F_{5m+2}\pmod {15m+2}$. Moreover, we have
$$F_{10m+3}\equiv F^2_{5m+2}\equiv F^2_{5m}\pmod {15m+2}.$$ Or, we have
$$
F_{15m+2}=F_{10m+2+5m}=F_{5m}F_{10m+3}+F_{5m-1}F_{10m+2}.
$$
Since $F_{5k+2}\equiv 5k+1\equiv -1\pmod {5k+2}$ with $5k+2$ prime
($k$ positive odd) and so if $k=3m$ such that $m$ positive odd,
$$F_{15m+2}\equiv -1\pmod {15m+2}$$ with $15m+2$ prime ($m$ positive
odd), since $$F_{10m+3}\equiv F^2_{5m}\pmod
{15m+2}$$ and $$F_{10m+2}\equiv 0\pmod {15m+2},$$ 
it implies that $$F_{5m}F_{10m+3}\equiv F^3_{5m}\equiv -1\pmod
{15m+2}.$$ We get
\begin{equation}\label{b6}
F^3_{5m}+1\equiv 0\pmod {15m+2},
\end{equation}
and
$$
(F_{5m}+1)(F^2_{5m}-F_{5m}+1)\equiv 0\pmod {15m+2}.
$$
So, either
$$
F_{5m}+1\equiv 0\pmod {15m+2}
$$
or
$$
F^2_{5m}-F_{5m}+1\equiv 0\pmod {15m+2}.
$$

If $$F_{5m+1}\equiv 0\pmod {15m+2}$$ and 
if $$F_{5m}+1\equiv 0\pmod {15m+2}$$ and so $$F_{5m}\equiv -1\pmod
{15m+2}$$, then $$F_{10m+3}\equiv 1\pmod {15m+2},$$ It results that the
number $10m+2$ is a period of the Fibonacci sequence modulo $15m+2$
with $15m+2$ prime and $m$ an odd positive integer.

If $$F_{5m+1}\equiv 0\pmod {15m+2}$$ and 
if $$F^2_{5m}-F_{5m}+1\equiv 0\pmod {15m+2}$$ and so $$F^2_{5m}\equiv
F_{5m}-1\pmod {15m+2},$$ then since $$F_{10m+3}\equiv F^2_{5m}\pmod
{15m+2},$$ $$F_{10m+3}\equiv F_{5m}-1\pmod {15m+2}.$$ Notice that in this
case, we cannot have $$F_{5m} \equiv -1\pmod {15m+2}$$ since
$3\not\equiv 0\pmod {15m+2}$ with $m$ an odd positive integer such
that $15m+2$ prime (and so $15m+2>3$). Then, let assume absurdly that
if $$F^2_{5m}-F_{5m}+1\equiv 0\pmod {15m+2},$$ then
the number $10m+2$ is a period of the Fibonacci sequence modulo $15m+2$
with $15m+2$ prime and $m$ an odd positive integer. In such a case,
$$F_{10m+3}\equiv 1\pmod {15m+2}$$ which implies that $$F_{5m}\equiv
2\pmod {15m+2}.$$ Since $$F^2_{5m}\equiv F_{5m}-1\pmod {15m+2},$$ it
gives $$4\equiv 1\pmod {15m+2}.$$ But, since $15m+2$ is a prime number
such that $m$ is an odd positive integer, we have $15m+2>4$ and so
$4\not\equiv 1\pmod {15m+2}$. So, we reach to a contradiction meaning
that if $$F^2_{5m}-F_{5m}+1\equiv 0\pmod {15m+2}$$ and so if
$$F_{5m}\not\equiv -1\pmod {15m+2},$$ the number $10m+2$ is not a period
of the Fibonacci sequence modulo $15m+2$ 
with $15m+2$ prime and $m$ an odd positive integer.

Moreover, if $$F_{5m+1}\equiv 0\pmod {15m+2}$$ and
reciprocally if the number $10m+2$ is a period
of the Fibonacci sequence modulo $15m+2$ 
with $15m+2$ prime and $m$ an odd positive integer, then
$$F_{10m+3}\equiv 1\pmod {15m+2}$$ which implies that $$F^2_{5m}\equiv
1\pmod {15m+2}.$$ So, either $$F_{5m}\equiv 1\pmod {15m+2}$$ or
$$F_{5m}\equiv -1\pmod {15m+2}.$$ Since we have \eqref{b6}, it remains only one possibility, that is to say $$F_{5m}\equiv
-1\pmod {15m+2}.$$

$\frac{2(5k+3)}{3}=10m+2$ is a period of the Fibonacci sequence, we
must have $F_{10m+3}\equiv F^2_{5m}\equiv 1\pmod {15m+2}$ in addition
to the condition $$F_{5m+1}\equiv 0\pmod {15m+2}.$$

If $$3F_{5m}F_{5m+2}\equiv -2F^2_{5m+1}\pmod {15m+2},$$ then from Property \ref{peqc}, we can find an integer $c$ such that
$$
\left\{\begin{array}{c}
F_{5m}F_{5m+2}\equiv -2c\pmod {15m+2},
\\
F^2_{5m+1}\equiv 3c\pmod {15m+2}.
\end{array}
\right.
$$
So
$$
c\equiv F^2_{5m+1}+F_{5m}F_{5m+2}\pmod {15m+2},
$$
or equivalently ($F_{5m+2}=F_{5m+1}+F_{5m}$ and $F_{10m+1}=F^2_{5m+1}+F^2_{5m}$)
\begin{align}
c &\equiv F^2_{5m+1}+F_{5m+1}F_{5m}+F^2_{5m}\pmod {15m+2}\nonumber \\
&\equiv F_{10m+1}+F_{5m+1}F_{5m}\pmod {15m+2}.\nonumber
\end{align}
So, if the number $10m+2$ with $m$ an odd positive integer 
is a period of the Fibonacci sequence modulo $15m+2$
with $15m+2$ prime, we should have $$F_{10m+2}\equiv 0\pmod {15m+2}$$
and $$F_{10m+1}\equiv F_{10m+3}\equiv 1\pmod {15m+2}.$$ Since
$$F_{10m+2}=F^2_{5m+2}-F^2_{5m}$$ and $$c\equiv
F_{10m+1}+F_{5m+1}F_{5m}\pmod {15m+2},$$ it implies that
$$F^2_{5m}\equiv F^2_{5m+2}\pmod {15m+2}$$ and $$c\equiv
1+F_{5m}F_{5m+1}\pmod {15m+2}.$$ So, either $$F_{5m}\equiv F_{5m+2}\pmod
{15m+2}$$ or $$F_{5m}\equiv -F_{5m+2}\pmod {15m+2}.$$ If $$F_{5m}\equiv
F_{5m+2}\pmod {15m+2},$$ then $$F_{5m+1}\equiv 0\pmod {15m+2}$$ and
$$c\equiv 1\equiv 0\pmod {15m+2}$$ where we used the fact that $$3c\equiv
F^2_{5m+1}\pmod {15m+2}$$ and $(3,15m+2)=1$ with $15m+2$ prime. But,
$1\not\equiv 0\pmod 
{15m+2}$. So, we reach a contradiction meaning that 
this case is not possible. Otherwise, if $$F_{5m}\equiv
-F_{5m+2}\pmod {15m+2},$$ then using the recurrence relation of the
Fibonacci sequence, we must have $$F_{5m+1}\equiv -2F_{5m}\pmod
{15m+2}$$ and so $$c\equiv 1-2F^2_{5m}\equiv 3F^2_{5m}\pmod {15m+2}$$
where we used the fact that $$c\equiv F^2_{5m+1}+F_{5m}F_{5m+2}\pmod {15m+2}.$$ It
implies that $$5F^2_{5m}\equiv 1\pmod {15m+2}$$ and using Theorem \ref{flt}, it gives $$F^2_{5m}\equiv 5^{15m}\equiv
6m+1\pmod {15m+2}$$ since $5^{15m+1}\equiv 1\equiv 30m+5\pmod {15m+2}$
which implies that $5^{15m}\equiv 6m+1\pmod {15m+2}$ (reccall that
$15m+2$ is prime and so $(5,15m+2)=1$). Since $$F_{5m+1}\equiv -2F_{5m}\pmod
{15m+2},$$ $$F^2_{5m+1}\equiv 3c\pmod {15m+2}$$ and $$c\equiv
3F^2_{5m}\pmod {15m+2},$$ it results that 
$$F^2_{5m+1}\equiv 4F^2_{5m}\equiv 3c\equiv 9F^2_{5m}\pmod {15m+2}$$ and
so $4(6m+1)\equiv 9(6m+1)\pmod {15m+2}$. Since $4(6m+1)=24m+4\equiv
9m+2\pmod {15m+2}$, it
implies that $45m+7\equiv 0\pmod {15m+2}$ and so $1\equiv 0\pmod
{15m+2}$ which is
not possible since $1\not\equiv 0\pmod {15m+2}$. So, we obtain again a
contradiction meaning that this latter case is not also possible.

Therefore, when $5k+2=15m+2$ is prime with $k=3m$ and $m$ an odd positive
integer, if $10m+2$ is a period of the Fibonacci sequence modulo
$15m+2$ with $15m+2$ prime, then $$F_{15m+3}\equiv 0\pmod {15m+2}$$ if and only if
$$F_{5m+1}\equiv 0\pmod {15m+2}.$$ Since $$F_{15m+3}=F_{5k+3}\equiv
0\pmod {5k+2}$$ is true when $5k+2$ is prime, we deduce that 
$$F_{\frac{5k+3}{3}}\equiv 0\pmod {5k+2}$$ is also true when $k\equiv
0\pmod 3$ and $5k+2$ prime. 

Thus, if $10m+2$ is a period of the Fibonacci sequence modulo
$15m+2$ with $15m+2$ prime, then we have
$$
F_{15m+3}\equiv 0\pmod {15m+2}$$
if and only if
$$
F_{5m+1}\equiv 0\pmod {15m+2}$$
if and only if
$$
F_{5m}\equiv F_{5m+2}\pmod {15m+2}.
$$
Besides, $$F_{5m+1}\equiv 0\pmod {15m+2}$$ implies that $$F_{10m+2}\equiv
0\pmod {15m+2}.$$ 

Reciprocally, if $$F_{10m+2}\equiv 0\pmod {15m+2},$$
then $$F^2_{5m}\equiv F^2_{5m+2}\pmod {15m+2}.$$ So, either
$$F_{5m}\equiv F_{5m+2}\pmod {15m+2}$$ or $$F_{5m}\equiv -F_{5m+2}\pmod
{15m+2}.$$ If $$F_{5m}\equiv -F_{5m+2}\pmod {15m+2},$$ then
$$F_{5m+1}\equiv -2F_{5m}\pmod {15m+2}$$ and since $$F_{5m+1}\equiv
0\pmod {15m+2},$$ using the fact that $(2,15m+2)=1$ with $15m+2$ prime
such that $m$ an odd positive integer ($15m+2>2$), 
$$F_{5m}\equiv 0\pmod {15m+2}.$$ But, then, if
$$F_{10m+2}\equiv 0\pmod {15m+2},$$ we have
$$F_{15m+2}\equiv F_{10m+3}F_{5m}\equiv 0\pmod {15m+2}.$$ Or,
$$F_{15m+2}\equiv -1\pmod {15m+2}.$$ It leads to a contradiction meaning
that $$F_{5m}\equiv -F_{5m+2}\pmod {15m+2}$$ is not possible. So, if
$$F_{10m+2}\equiv 0\pmod {15m+2},$$ there
is only one possibility, that is to say $$F_{5m}\equiv F_{5m+2}\pmod
{15m+2}$$ which implies the congruence $$F_{5m+1}\equiv 0\pmod {15m+2}$$
and so which translates the congruence $$F_{15m+2}\equiv -1\pmod {15m+2}$$
into the congruence $$F^3_{5m}\equiv -1\pmod {15m+2}$$ which has at
least one solution. So, if $10m+2$ is a period of the Fibonacci sequence modulo
$15m+2$ with $15m+2$ prime, then we have
$$
F_{15m+3}\equiv 0\pmod {15m+2}$$
if and only if
$$
F_{5m+1}\equiv 0\pmod {15m+2}
$$
if and only if
$$
F_{5m}\equiv F_{5m+2}\pmod {15m+2}.
$$
if and only if
$$
F_{10m+2}\equiv 0\pmod {15m+2}.
$$

Since $10m+2=2(5m+1)=\frac{2(5k+3)}{3}$ with $k=3m$ and $m$ an odd positive
integer, from above, we conclude that the number
$\frac{2(5k+3)}{3}$ is a period of the Fibonacci 
sequence modulo $5k+2$ if and only if $$F_{\frac{5k}{3}}\equiv -1\pmod {5k+2}.$$

\end{proof}

\begin{property}\label{p632}
Let $5k+3$ be a prime with $k$ an even positive integer. Then,
for all $m\in[[ 0,5k]]$
\begin{equation}\label{b7}
F_{5k-m}\equiv (-1)^mF_{m+4}\pmod {5k+3}.
\end{equation}
\end{property}

\begin{proof}
Let us prove Property \ref{p632} by induction on the integer $m$.

From Properties \ref{p51} and \ref{p52}, we have
$$
F_{5k+3}\equiv -1\pmod {5k+3},
$$
and
$$
F_{5k+2}\equiv 1\pmod {5k+3}.
$$
Using the recurrence relation of the Fibonacci sequence, it comes
that
$$
F_{5k+1}\equiv -2\pmod {5k+3},
$$
$$
F_{5k}\equiv 3\pmod {5k+3},
$$
and
$$
F_{5k-1}\equiv -5\pmod {5k+3}.
$$
So, we verify \eqref{b7} is true when $m=0$ and $m=1$. 

Notice that \eqref{b7} is
verified when $m=5k$ since $F_0=0\equiv 0\pmod {5k+3}$ and
$F_{5k+4}\equiv 0\pmod {5k+3}$.

Let us assume for an integer $m\in[[ 0,5k-1]]$, we have
$F_{5k-i}\equiv (-1)^iF_{i+4}\pmod {5k+3}$ with $i=0,1,\ldots,m$. 
Then, using the recurrence relation of the Fibonacci sequence, 
we have ($0\leq m\leq 5k-1$)
\begin{align}
F_{5k-m-1} &= F_{5k-m+1}-F_{5k-m}\equiv (-1)^{m-1}F_{m+3}-(-1)^mF_{m+4}\pmod {5k+3}\nonumber \\
&\equiv (-1)^{m-1}(F_{m+3}+F_{m+4})\equiv (-1)^2(-1)^{m-1}F_{m+5}\pmod {5k+3}\nonumber \\
&\equiv (-1)^{m+1}F_{m+5}\pmod {5k+3}\nonumber
\end{align}
since $(-1)^2=1$. It achieves the proof of Property \ref{p632} by
induction on the integer $m$. 
\end{proof}

Notice that Property \ref{p632} is also true for $m=-3,-2,-1$.

\begin{remark}
It can be noticed that for $k=0$, $5k+3=3$ is prime and it
can be verified that $2(5k+4)=8$ for $k=0$ is the minimal period of
the Fibonacci sequence modulo $3$. Nevertheless,
in general, the number $2(5k+4)$ is not the minimal period of the
Fibonacci sequence modulo $5k+3$ with $5k+3$ prime such that 
$k$ an even positive integer.
Indeed, if $k\equiv 1\pmod 3$ and $k$ an even positive integer, 
then in some cases as for
instance k=22,52,70,112,148,244, it can be verified that the
numbers $\frac{2(5k+4)}{3}$ and $\frac{4(5k+4)}{3}$ are periods of the
Fibonacci sequence modulo $5k+3$ with $5k+3$ prime. 
\end{remark}

\begin{theorem}
Let $5k+3$ be a prime number with $k$ a non-zero even positive number. If
$k\equiv 2\pmod 4$, then $$F_{\frac{5k+4}{2}}\equiv 0\pmod {5k+3}.$$
\end{theorem}

\begin{proof}

Since $5k+3$ with $k$ a non-zero even positive number, is prime, the numbers 
$5k\pm 4$ are non-zero even positive integers. So, the numbers
$\frac{5k\pm 4}{2}$ are
non-zero positive integers. Moreover, if $k\equiv 2\pmod 4$, then
$5k-4\equiv 2\pmod 4$. So, the integer $\frac{5k-4}{2}$ is odd.

Using Property \ref{p632} and taking $m=\frac{5k-4}{2}$, it gives
$$
F_{\frac{5k+4}{2}}\equiv -F_{\frac{5k+4}{2}}\pmod {5k+3},
$$
or,
$$
2F_{\frac{5k+4}{2}}\equiv 0\pmod {5k+3},
$$
and finally,
$$
F_{\frac{5k+4}{2}}\equiv 0\pmod {5k+3},
$$
since $2$ and $5k+3$ with $5k+3$ prime are relatively prime.
\end{proof}

\begin{theorem}
Let $5k+3$ be a prime number with $k$ an even positive integer. If
$k\equiv 1\pmod 3$ and if $\frac{2(5k+4)}{3}$ is a period of the Fibonacci
sequence modulo $5k+3$, the congruence $$F_{5k+4}\equiv 0\pmod {5k+3}$$ is
equivalent to the congruence $$F_{\frac{5k+4}{3}}\equiv 0\pmod {5k+3}$$
which is equivalent to the congruence $$F_{\frac{2(5k+4)}{3}}\equiv 0\pmod {5k+3}$$
Moreover, if $k\equiv 1\pmod 3$ and if 
$$F_{\frac{5k+4}{3}}\equiv 0\pmod {5k+3},$$ then 
the number $\frac{2(5k+4)}{3}$ is a period of the Fibonacci
sequence modulo $5k+3$ if and only if $$F_{\frac{5k+1}{3}}\equiv -1\pmod {5k+3}.$$
\end{theorem}

\begin{proof}

If $k\equiv 1\pmod 3$ and $k$ an
even positive integer, then there exists a non-zero positive integer
$m$ such that $k=3m+1$. Notice that $m$ is odd since $k$ is even.
Since $$F_{5k+4}\equiv 0\pmod {5k+3}$$ with $5k+3$ prime ($k$
positive even), we have 
also $$F_{15m+9}\equiv 0\pmod {15m+8}$$ with $15m+8$ prime ($m$ positive
odd). Using Theorem \ref{t2.3}, we have
\begin{align}
F_{15m+9} &= F_{3(5m+3)}=F_{2(5m+3)+5m+3}=F_{5m+3}F_{2(5m+3)+1}+F_{5m+2}F_{2(5m+3)}\nonumber\\
&= F_{5m+3}F_{10m+7}+F_{5m+2}F_{10m+6}.\nonumber
\end{align}
Or, from Remark \ref{rmtm}, we have
$$
F_{10m+6}=F_{2(5m+3)}=F_{5m+3}(F_{5m+4}+F_{5m+2})=F^2_{5m+4}-F^2_{5m+2},
$$
and
$$
F_{10m+7}=F_{2(5m+3)+1}=F^2_{5m+3}+F^2_{5m+4}.
$$
We have also
$$
F_{10m+5}=F_{5m+2+5m+3}=F^2_{5m+3}+F^2_{5m+2}.
$$
So
\begin{align}
F_{15m+9} &= F_{5m+3}(F^2_{5m+3}+F^2_{5m+4})+F_{5m+2}F_{5m+3}(F_{5m+4}+F_{5m+2})\nonumber\\
&= F_{5m+3}(F^2_{5m+3}+F^2_{5m+4}+F^2_{5m+2}+F_{5m+2}F_{5m+4})\nonumber \\
&= F_{5m+3}(3F^2_{5m+2}+3F_{5m+2}F_{5m+3}+2F^2_{5m+3})\nonumber \\
&= F_{5m+3}(3F_{5m+2}F_{5m+4}+2F^2_{5m+3}).\nonumber
\end{align}

So, the congruence 
$$F_{15m+9}\equiv 0\pmod {15m+8}$$ with $m$ an odd positive integer such
that $15m+8$ prime is satisfied if and only if either
$$
F_{5m+3}\equiv 0\pmod {15m+8},
$$ 
or
$$
3F_{5m+2}F_{5m+4}\equiv -2F^2_{5m+3}\pmod {15m+8}.
$$
If $$F_{5m+3}\equiv 0\pmod {15m+8},$$ then from above, we have necessarily
$$F_{10m+6}\equiv 0\pmod {15m+8}.$$ Using the 
recurrence relation of the Fibonacci sequence, it implies also that
$$F_{5m+2}\equiv F_{5m+4}\pmod {15m+8}.$$ Moreover, we have
$$F_{10m+7}\equiv F^2_{5m+4}\equiv F^2_{5m+2}\pmod {15m+8}.$$ Or we have,
$$
F_{15m+8}=F_{10m+6+5m+2}=F_{5m+2}F_{10m+7}+F_{5m+1}F_{10m+6}.
$$ 

Since $$F_{5k+3}\equiv 5k+2\equiv -1\pmod {5k+3}$$ with $5k+3$ prime
($k$ positive even) and so if $k=3m+1$ such that $m$ positive odd,
$$F_{15m+8}\equiv -1\pmod {15m+8}$$ with $15m+8$ prime ($m$ positive
odd), since $$F_{10m+7}\equiv F^2_{5m+2}\pmod
{15m+8}$$ and $$F_{10m+6}\equiv 0\pmod {15m+8},$$ 
it implies that $$F_{5m+2}F_{10m+7}\equiv F^3_{5m+2}\equiv -1\pmod
{15m+8}.$$ It comes that
$$
F^3_{5m+2}+1\equiv 0\pmod {15m+8},
$$
or,
$$
(F_{5m+2}+1)(F^2_{5m+2}-F_{5m+2}+1)\equiv 0\pmod {15m+8}.
$$
So, either
$$
F_{5m+2}+1\equiv 0\pmod {15m+8},
$$
or
$$
F^2_{5m+2}-F_{5m+2}+1\equiv 0\pmod {15m+8}.
$$

If $$F_{5m+3}\equiv 0\pmod {15m+8}$$ and 
if $$F_{5m+2}+1\equiv 0\pmod {15m+8}$$ and so $$F_{5m+2}\equiv -1\pmod
{15m+8},$$ then $$F_{10m+7}\equiv 1\pmod {15m+8}.$$ It results that the
number $10m+6$ is a period of the Fibonacci sequence modulo $15m+8$
with $15m+8$ prime and $m$ an odd positive integer.

If $$F_{5m+3}\equiv 0\pmod {15m+8}$$ and 
if $$F^2_{5m+2}-F_{5m+2}+1\equiv 0\pmod {15m+8}$$ and so $$F^2_{5m+2}\equiv
F_{5m+2}-1\pmod {15m+8},$$ then since $$F_{10m+7}\equiv F^2_{5m+2}\pmod
{15m+8},$$ $$F_{10m+7}\equiv F_{5m+2}-1\pmod {15m+8}.$$ Notice that in this
case, we cannot have $$F_{5m+2}\equiv -1\pmod {15m+8}$$ since
$3\not\equiv 0\pmod {15m+8}$ with $m$ an odd positive integer such
that $15m+8$ prime (and so $15m+8>3$). Then, let us assume absurdly that
if $$F^2_{5m+2}-F_{5m+2}+1\equiv 0\pmod {15m+8},$$ then
the number $10m+6$ is a period of the Fibonacci sequence modulo $15m+8$
with $15m+8$ prime and $m$ an odd positive integer. In such a case,
$$F_{10m+7}\equiv 1\pmod {15m+8}$$ which implies that $$F_{5m+2}\equiv
2\pmod {15m+8}.$$ Since $$F^2_{5m+2}\equiv F_{5m+2}-1\pmod {15m+8},$$ it
gives $4\equiv 1\pmod {15m+8}$. But, since $15m+8$ is a prime number
such that $m$ is an odd positive integer, we have $15m+8>4$ and so
$4\not\equiv 1\pmod {15m+8}$. So, we reach to a contradiction meaning
that if $$F^2_{5m+2}-F_{5m+2}+1\equiv 0\pmod {15m+8}$$ and so if
$$F_{5m+2}\not\equiv -1\pmod {15m+8},$$ the number $10m+6$ is not a period
of the Fibonacci sequence modulo $15m+8$ 
with $15m+8$ prime and $m$ an odd positive integer.

Moreover, if $$F_{5m+3}\equiv 0\pmod {15m+8}$$ and
reciprocally if the number $10m+6$ is a period
of the Fibonacci sequence modulo $15m+8$ 
with $15m+8$ prime and $m$ an odd positive integer, then
$$F_{10m+7}\equiv 1\pmod {15m+8}$$ which implies that $$F^2_{5m+2}\equiv
1\pmod {15m+8}.$$ So, either $$F_{5m+2}\equiv 1\pmod {15m+8}$$ or
$$F_{5m+2}\equiv -1\pmod {15m+8}.$$ Since we have also 
$$F^3_{5m+2}\equiv -1\pmod {15m+8}$$
(see above), it remains only one possibility, that is to say $$F_{5m+2}\equiv
-1\pmod {15m+8}.$$

$\frac{2(5k+4)}{3}=10m+6$ is a period of the Fibonacci sequence, we
must have $$F_{10m+7}\equiv F^2_{5m+2}\equiv 1\pmod {15m+8}$$ in addition
to the condition $$F_{5m+3}\equiv 0\pmod {15m+8}.$$ 

If $3F_{5m+2}F_{5m+4}\equiv -2F^2_{5m+3}\pmod {15m+8}$, then from Property \ref{peqc}, we can find an integer $c$ such that
$$
\left\{\begin{array}{c}
F_{5m+2}F_{5m+4}\equiv -2c\pmod {15m+8},
\\
F^2_{5m+3}\equiv 3c\pmod {15m+8}
\end{array}
\right.
$$
So
\begin{equation}\label{b0}
c\equiv F^2_{5m+3}+F_{5m+2}F_{5m+4}\pmod {15m+8},
\end{equation}
or equivalently ($F_{5m+4}=F_{5m+3}+F_{5m+2}$ and $F_{10m+5}=F^2_{5m+3}+F^2_{5m+2}$)
\begin{align}
c &\equiv F^2_{5m+3}+F_{5m+3}F_{5m+2}+F^2_{5m+2}\pmod {15m+8}\nonumber\\
&\equiv F_{10m+5}+F_{5m+3}F_{5m+2}\pmod {15m+8}.\label{b9}
\end{align}
So, if the number $10m+6$ with $m$ an odd positive integer 
is a period of the Fibonacci sequence modulo $15m+8$
with $15m+8$ prime, we should have $$F_{10m+6}\equiv 0\pmod {15m+8}$$
and $$F_{10m+5}\equiv F_{10m+7}\equiv 1\pmod {15m+8}.$$ Since
$$F_{10m+6}=F^2_{5m+4}-F^2_{5m+2}$$ and from the relations $F_{10m+6}\equiv 0 \pmod{15m+8}$ and $F_{10m+6}=F_{5m+4}^2-F_{5m+2}^2$, we have
$$F^2_{5m+2}\equiv F^2_{5m+4}\pmod {15m+8}$$ and $$c\equiv
1+F_{5m+2}F_{5m+3}\pmod {15m+8}.$$ So, either $$F_{5m+2}\equiv F_{5m+4}\pmod
{15m+8}$$ or $$F_{5m+2}\equiv -F_{5m+4}\pmod {15m+8}.$$ If $$F_{5m+2}\equiv
F_{5m+4}\pmod {15m+8},$$ then $$F_{5m+3}\equiv 0\pmod {15m+8}$$ and
$$c\equiv 1\equiv 0\pmod {15m+8}$$ where we used the fact that $$3c\equiv
F^2_{5m+3}\pmod {15m+8}$$ and $(3,15m+8)=1$ with $15m+8$ prime. But,
$1\not\equiv 0\pmod 
{15m+8}$. So, we reach a contradiction meaning that 
this case is not possible. Otherwise, if $$F_{5m+2}\equiv
-F_{5m+4}\pmod {15m+8},$$ then using the recurrence relation of the
Fibonacci sequence, we must have $$F_{5m+3}\equiv -2F_{5m+2}\pmod
{15m+8}$$ and so $$c\equiv 1-2F^2_{5m+2}\equiv 3F^2_{5m+2}\pmod {15m+8}$$
where we used \eqref{b0}. It
implies that $$5F^2_{5m+2}\equiv 1\pmod {15m+8}$$ and using Theorem \ref{flt}, it gives $$F^2_{5m+2}\equiv 5^{15m+6}\equiv
9m+5\pmod {15m+8}$$ since $$5^{15m+7}\equiv 1\equiv 45m+25\pmod {15m+8}$$
which implies that $5^{15m+6}\equiv 9m+5\pmod {15m+8}$ (reccall that
$15m+8$ is prime and so $(5,15m+8)=1$). Since $$F_{5m+3}\equiv -2F_{5m+2}\pmod
{15m+8},$$ $$F^2_{5m+3}\equiv 3c\pmod {15m+8}$$ and $$c\equiv
3F^2_{5m+2}\pmod {15m+8},$$ it results that 
$$F^2_{5m+3}\equiv 4F^2_{5m+2}\equiv 3c\equiv 9F^2_{5m+2}\pmod {15m+8}$$ and
so $4(9m+5)\equiv 9(9m+5)\pmod {15m+8}$. Since $4(9m+5)=36m+20\equiv
6m+4\pmod {15m+8}$, it
implies that $75m+41\equiv 0\pmod {15m+8}$ and so $1\equiv 0\pmod
{15m+8}$ which is
not possible since $1\not\equiv 0\pmod {15m+2}$. So, we obtain again a
contradiction meaning that this latter case is not also possible.

Therefore, when $5k+3=15m+8$ is prime with $k=3m+1$ and $m$ an odd positive
integer, if $10m+6$ is a period of the Fibonacci sequence modulo
$15m+8$ with $15m+8$ prime, $$F_{15m+9}\equiv 0\pmod {15m+8}$$ if and only if
$$F_{5m+3}\equiv 0\pmod {15m+8}.$$ Since $$F_{15m+9}=F_{5k+4}\equiv
0\pmod {5k+3}$$ is true when $5k+3$ is prime, we deduce that 
$$F_{\frac{5k+4}{3}}\equiv 0\pmod {5k+3}$$ is also true when $k\equiv
1\pmod 3$ and $5k+3$ prime. 

Thus, if $10m+6$ is a period of the Fibonacci sequence modulo
$15m+8$ with $15m+8$ prime, then we have
$$
F_{15m+9}\equiv 0\pmod {15m+8}$$
if and only if
$$
F_{5m+3}\equiv 0\pmod {15m+8}
$$
if and only if
$$
F_{5m+2}\equiv F_{5m+4}\pmod {15m+8}.
$$
Besides, $$F_{5m+3}\equiv 0\pmod {15m+8}$$ implies that $$F_{10m+6}\equiv
0\pmod {15m+8}.$$

Reciprocally, if $$F_{10m+6}\equiv 0\pmod {15m+8},$$
then $$F^2_{5m+2}\equiv F^2_{5m+4}\pmod {15m+8}.$$ So, either
$$F_{5m+2}\equiv F_{5m+4}\pmod {15m+8}$$ or $$F_{5m+2}\equiv -F_{5m+4}\pmod
{15m+8}.$$ If $$F_{5m+2}\equiv -F_{5m+4}\pmod {15m+8},$$ then
$$F_{5m+3}\equiv -2F_{5m+2}\pmod {15m+8}$$ and since $$F_{5m+3}\equiv
0\pmod {15m+8},$$ using the fact that $(2,15m+8)=1$ with $15m+8$ prime
such that $m$ an odd positive integer ($15m+8>2$), 
$$F_{5m+2}\equiv 0\pmod {15m+8}.$$ But, then, if
$$F_{10m+6}\equiv 0\pmod {15m+8},$$ we have
$$F_{15m+8}\equiv F_{10m+7}F_{5m+2}\equiv 0\pmod {15m+8}.$$ Or,
$$F_{15m+8}\equiv -1\pmod {15m+8}.$$ It leads to a contradiction meaning
that $$F_{5m+2}\equiv -F_{5m+4}\pmod {15m+8}$$ is not possible. So, if
$$F_{10m+6}\equiv 0\pmod {15m+8},$$ there
is only one possibility, that is to say $$F_{5m+2}\equiv F_{5m+4}\pmod
{15m+8}$$ which implies the congruence $$F_{5m+3}\equiv 0\pmod {15m+8}$$
and so which translates the congruence $$F_{15m+8}\equiv -1\pmod {15m+8}$$
into the congruence $$F^3_{5m+2}\equiv -1\pmod {15m+8}$$ which has at
least one solution. So, if $10m+6$ is a period of the Fibonacci sequence modulo
$15m+8$ with $15m+8$ prime, then we have
$$
F_{15m+9}\equiv 0\pmod {15m+8}$$
if and only if
$$
F_{5m+3}\equiv 0\pmod {15m+8}
$$
if and only if
$$
F_{5m+2}\equiv F_{5m+4}\pmod {15m+8}
$$
if and only if
$$
F_{10m+6}\equiv 0\pmod {15m+8}.
$$

Since $10m+6=2(5m+3)=\frac{2(5k+4)}{3}$ with $k=3m+1$ and $m$ an odd positive
integer, from above, we conclude that the number
$\frac{2(5k+4)}{3}$ is a period of the Fibonacci 
sequence modulo $5k+3$ if and only if $F_{\frac{5k+1}{3}}\equiv -1\pmod {5k+3}$.

\end{proof}

\begin{property}\label{p633}
Let $5k+4$ be a prime with $k$ an odd positive integer. Then,
for all $m\in[[ 0,5k]]$
\begin{equation}\label{c2}
F_{5k-m}\equiv (-1)^mF_{m+3}\pmod {5k+4}.
\end{equation}
\end{property}

\begin{proof}
From Properties \ref{p52} and \ref{p53}, we have
$$
F_{5k+3}\equiv 0\pmod {5k+4},
$$
and
$$
F_{5k+2}\equiv 1\pmod {5k+4}.
$$
Then, using the recurrence relation of the Fibonacci sequence, it
comes that
$$
F_{5k+1}\equiv -1\pmod {5k+4},
$$
$$
F_{5k}\equiv 2\pmod {5k+4},
$$
and
$$
F_{5k-1}\equiv -3\pmod {5k+4}.
$$
So, we verify \eqref{c2} is true when $m=0$ and $m=1$. 

Notice that \eqref{c2} is
verified when $m=5k$ since $F_0=0\equiv 0\pmod {5k+4}$ and $F_{5k+3}\equiv 0\pmod{5k+4}$.

Let assume for an integer $m\in[[ 0,5k-1]]$, we have
$F_{5k-i}\equiv (-1)^iF_{i+3}\pmod {5k+4}$ with $i=0,1,\ldots,m$. 
Then, using the recurrence relation of the Fibonacci sequence, 
we have ($0\leq m\leq 5k-1$),
\begin{align}
F_{5k-m-1} &= F_{5k-m+1}-F_{5k-m}\equiv (-1)^{m-1}F_{m+2}-(-1)^mF_{m+3}\pmod {5k+4}\nonumber \\
&\equiv (-1)^{m-1}(F_{m+2}+F_{m+3})\equiv (-1)^{m-1}F_{m+4}\pmod {5k+4}\nonumber \\
&\equiv (-1)^2(-1)^{m-1}F_{m+4}\equiv (-1)^{m+1}F_{m+4}\pmod {5k+4}\nonumber
\end{align}
since $(-1)^2=1$. It achieves the proof of Property \ref{p633} by
induction on the integer $m$. 
\end{proof}

Notice that Property \ref{p633} is also true for $m=-2,-1$.

\begin{remark}
Property \ref{p633} implies that we can limit ourself to the
integer interval $[ 1,\frac{5k+3}{2}]$ (knowing that
the case $m=0$ is a trivial case) in order to search or to rule out a
value for a possible
period of the Fibonacci sequence modulo $5k+4$ with $5k+4$ prime 
(such that $k$ is an odd positive integer) which is less than
$5k+3$. Notice that $5k+3$ is not in general the minimal period of the
Fibonacci sequence modulo $5k+4$ with $5k+4$ prime 
(such that $k$ is an odd positive integer). Indeed, for
instance, if $5k+4=29$ (and so for $k=5$), then it can be shown by calculating
the residue of $F_m$ with $m\in[ 1,14]$ modulo
$5k+4=29$, that the minimal period is $\frac{5k+3}{2}=14$.
\end{remark}

\begin{theorem}
Let $5k+4$ be a prime number with $k$ an odd positive number. If
$k\equiv 1\pmod 4$, then $$F_{\frac{5k+3}{2}}\equiv 0\pmod {5k+4}.$$
\end{theorem}

\begin{proof}

Since $5k+4$ with $k$ an odd positive number, is prime, the numbers 
$5k\pm 3$ are non-zero even positive integers. So, the numbers
$\frac{5k\pm 3}{2}$ are
non-zero positive integers. Moreover, if $k\equiv 1\pmod 4$, then
$5k-3\equiv 2\pmod 4$. So, the integer $\frac{5k-3}{2}$ is odd.

Using Property \ref{p633} and taking $m=\frac{5k-3}{2}$, it gives
$$
F_{\frac{5k+3}{2}}\equiv -F_{\frac{5k+3}{2}}\pmod {5k+2},
$$
or,
$$
2F_{\frac{5k+3}{2}}\equiv 0\pmod {5k+2},
$$
finally,
$$
F_{\frac{5k+3}{2}}\equiv 0\pmod {5k+2}
$$
since $2$ and $5k+4$ with $5k+4$ prime are relatively prime.
\end{proof}

\begin{theorem}\label{t1}
Let $5k+1$ be a prime with $k$ a non-zero positive even integer. If
$k\equiv 0\pmod 3$ and if $\frac{10k}{3}$ is a period of the Fibonacci
sequence modulo $5k+1$, then the congruence
$$
F_{5k}\equiv 0\pmod{5k+1}
$$
is equivalent to the congruence
$$
F_{\frac{5k}{3}}\equiv 0\pmod{5k+1}
$$
which is equivalent to the congruence
$$
F_{\frac{10k}{3}}\equiv 0\pmod{5k+1}.
$$
Moreover, if $k\equiv 0\pmod 3$ and if $F_{\frac{5k}{3}}\equiv
0\pmod{5k+1}$, then the number $\frac{10k}{3}$ is a period of the
Fibonacci sequence modulo $5k+1$ if and only if
$$
F_{\frac{5k\pm 3}{3}}\equiv 1\pmod{5k+1}.
$$
\end{theorem}

\begin{proof}
If $k\equiv 0\pmod 3$ and $k$ a non-zero positive even integer, then
there exists a non-zero positive integer $m$ such that $k=3m$. Notice
that $m$ is even since k is even. Since $F_{5k}\equiv 0\pmod{5k+1}$
with $5k+1$ prime ($k$ positive even), we 
have also $F_{15m}\equiv 0\pmod{15m+1}$ with $15m+1$ prime ($m$ positive even).
Using Theorem 1.27, we have
$$
F_{15m}=F_{10m+5m}=F_{5m}F_{10m+1}+F_{5m-1}F_{10m},
$$
$$
F_{10m-1}=F_{5m-1+5m}=F^2_{5m}+F^2_{5m-1}.
$$
From Remark 4.6, we have
$$
F_{10m}=F_{2\times 5m}=F_{5m}(F_{5m+1}+F_{5m-1})=F^2_{5m+1}-F^2_{5m-1},
$$
$$
F_{10m+1}=F_{2\times 5m+1}=F^2_{5m+1}+F^2_{5m}.
$$
So
\begin{align}
F_{15m} &= F_{5m}(F^2_{5m+1}+F^2_{5m})+F_{5m-1}F_{5m}(F_{5m+1}+F_{5m-1})\nonumber\\
 &= F_{5m}(F^2_{5m+1}+F^2_{5m}+F_{5m-1}F_{5m+1}+F^2_{5m-1})\nonumber\\
 &= F_{5m}(3F^2_{5m-1}+3F_{5m-1}F_{5m}+2F^2_{5m})\nonumber\\
 &= F_{5m}(3F_{5m-1}F_{5m+1}+2F^2_{5m}).\nonumber
\end{align}

So, the congruence $F_{15m}\equiv 0\pmod{15m+1}$ with $m$ an even positive
integer such that $15m+1$ prime is satisfied if and only if either
$$
F_{5m}\equiv 0\pmod{15m+1}
$$
or
$$
3F_{5m-1}F_{5m+1}\equiv -2F^2_{5m}\pmod{15m+1}.
$$
If $F_{5m}\equiv 0\pmod{15m+1}$, then from above, we have
necessarily
$$
F_{10m}\equiv 0\pmod{15m+1}.
$$
Using the recurrence relation of the Fibonacci sequence, it implies also that
$F_{5m+1}\equiv F_{5m-1}\pmod{15m+1}$. Moreover, we have
$$
F_{10m+1}\equiv F^2_{5m+1}\equiv F^2_{5m-1}\pmod{15m+1}.
$$
Or, using Theorem 1.27, we have
$$
F_{15m+1}=F_{5m+10m+1}=F_{10m+1}F_{5m+1}+F_{10m}F_{5m}.
$$
Since $F_{5k+1}\equiv 1\pmod{5k+1}$ with $5k+1$ prime ($k$ non-zero
positive even) and so if $k=3m$ such that $m$ non-zero positive even,
$$
F_{15m+1}\equiv 1\pmod{15m+1}
$$
with $15m+1$ prime ($m$ non-zero positive even), since
$$
F_{10m+1}\equiv F^2_{5m+1}\pmod{15m+1}
$$
and
$$
F_{10m}\equiv 0\pmod{15m+1}
$$
it implies that
$$
F_{10m+1}F_{5m+1}\equiv F^3_{5m+1}\equiv 1\pmod{15m+1}.
$$
We get
\be
\label{E1}
F^3_{5m+1}-1\equiv 0\pmod{15m+1}
\ee
and
$$
(F_{5m+1}-1)(F^2_{5m+1}+F_{5m+1}+1)\equiv 0\pmod{15m+1}.
$$
So, either
$$
F_{5m+1}-1\equiv 0\pmod{15m+1}
$$
or
$$
F^2_{5m+1}+F_{5m+1}+1\equiv 0\pmod{15m+1}.
$$
If
$$
F_{5m}\equiv 0\pmod{15m+1}
$$
and if
$$
F_{5m+1}-1\equiv 0\pmod{15m+1}
$$
and so
$$
F_{5m+1}\equiv 1\pmod{15m+1},
$$
then
$$
F_{10m+1}\equiv 1\pmod{15m+1}.
$$
It results that the number $10m$ is a period of the Fibonacci sequence
modulo $15m+1$ with $15m+1$ prime and $m$ a non-zero positive even integer.
If
$$
F_{5m}\equiv 0\pmod{15m+1}
$$
and if
$$
F^2_{5m+1}+F_{5m+1}+1\equiv 0\pmod{15m+1}
$$
and so
$$
F^2_{5m+1}\equiv -F_{5m+1}-1\pmod{15m+1}
$$
then since
\begin{align}
F_{10m+1} &\equiv F^2_{5m+1}\pmod{15m+1}\nonumber\\
 &\equiv -F_{5m+1}-1\pmod{15m+1}.\nonumber
 \end{align}
 
Notice that in this case, we cannot have
$$
F_{5m+1}\equiv 1\pmod{15m+1}
$$
since $3\not\equiv 0\pmod{15m+2}$ with $m$ a non-zero positive even
integer such that $15m+1$ prime (and so $15m+1>3$). Then, let us assume
absurdly that if
$$
F^2_{5m+1}+F_{5m+1}+1\equiv 0\pmod{15m+1}
$$
then the number $10m$ is a period of the Fibonacci sequence modulo
$15m+1$ with $15m+1$ prime and $m$ a non-zero positive even
integer. In such a case,
$$
F_{10m+1}\equiv 1\pmod{15m+1}
$$
which implies that
$$
F_{5m+1}\equiv -2\pmod{15m+1}.
$$
Since
$$
F^2_{5m+1}\equiv -F_{5m+1}-1\pmod{15m+1}
$$
it gives
$$
4\equiv 1\pmod{15m+1}.
$$
But, since $15m+1$ is a prime number such that $m$ is a non-zero
positive even integer, we have $15m+1>4$ and so $4\not\equiv
1\pmod{15m+1}$. So, we reach a contradiction meaning that if
$$
F^2_{5m+1}+F_{5m+1}+1\equiv 0\pmod{15m+1}
$$
and so if
$$
F_{5m+1}\not\equiv 1\pmod{15m+1}
$$
the number $10m$ is not a period of the Fibonacci sequence modulo
$15m+1$ with $15m+1$ prime and $m$ a non-zero positive even integer.
Moreover, if
$$
F_{5m}\equiv 0\pmod{15m+1}
$$
and reciprocally if the number $10m$ is a period of the Fibonacci
sequence modulo $15m+1$ with $15m+1$ prime and $m$ a non-zero positive
even integer, then
$$
F_{10m+1}\equiv 1\pmod{15m+1}
$$
which implies that
$$
F^2_{5m+1}\equiv 1\pmod{15m+1}.
$$
So, either
$$
F_{5m+1}\equiv 1\pmod{15m+1}
$$
or
$$
F_{5m+1}\equiv -1\pmod{15m+1}.
$$
Since we have (\ref{E1}), it remains only one possibility, that is to
say
$$
F_{5m+1}\equiv 1\pmod{15m+1}
$$
$\frac{10k}{3}=10m$ is a period of the Fibonacci sequence modulo
$15m+1$, we must have $F_{10m+1}\equiv F^2_{5m+1}\equiv 1\pmod{15m+1}$
in addition to the condition
$$
F_{5m}\equiv 0\pmod{15m+1}.
$$
If
$$
3F_{5m-1}F_{5m+1}\equiv -2F^2_{5m}\pmod{15m+2}
$$
then from Property 1.3, we can find an integer $c$
such that
$$
\left\{\begin{array}{c}
F_{5m-1}F_{5m+1}\equiv -2c\pmod{15m+1},
\\
F^2_{5m}\equiv 3c\pmod{15m+1},
\end{array}\right.
$$
or equivalently ($F_{5m+1}=F_{5m}+F_{5m-1}$ and $F_{10m-1}=F^2_{5m}+F^2_{5m-1}$)
\begin{align}
c &\equiv F^2_{5m}+F_{5m-1}F_{5m+1}\pmod{15m+1}\nonumber\\
 &\equiv F^2_{5m}+F_{5m-1}F_{5m}+F^2_{5m-1}\nonumber\\
 &\equiv F_{10m-1}+F_{5m-1}F_{5m}\pmod{15m+1}.\nonumber
 \end{align}
So, if the number $10m$ with $m$ a non-zero positive even integer is a
period of the Fibonacci sequence modulo $15m+1$ with $15m+1$ prime, we
should have
$$
F_{10m}\equiv 0\pmod{15m+1}
$$
and
$$
F_{10m-1}\equiv F_{10m+1}\equiv 1\pmod{15m+1}.
$$
Since
$$
F_{10m}=F^2_{5m+1}-F^2_{5m-1}
$$
and
$$
c\equiv F_{10m-1}+F_{5m-1}F_{5m}\pmod{15m+1}
$$
it implies that
$$
F^2_{5m+1}\equiv F^2_{5m-1}\pmod{15m+1}
$$
and
$$
c\equiv 1+F_{5m-1}F_{5m}\pmod{15m+1}.
$$
So, either
$$
F_{5m+1}\equiv F_{5m-1}\pmod{15m+1}
$$
or
$$
F_{5m+1}\equiv -F_{5m-1}\pmod{15m+1}.
$$
If
$$
F_{5m+1}\equiv F_{5m-1}\pmod{15m+1}
$$
then
$$
F_{5m}\equiv 0\pmod{15m+1}
$$
and
$$
c\equiv 1\equiv 0\pmod{15m+1}
$$
where we used the fact that
$$
3c\equiv F^2_{5m}\pmod{15m+1}
$$
and $(3,15m+1)=1$ with $15m+1$ prime. But, $1\not\equiv
0\pmod{15m+1}$. So, we reach a contradiction meaning that this case is
not possible. Otherwise, if
$$
F_{5m+1}\equiv -F_{5m-1}\pmod{15m+1}
$$
then using the recurrence relation of the Fibonacci sequence, we must
have
$$
F_{5m}\equiv -2F_{5m-1}\pmod{15m+1}
$$
and so
$$
c\equiv 1-2F^2_{5m-1}\equiv 3F^2_{5m-1}\pmod{15m+1}
$$
where we used the fact that
$$
c\equiv F^2_{5m}+F_{5m-1}F_{5m+1}\pmod{15m+1}.
$$
It implies that
$$
5F^2_{5m-1}\equiv 1\pmod{15m+1}
$$
and using Theorem 1.5, it gives
$$
F^2_{5m-1}\equiv 5^{15m-1}\equiv 12m+1\pmod{15m+1}
$$
since $5^{15m}\equiv 1\equiv 60m+5\pmod{15m+1}$ which implies that
$5^{15m-1}\equiv 12m+1\pmod{15m+1}$ (recall that $15m+1$ is prime and
so $(5,15m+1)=1$. Since
$$
F_{5m}\equiv -2F_{5m-1}\pmod{15m+1}
$$
$$
F^2_{5m}\equiv 3c\pmod{15m+1}
$$
and
$$
c\equiv 3F^2_{5m-1}\pmod{15m+1}
$$
it results that
$$
F^2_{5m}\equiv 9F^2_{5m-1}\equiv 3c\equiv 4F^2_{5m-1}\pmod{15m+1}
$$
and so $4(12m+1)\equiv 9(12m+1)\pmod{15m+1}$. Since
$4(12m+1)=48m+4\equiv 3m+1\pmod{15m+1}$, it implies that $105m+8\equiv
0\pmod{15m+1}$ and so $1\equiv 0\pmod{15m+1}$ which is not possible
since $1\not\equiv 0\pmod{15m+1}0$. So, we obtain again a
contradiction meaning that this latter case is not also possible.

Therefore, when $5k+1=15m+1$ is prime with $k=3m$ and $m$ a non-zero
positive even integer, if $10m$ is a period of the Fibonacci sequence
modulo $15m+1$ with $15m+1$ prime, then
$$
F_{15m}\equiv 0\pmod{15m+1}
$$
if and only if
$$
F_{5m}\equiv 0\pmod{15m+1}.
$$
Since
$$
F_{15m}=F_{5k}\equiv 0\pmod{5k+1}
$$
is true when $5k+1$ is prime, we deduce that
$$
F_{\frac{5k}{3}}\equiv 0\pmod{5k+1}
$$
is also true when $k\equiv 0\pmod 3$ and $5k+1$ prime.

Thus, if $10m$ is a period of the Fibonacci sequence modulo $15m+1$
with $15m+1$ prime, then we have
$$
F_{15m}\equiv 0\pmod{15m+1}
$$
if and only if
$$
F_{5m}\equiv 0\pmod{15m+1}
$$
if and only if
$$
F_{5m-1}\equiv F_{5m+1}\pmod{15m+1}.
$$
Besides,
$$
F_{5m}\equiv 0\pmod{15m+1}
$$
implies that
$$
F_{10m}\equiv 0\pmod{15m+1}.
$$
Reciprocally, if
$$
F_{10m}\equiv 0\pmod{15m+1}
$$
then
$$
F^2_{5m+1}\equiv F^2_{5m-1}\pmod{15m+1}.
$$
So, either
$$
F_{5m+1}\equiv F_{5m-1}\pmod{15m+1}
$$
or
$$
F_{5m+1}\equiv -F_{5m-1}\pmod{15m+1}.
$$
If
$$
F_{5m+1}\equiv -F_{5m-1}\pmod{15m+1}
$$
then
$$
F_{5m}\equiv -2F_{5m-1}\pmod{15m+1}
$$
and since
$$
F_{5m}\equiv 0\pmod{15m+1}
$$
using the fact that $(2,15m+1)=1$ with $15m+1$ prime such that $m$ is
a non-zero positive even integer ($15m+1>2$),
$$
F_{5m-1}\equiv 0\pmod{15m+1}.
$$
But, then, if
$$
F_{10m}\equiv 0\pmod{15m+1}
$$
we have
$$
F_{15m+1}\equiv F_{10m+1}F_{5m+1}\equiv 0\pmod{15m+1}.
$$
Or,
$$
F_{15m+1}\equiv 1\pmod{15m+1}.
$$
It leads to a contradiction meaning that
$$
F_{5m+1}\equiv -F_{5m-1}\pmod{15m+1}
$$
is not possible. So, if
$$
F_{10m}\equiv 0\pmod{15m+1}
$$
there is only one possibility, that is to say
$$
F_{5m+1}\equiv F_{5m-1}\pmod{15m+1}
$$
which implies the congruence
$$
F_{5m}\equiv 0\pmod{15m+1}
$$
and so which translates the congruence
$$
F_{15m+1}\equiv 1\pmod{15m+1}
$$
into the congruence
$$
F^3_{5m+1}\equiv 1\pmod{15m+1}
$$
which has at least one solution. So, if $10m$ is a period of the
Fibonacci sequence modulo $15m+1$ with $15m+1$ prime, then we have
$$
F_{15m}\equiv 0\pmod{15m+1}
$$
if and only if
$$
F_{5m}\equiv 0\pmod{15m+1}
$$
if and only if
$$
F_{5m+1}\equiv F_{5m-1}\pmod{15m+1}
$$
if and only if
$$
F_{10m}\equiv 0\pmod{15m+1}.
$$
Since $10m=\frac{10k}{3}$ with $k=3m$ and $m$ a non-zero positive even
integer, from above, we conclude that if $F_{\frac{5k}{3}}\equiv
0\pmod{5k+1}$, then $\frac{10k}{3}$ is a period of the Fibonacci
sequence modulo $5k+1$ with $5k+1$ prime if and only if
$$
F_{\frac{5k\pm 3}{3}}\equiv 1\pmod{5k+1}.
$$
\end{proof}

\begin{theorem}\label{t2}
Let $5k+4$ be a prime with $k$ an odd positive integer. If
$k\equiv 0\pmod 3$ and if $\frac{2(5k+3)}{3}$ is a period of the Fibonacci
sequence modulo $5k+4$, then the congruence
$$
F_{5k+3}\equiv 0\pmod{5k+4}
$$
is equivalent to the congruence
$$
F_{\frac{5k+3}{3}}\equiv 0\pmod{5k+4}
$$
which is equivalent to the congruence
$$
F_{\frac{2(5k+3)}{3}}\equiv 0\pmod{5k+4}.
$$
Moreover, if $k\equiv 0\pmod 3$ and if $F_{\frac{5k+3}{3}}\equiv
0\pmod{5k+4}$, then the number $\frac{2(5k+3)}{3}$ is a period of the
Fibonacci sequence modulo $5k+4$ if and only if
$$
F_{\frac{5k}{3}}\equiv 1\pmod{5k+4}
$$
\end{theorem}
\begin{proof}
The proof is very similar than to proof of Theorem 4.33.
\end{proof}

The next theorem below is a generalization of Theorems 4.33, 4.37 and Theorem \ref{t1} and \ref{t2} given above. The
number $\ell_{5k+r}$ with $5k+r$ prime such that $r\in[[1,4]]$ and $k\equiv r+1\pmod 2$ is a period of the Fibonacci
sequence modulo $5k+r$. Its expression is given in Corollary 4.10.

\begin{theorem}\label{t3}
Let $5k+r$ be a prime such that $r\in[[1,4]]$ and
$k\equiv r+1\pmod 2$. If 
$k\equiv \frac{(r-1)(r-2)}{2}\pmod 3$ and if $\frac{\ell_{5k+r}}{3}$
is a period of the Fibonacci 
sequence modulo $5k+r$, then the congruence
$$
F_{\frac{\ell_{5k+r}}{2}}\equiv 0\pmod{5k+r}
$$
is equivalent to the congruence
$$
F_{\frac{\ell_{5k+r}}{6}}\equiv 0\pmod{5k+r}
$$
which is equivalent to the congruence
$$
F_{\frac{\ell_{5k+r}}{3}}\equiv 0\pmod{5k+r}.
$$
Moreover, if $k\equiv \frac{(r-1)(r-2)}{2}\pmod 3$ and if 
$F_{\frac{\ell_{5k+r}}{6}}\equiv 0\pmod{5k+r}$, then the number
$\frac{\ell_{5k+r}}{3}$ is a period of the 
Fibonacci sequence modulo $5k+r$ if and only if
$$
F_{\frac{\ell_{5k+r}}{6}-1}\equiv\left\{\begin{array}{ccccc}
1\pmod{5k+r} & if & r=1 & or & r=4,
\\
-1\pmod{5k+r} & if & r=2 & or & r=3.
\end{array}\right.
$$
\end{theorem}

\begin{proof}
The results stated in Theorem \ref{t3} can be deduced from Theorems
4.33, 4.37 and Theorems \ref{t1} and \ref{t2} given above.
\end{proof}

\section{{Some results on Generalized Fibonacci numbers}}

In this section, we deduce some small results related to the generalized fibonacci numbers as defined below.

\begin{definition}\label{d4}
Let $a,b,r$ be three numbers. The sequence
$(C_{n,2}(a,b,r))$ is defined by
$$
C_{n,2}(a,b,r)=C_{n-1,2}(a,b,r)+C_{n-2,2}(a,b,r)+r,\,\,\,\forall\,n\geq 2$$ with$$\left\{\begin{array}{l}
C_{0,2}(a,b,r)=b-a-r,
\\
C_{1,2}(a,b,r)=a.
\end{array}
\right.
$$
\end{definition}

In particular, we have
$$
F_n=C_{n,2}(1,1,0),\,\,\,\forall\,n\geq 0.
$$

\begin{remark}\label{r5}
This sequence can be defined from $n=1$ by setting $C_{2,2}(a,b,r)=b$
as in \cite{bj}.
\end{remark}

\begin{proposition}\label{p7}
Let $a,b,r$ be three numbers.
The sequences $(C_{n,2}(a,b,r))$, $(C_{n,2}(1,0,-1))$, $(F_n)$ satisfies
$$
C_{n,2}(a,b,r)=aF_{n-2}+bF_{n-1}-rC_{n+1,2}(1,0,-1),\,\,\,\forall\,n\geq 2.
$$
\end{proposition}
\begin{proof}
Let $a,b,r$ be three numbers.
Let us prove Proposition \ref{p7} by induction on the integer $n\geq 2$. We have
$$
C_{2,2}(a,b,r)=b=a\times 0+b\times 1+r\times 0
=a\times F_0+b\times F_1-r\times C_{3,2}(1,0,-1)
$$
Let us assume
that this proposition is true up to $n\geq 2$. Using the recurrence
relations of sequences $(C_{n,2}(a,b,r))$, $(C_{n,2}(1,0,-1))$ and
$(F_n)$, we have
\begin{align}
C_{n+1,2}(a,b,r) &= C_{n,2}(a,b,r)+C_{n-1,2}(a,b,r)+r\nonumber\\
 &= (aF_{n-2}+bF_{n-1}-rC_{n+1,2}(1,0,-1))+(aF_{n-3}+bF_{n-2}
-rC_{n,2}(1,0,-1))+r\nonumber\\
 &= a(F_{n-2}+F_{n-3})+b(F_{n-1}+F_{n-2})
-r(C_{n+1,2}(1,0,-1)+C_{n,2}(1,0,-1)-1)\nonumber\\
 &= aF_{n-1}+bF_n-rC_{n+2,2}(1,0,-1)\nonumber
 \end{align}
 
Thus by induction, the proof is complete.
\end{proof}

\begin{proposition}\label{p8}
The sequences $(C_{n,2}(1,0,-1))$ and $(F_n)$ satisfies
$$
C_{n,2}(1,0,-1)=C_{n-2,2}(1,0,-1)-F_{n-2},\,\,\,\forall\,n\geq 2.
$$
$$
C_{n,2}(1,0,-1)=-{\displaystyle\sum^{n-3}_{k=1}}F_k,\,\,\,\forall\,n\geq 4.
$$
From Proposition \ref{p7}, for any numbers $a,b,r$, it results that
$$
C_{n,2}(a,b,r)=aF_{n-2}+bF_{n-1}+r{\displaystyle\sum^{n-2}_{k=1}}F_k,
\,\,\,\forall\,n\geq 2
$$
$$
C_{n,2}(a,b,r)=aF_{n-2}+bF_{n-1}+r(F_n-1),
\,\,\,\forall\,n\geq 2.
$$
\end{proposition}

This result can be easily verifies using mathematical induction and Theorem 1.26 and Proposition 5.4. We shall omit the details here.

The theorem below appears in any standard linear algebra textbook. 

\begin{theorem}\label{t9}
(i) A linear recurrence sequence $(u_n)_{n\geq 0}$ of order $2$ which satisfies a
linear recurrence relation as
$$
u_n=\alpha_1u_{n-1}+\alpha_2u_{n-2},\,\,\,\forall\,n\geq 2
$$
with $\alpha_1$, $\alpha_2$ in a field $K$ ($K=\mathbb{R}$
or $K=\mathbb{C}$), is completely and uniquely determined by its first
terms $u_0$ and $u_1$. 
\\[0.1in]
(ii) If $(u_n)_{n\geq 0}$, $(v_n)_{n\geq 0}$ are
two linear recurrence sequences of order $2$ such that
$$
\det\left(\begin{array}{ccc}
u_0 & v_0
\\
u_1 & v_1
\end{array}\right)=u_0 v_1-u_1 v_0\neq 0
$$ 
then any linear recurrence sequence $(w_n)_{n\geq 0}$ of order $2$ is
uniquely written as
$$
(w_n)_{n\geq 0}=\lambda (u_n)_{n\geq 0}+\mu (v_n)_{n\geq 0}
$$
with $\lambda$, $\mu$ in a field $K$ ($K=\mathbb{R}$
or $K=\mathbb{C}$).
\end{theorem}

\begin{proof}
The statement (i) is proved by induction.
\\[0.1in]
The statement (ii) can be proved from (i) and from the Cramer's rule for
system of linear equations.
\end{proof}

\begin{definition}\label{d10}
Let $k$ be an integer which is greater than $2$ and let
$a_0,\ldots, a_{k-1}$ be $k$ numbers. The sequence $(F_{n,k}(a_0,\ldots,a_{k-1}))$ for $k\geq 2$ is defined by
$$
F_{n,k}(a_0,\ldots,a_{k-1})=F_{n-1,k}(a_0,\ldots,a_{k-1})+F_{n-k,k}(a_0,\ldots,a_{k-1}),
\,\,\,\forall\,n\geq k$$
with $$F_{i,k}(a_0,\ldots,a_{k-1})=a_i,\,\,\,\forall\,i\in\{0,\ldots,k-1\}.
$$
The sequence $(F_{n,k}(a_0,\ldots,a_{k-1}))$ is called
the $k$-Fibonacci sequence with initial conditions $a_0,\ldots,a_{k-1}$.
\end{definition}

\begin{proposition}\label{p11}
Let $a_0,a_1$ be two numbers. The $2$-Fibonacci numbers sequence
$(F_{n,2}(a_0,a_1))$ has general term
$$
F_{n,2}(a_0,a_1)=\alpha \varphi^n+\beta (1-\varphi)^n,
\,\,\,\forall\,n\geq 0
$$
where $\varphi=\frac{1+\sqrt{5}}{2}$ is the golden ratio and
$$
\alpha=\frac{a_0(\varphi-1)+a_1}{\sqrt{5}},\qquad\qquad
\beta=\frac{a_0\varphi-a_1}{\sqrt{5}}.
$$
In particular, we have
$$
F_n=F_{n,2}(0,1),\qquad\qquad L_n=F_{n,2}(2,1).
$$
\end{proposition}

\begin{proof}
Let $a_0,a_1$ be two numbers. Using the relation of recurrence of the sequence
$(F_{n,2}(a_0,a_1))$ and taking the Ansatz
$F_{n,2}(a_0,a_1)=z^n$, we have for $n\geq 2$
$$
z^n=z^{n-1}+z^{n-2}.
$$
For $z\neq 0$, it gives ($n\geq 2$) $z^2-z-1=0$. The discriminant of this polynomial equation of second degree is
$\Delta=\sqrt{5}$. So, the roots of this equation are:
$$
\varphi=\frac{1+\sqrt{5}}{2},\qquad\qquad 1-\varphi=\frac{1-\sqrt{5}}{2}.
$$
We can notice that any linear combination of $\varphi^n,(1-\varphi)^n$
for $n\geq 0$ verifies the equation $z^n=z^{n-1}+z^{n-2}$ for $n\geq 0$. 
Since $0=0\cdot \varphi^n=0\cdot (1-\varphi)^n$, the
sequences which satisfy the recurrence relation of sequence
$(F_{n,2}(a_0,a_1))$ form a vector subspace of the set of
complex sequences. Given $a_0,a_1$, from Theorem \ref{t9} above,
since
$$
\det\left(\begin{array}{ccc}
1 & 1
\\
\varphi & 1-\varphi
\end{array}\right)=1-2\varphi=-\sqrt{5}\neq 0
$$
we deduce that there exist 
two numbers $\alpha,\beta$ such that
$$
F_{n,2}(a_0,a_1)=\alpha \varphi^n+\beta (1-\varphi)^n.
$$
Since
$$
F_{0,2}(a_0,a_1)=a_0\qquad\qquad F_{1,2}(a_0,a_1)=a_1,
$$
the coefficients $\alpha,\beta$ verify the matrix equation
$$
\left(\begin{array}{ccc}
1 & 1
\\
\varphi & 1-\varphi
\end{array}\right)
\left(\begin{array}{c}
\alpha
\\
\beta
\end{array}\right)
=\left(\begin{array}{c}
a_0
\\
a_1
\end{array}\right)
$$
So:
$$
\left(\begin{array}{c}
\alpha
\\
\beta
\end{array}\right)=\left(\begin{array}{ccc}
1 & 1
\\
\varphi & 1-\varphi
\end{array}\right)^{-1}\left(\begin{array}{c}
a_0
\\
a_1
\end{array}\right)
$$
where
$$
\left(\begin{array}{ccc}
1 & 1
\\
\varphi & 1-\varphi
\end{array}\right)^{-1}=\frac{1}{\sqrt{5}}\left(\begin{array}{ccc}
\varphi-1 & 1
\\
\varphi & -1
\end{array}\right)
$$
So:
$$
\left(\begin{array}{c}
\alpha
\\
\beta
\end{array}\right)=\frac{1}{\sqrt{5}}\left(\begin{array}{c}
a_0(\varphi-1)+a_1
\\
a_0\varphi-a_1
\end{array}\right).
$$
\end{proof}

\begin{proposition}\label{p12}
Let $a_0,a_1$ be two numbers. We have
$$
F_{n,2}(a_0,a_1)=a_0F_{n+1}+(a_1-a_0)F_n,\,\,\,\forall\,n\geq 0.
$$
\end{proposition}

\begin{proof}
From Proposition \ref{p11}, we have
\begin{align}
F_{n,2}(a_0,a_1) &= \frac{(a_0(\varphi-1)+a_1)\varphi^n
+(a_0\varphi-a_1)(1-\varphi)^n}{\sqrt{5}}\nonumber\\
 &= \frac{a_0\left[(\varphi-1)\varphi^n+\varphi(1-\varphi)^n\right]
+a_1\left[\varphi^n-(1-\varphi)^n\right]}{\sqrt{5}}\nonumber\\
 &= -a_0\left\{\frac{(1-\varphi)\varphi^n
-\varphi(1-\varphi)^n}{\sqrt{5}}\right\}
+a_1\left\{\frac{\varphi^n-(1-\varphi)^n}{\sqrt{5}}\right\}\nonumber\\
 &= -a_0\left\{\frac{(1-\varphi)\varphi^n
+(1-\varphi-1)(1-\varphi)^n}{\sqrt{5}}\right\}+a_1F_n\nonumber\\
 &= -a_0\left\{\frac{\varphi^n-(1-\varphi)^n}{\sqrt{5}}
-\left(\frac{\varphi^{n+1}-(1-\varphi)^{n+1}}{\sqrt{5}}\right)\right\}+a_1F_n\nonumber\\
 &= -a_0(F_n-F_{n+1})+a_1F_n\nonumber\\
 &= a_0F_{n+1}+(a_1-a_0)F_n.\nonumber
 \end{align}
\end{proof}

\begin{proposition}\label{p13}
Let $z$ be a real complex number such that $\varphi |z|<1$. We have
$$
{\displaystyle\sum^{+\infty}_{n=0}}F_nz^n=\frac{z}{(1-\varphi z)(1-z+\varphi z)}
=\frac{z}{1-z-z^2}
$$
\end{proposition}

This is a standard result and we omit the proof here.

\begin{example}\label{ex14}
Applying Proposition \ref{p13} when $z=1/2$, we have
\begin{align}
{\displaystyle\sum^{+\infty}_{n=0}}\frac{F_n}{2^n} &
= \frac{1/2}{\left(1-\frac{\varphi}{2}\right)
\left(\frac{1}{2}+\frac{\varphi}{2}\right)}\nonumber\\
 &= \frac{1/2}
{\left(\frac{2-\varphi}{2}\right)\left(\frac{1+\varphi}{2}\right)}\nonumber\\
 &= \frac{2}
{2+2\varphi-\varphi-\varphi^2}.\nonumber\\
 &= \frac{2}{2+\varphi-(\varphi+1)}.\nonumber\\
 &= 2\nonumber
 \end{align}
Thus $$
{\displaystyle\sum^{+\infty}_{n=0}}\frac{F_n}{2^{n+1}}=1.
$$
\end{example}

\begin{proposition}\label{p15}
Let $z$ be a real complex number such that $\varphi |z|<1$. Let $a_0$
and $a_1$ be two numbers. We have the generating function
$$
{\displaystyle\sum^{+\infty}_{n=0}}F_{n,2}(a_0,a_1)z^n
=\frac{a_0+(a_1-a_0)z}{(1-\varphi z)(1-z+\varphi z)}
=\frac{a_0+(a_1-a_0)z}{1-z-z^2}.
$$
\end{proposition}

\begin{proof}
Let $z$ be a real complex number such that $\varphi |z|<1$. 
When $z=0$, we have
$$
\left({\displaystyle\sum^{+\infty}_{n=0}}F_{n,2}(a_0,a_1)z^n\right)_{z=0}=F_{0,2}(a_0,a_1)=a_0
$$
and
$$
\left(\frac{a_0+(a_1-a_0)z}{(1-\varphi z)(1-z+\varphi z)}\right)_{z=0}=a_0.
$$
So, the formula of Proposition \ref{p15} is true for $z=0$. In the
following, we assume that $z\neq 0$. From Proposition \ref{p12}, we know that
$$
F_{n,2}(a_0,a_1)=a_0F_{n+1}+(a_1-a_0)F_n,\,\,\,\forall\,n\geq 0.
$$
So, using Proposition \ref{p13}, we have ($\varphi |z|<1$ and $z\neq 0$)
$$
{\displaystyle\sum^{+\infty}_{n=0}}F_{n,2}(a_0,a_1)z^n
=a_0{\displaystyle\sum^{+\infty}_{n=0}}F_{n+1}z^n
+(a_1-a_0){\displaystyle\sum^{+\infty}_{n=0}}F_nz^n.
$$
Or ($\varphi |z|<1$ and $z\neq 0$)
$$
{\displaystyle\sum^{+\infty}_{n=0}}F_{n+1}z^n=\frac{1}{z}
{\displaystyle\sum^{+\infty}_{n=0}}F_{n+1}z^{n+1}
=\frac{1}{z}{\displaystyle\sum^{+\infty}_{n=1}}F_nz^n
=\frac{1}{z}{\displaystyle\sum^{+\infty}_{n=0}}F_nz^n
$$
where we used the fact that $F_0=0$.
\\[0.1in]
It follows that ($\varphi |z|<1$ and $z\neq 0$)
$$
{\displaystyle\sum^{+\infty}_{n=0}}F_{n,2}(a_0,a_1)z^n
=\left(\frac{a_0}{z}+a_1-a_0\right){\displaystyle\sum^{+\infty}_{n=0}}F_nz^n.
$$
From Proposition \ref{p13}, it results that ($\varphi |z|<1$ and $z\neq 0$)
\begin{align}
{\displaystyle\sum^{+\infty}_{n=0}}F_{n,2}(a_0,a_1)z^n
 &= \left(\frac{a_0+(a_1-a_0)z}{z}\right)\frac{z}{(1-\varphi z)(1-z+\varphi z)}\nonumber\\
 &= \frac{a_0+(a_1-a_0)z}{(1-\varphi z)(1-z+\varphi z)}=\frac{a_0+(a_1-a_0)z}{1-z-z^2}.\nonumber
 \end{align}
 
Since this relation is also true for $z=0$ (see above), this relation
is true for $\varphi |z|<1$.
\end{proof}

\begin{example}\label{ex16}
Applying Proposition \ref{p15} when $a_0=2$, $a_1=1$ and $z=1/3$, since
$F_{n,2}(2,1)=L_n$ for all $n\geq 0$, we have
\begin{align}
{\displaystyle\sum^{+\infty}_{n=0}}\frac{L_n}{3^n} &= \frac{2-\frac{1}{3}}
{\left(1-\frac{\varphi}{3}\right)\left(1-\frac{1}{3}+\frac{\varphi}{3}\right)}\nonumber\\
 &= \frac{\frac{5}{3}}
{\left(\frac{3-\varphi}{3}\right)\left(\frac{2+\varphi}{3}\right)}\nonumber\\
 &= \frac{15}{6+3\varphi-2\varphi-\varphi^2}\nonumber\\
 &= \frac{15}{6+\varphi-(\varphi+1)}\nonumber\\
 &= \frac{15}{5}=3.\nonumber
\end{align}
Therefore
$$
{\displaystyle\sum^{+\infty}_{n=0}}\frac{L_n}{3^{n+1}}=1.
$$
\end{example}

\begin{proposition}\label{p17}
Let $z$ be a real complex number such that $\varphi |z|<1$. Let
$a,b,r$ be three numbers. We have
$$
{\displaystyle\sum^{+\infty}_{n=0}}C_{n,2}(a,b,r)z^n
=b-a-r+az+\frac{z^2\left[(az+b)(1-z)+rz\right]}{1-2z+z^3}
$$
or equivalently
$$
{\displaystyle\sum^{+\infty}_{n=0}}C_{n,2}(a,b,r)z^n
=\frac{a(1-z)(2z-1)+b(1-z)^2+r(2z-1)}{1-2z+z^3}.
$$
\end{proposition}

This result can be derived routinely using the results we have derived so far. Although the proof is a little involved, but it follows essentially the same pattern as the previous result. So for the sake of brevity we shall omit it here.

\begin{example}\label{ex18}
Applying Proposition \ref{p17} when $z=1/2$, we have
$$
{\displaystyle\sum^{+\infty}_{n=0}}\frac{C_{n,2}(a,b,r)}{2^n}
=\frac{\frac{b}{4}}{\frac{1}{8}}=2b.
$$
So
$$
{\displaystyle\sum^{+\infty}_{n=0}}\frac{C_{n,2}(a,b,r)}{2^{n+1}}=b.
$$
Applying Proposition \ref{p17} when $a=-r=1$, $b=0$ and $z=1/3$, we have
$$
{\displaystyle\sum^{+\infty}_{n=0}}\frac{C_{n,2}(1,0,-1)}{3^n}
=\frac{\frac{2}{3}\left(-\frac{1}{3}\right)-\left(-\frac{1}{3}\right)}
{1-\frac{2}{3}+\frac{1}{27}}
=\frac{-\frac{2}{9}+\frac{1}{3}}{\frac{1}{3}+\frac{1}{27}}
=\frac{\frac{1}{9}}{\frac{10}{27}}=\frac{3}{10}.
$$
So,
$$
{\displaystyle\sum^{+\infty}_{n=0}}\frac{C_{n,2}(1,0,-1)}{3^{n+1}}=\frac{1}{10}.
$$
\end{example}

\begin{proposition}\label{p19}
Let $a_0,a_1$ be two numbers. We have
$$
F_{k+l,2}(a_0,a_1)=F_{l,2}(a_0,a_1)F_{k+1}+F_{l-1,2}(a_0,a_1)F_k,
\,\,\,\forall\,k\geq 0,\,\,\,\forall\,l\geq 1,
$$
or equivalently
$$
F_{k+l,2}(a_0,a_1)=F_{k,2}(F_{l,2}(a_0,a_1),F_{l+1,2}(a_0,a_1)),
\,\,\,\forall\,k\geq 0,\,\,\,\forall\,l\geq 0.
$$
\end{proposition}

\begin{proof}
Let $a_0,a_1$ be two numbers. From Proposition \ref{p12}, we know
that for $k+l\geq 0$ we have
$$
F_{k+l,2}(a_0,a_1)=a_0F_{k+l+1}+(a_1-a_0)F_{k+l}.
$$
Using Theorem 1.27, we have
\begin{align}
F_{k+l,2}(a_0,a_1) &= a_0(F_{l+1}F_{k+1}+F_lF_k)+(a_1-a_0)(F_lF_{k+1}+F_{l-1}F_k)\nonumber\\
 &= (a_0F_{l+1}+(a_1-a_0)F_l)F_{k+1}+(a_0F_l+(a_1-a_0)F_{l-1})F_k.\nonumber
 \end{align}
 
Using Proposition \ref{p12}, we get 
\begin{align}
F_{k+l,2}(a_0,a_1) &= F_{l,2}(a_0,a_1)F_{k+1}+F_{l-1,2}(a_0,a_1)F_k\nonumber\\
 &= F_{k,2}(F_{l,2}(a_0,a_1),F_{l+1,2}(a_0,a_1)).\nonumber
 \end{align}
\end{proof}

In a similar way we can obtain the following result by using the corresponding results dervide so far.

\begin{proposition}\label{p20}
Let $a,b,r$ be three numbers. We have:
$$
C_{k+l,2}(a,b,r)=C_{l-1,2}(a,b,r)F_k+C_{l,2}(a,b,r)F_{k+1}+r(F_{k+2}-1),
\,\,\,\forall\,k\geq 0,\,\,\,\forall\,l\geq 1,
$$
or equivalently
$$
C_{k+l,2}(a,b,r)=C_{k+2,2}(C_{l-1,2}(a,b,r),C_{l,2}(a,b,r),r),
\,\,\,\forall\,k\geq 0,\,\,\,\forall\,l\geq 1.
$$
\end{proposition}
We now have the following more general results.

\begin{theorem}\label{t21}
Let $a_0,a_1$ be two numbers. We have
$$
F_{k,2}(a_0F_{l-1,2}(a_0,a_1)+a_1F_{l,2}(a_0,a_1),
a_0F_{l,2}(a_0,a_1)+a_1F_{l+1,2}(a_0,a_1))
$$
$$
=F_{l,2}(a_0,a_1)F_{k+1,2}(a_0,a_1)+F_{l-1,2}(a_0,a_1)F_{k,2}(a_0,a_1),
\,\,\,\forall\,k\geq 0,\,\,\,\forall\,l\geq 1
$$
\end{theorem}

The proof is an easy application of Proposition 5.9 and we shall omit it here.

\begin{theorem}\label{t22}
Let $a,b,r$ be three numbers. We have
$$
C_{k,2}(aC_{l-1,2}(a,b,r)+bC_{l,2}(a,b,r),(a+r)C_{l,2}(a,b,r)+b(C_{l+1,2}(a,b,r)-r),
r(C_{l+1,2}(a,b,r)-r))
$$
\be
\label{E2}
=C_{l,2}(a,b,r)C_{k+1,2}(a,b,r)+C_{l-1,2}(a,b,r)C_{k,2}(a,b,r),
\,\,\,\forall\,k\geq 0,\,\,\,\forall\,l\geq 1
\ee
\end{theorem}

Using Proposition 5.5 and the principle of mathematical induction the above result can be verified. We omit the details here.

\begin{remark}\label{rem23}
Using Proposition \ref{p8} and using Proposition \ref{p12}, we can notice that
\be
\label{E3}
C_{n,2}(a,b,0)=F_{n,2}(b-a,a),\,\,\,\forall\,n\geq 0
\ee
Indeed, we have ($n\geq 0$)
$$
F_{n,2}(b-a,a)=(b-a)F_{n+1}+(a-b+a)F_n=(b-a)F_{n+1}+(2a-b)F_n
$$
Using the definition of the Fibonacci sequence, we have for $n\geq 2$
\begin{align}
F_{n,2}(b-a,a) &= (b-a)(F_n+F_{n-1})+(2a-b)(F_{n-1}+F_{n-2})\nonumber\\
 &= (b-a)(2F_{n-1}+F_{n-2})+(2a-b)(F_{n-1}+F_{n-2})\nonumber\\
 &= (2(b-a)+2a-b)F_{n-1}+(b-a+2a-b)F_{n-2}=bF_{n-1}+aF_{n-2}\nonumber\\
 &= aF_{n-2}+bF_{n-1}=C_{n,2}(a,b,0).\nonumber
 \end{align}
 
Since $F_{0,2}(b-a,a)=C_{0,2}(a,b,0)=b-a$ and
$F_{1,2}(b-a,a)=C_{1,2}(a,b,0)=a$, the formula derived above for
$n\geq 2$ is also true for $n=0$ and for $n=1$.
\\[0.1in]
Taking $r=0$ in Theorem \ref{t22}, it can be shown
that Theorem \ref{t21} is a particular case of Theorem
\ref{t22}. Indeed, since ($l\geq 1$):
$$
aC_{l,2}(a,b,0)+bC_{l+1,2}(a,b,0)-(aC_{l-1,2}(a,b,0)+bC_{l,2}(a,b,0))
$$
$$
=a(C_{l,2}(a,b,0)-C_{l-1,2}(a,b,0))+b(C_{l+1,2}(a,b,0)-C_{l,2}(a,b,0))
$$
and so ($l\geq 2$):
$$
aC_{l,2}(a,b,0)+bC_{l+1,2}(a,b,0)-(aC_{l-1,2}(a,b,0)+bC_{l,2}(a,b,0))
$$
$$
=aC_{l-2,2}(a,b,0)+bC_{l-1,2}(a,b,0)
$$
using the relation (\ref{E3}), we have ($k\geq 0$ and $l\geq 2$):
$$
C_{k,2}(aC_{l-1,2}(a,b,0)+bC_{l,2}(a,b,0),aC_{l,2}(a,b,0)+bC_{l+1,2}(a,b,0),0)
$$ 
$$
=F_{k,2}(aC_{l-2,2}(a,b,0)+bC_{l-1,2}(a,b,0),aC_{l-1,2}(a,b,0)+bC_{l,2}(a,b,0))
$$
$$
=F_{k,2}((b-a)C_{l-1,2}(a,b,0)+a(C_{l-2,2}(a,b,0)+C_{l-1,2}(a,b,0)),
(b-a)C_{l,2}(a,b,0)$$
$$+a(C_{l-1,2}(a,b,0)+C_{l,2}(a,b,0))).
$$
So ($k\geq 0$ and $l\geq 1$):
$$
C_{k,2}(aC_{l-1,2}(a,b,0)+bC_{l,2}(a,b,0),aC_{l,2}(a,b,0)+bC_{l+1,2}(a,b,0),0)
$$
$$
=F_{k,2}((b-a)C_{l-1,2}(a,b,0)+aC_{l,2}(a,b,0),(b-a)C_{l,2}(a,b,0)+aC_{l+1,2}(a,b,0))
$$
$$
=F_{k,2}((b-a)F_{l-1,2}(b-a,a)+aF_{l,2}(b-a,a),(b-a)F_{l,2}(b-a,a)+aF_{l+1,2}(b-a,a)).
$$
Moreover, from Theorem \ref{t22}, we have ($k\geq 0$ and $l\geq 1$):
$$
C_{k,2}(aC_{l-1,2}(a,b,0)+bC_{l,2}(a,b,0),aC_{l,2}(a,b,0)+bC_{l+1,2}(a,b,0),0)
$$
$$
=C_{l,2}(a,b,0)C_{k+1,2}(a,b,0)+C_{l-1,2}(a,b,0)C_{k,2}(a,b,0)
$$
$$
=F_{l,2}(b-a,a)F_{k+1,2}(b-a,a)+F_{l-1,2}(b-a,a)F_{k,2}(b-a,a).
$$
Therefore ($k\geq 0$ and $l\geq 1$):
$$
F_{k,2}((b-a)F_{l-1,2}(b-a,a)+aF_{l,2}(b-a,a),(b-a)F_{l,2}(b-a,a)+aF_{l+1,2}(b-a,a))
$$
\be
\label{E4}
=F_{l,2}(b-a,a)F_{k+1,2}(b-a,a)+F_{l-1,2}(b-a,a)F_{k,2}(b-a,a)
\ee
which is equivalent to Theorem \ref{t21} when $a_0$ is replaced by
$b-a$ and when $a_1$ is replaced by $a$. Besides, taking $a=b=1$ in the
relation (\ref{E4}), using Theorem 1.27, since
$F_{n,2}(0,1)=F_n$, for all $n\geq 0$, we get ($l\geq 0$):
$$
F_{k,2}(F_l,F_{l+1})=F_{k+l},\,\,\,\forall\,k\geq 0
$$
\end{remark}

\begin{definition}\label{d29}
Let $a,b,r$ be three numbers, let $n\geq 0$ be a natural number and
let $l$ be a non-zero positive integer. The sequences
$(x_{n,l}(a,b,r))$, $(y_{n,l}(a,b,r))$ and $(z_{n,l}(a,b,r))$ are
defined by ($n\geq 0$ and $l\geq 1$):
\begin{align}
x_{n+1,l}(a,b,r) &= x_{n,l}(a,b,r)C_{l-1,2}(x_{n,l}(a,b,r),y_{n,l}(a,b,r),z_{n,l}(a,b,r))\nonumber\\
 & \quad +y_{n,l}(a,b,r)C_{l,2}(x_{n,l}(a,b,r),y_{n,l}(a,b,r),z_{n,l}(a,b,r))\nonumber
 \end{align}
 \begin{align}
 y_{n+1,l}(a,b,r) &= y_{n,l}(a,b,r)C_{l-1,2}(x_{n,l}(a,b,r),y_{n,l}(a,b,r),z_{n,l}(a,b,r))+(x_{n,l}(a,b,r)\nonumber\\
 & \quad +y_{n,l}(a,b,r)+z_{n,l}(a,b,r))C_{l,2}(x_{n,l}(a,b,r),y_{n,l}(a,b,r),z_{n,l}(a,b,r))\nonumber
 \end{align}
\begin{align}
r_{n+1,l}(a,b,r) &= z_{n,l}(a,b,r)(C_{l-1,2}(x_{n,l}(a,b,r),y_{n,l}(a,b,r),z_{n,l}(a,b,r))\nonumber\\
 & \quad +C_{l,2}(x_{n,l}(a,b,r),y_{n,l}(a,b,r),z_{n,l}(a,b,r)))
\nonumber
\end{align}
and for $l\geq 1$ 
$$
\begin{array}{l}
x_{0,l}(a,b,r)=aC_{l-1,2}(a,b,r)+bC_{l,2}(a,b,r)
\\
y_{0,l}(a,b,r)=bC_{l-1,2}(a,b,r)+(a+b+r)C_{l,2}(a,b,r)
\\
z_{0,l}(a,b,r)=r(C_{l-1,2}(a,b,r)+C_{l,2}(a,b,r))
\end{array}
$$
\end{definition}

\noindent
In the following, when there is no ambiguity and when it is possible,
we will abbreviate the 
notations used for terms of sequences $(x_{n,l}(a,b,r))$,
$(y_{n,l}(a,b,r))$ and $(z_{n,l}(a,b,r))$. More precisely, if $a,b,r$ don't
take particular values, then we will substitute $x_{n,l}$, $y_{n,l}$,
$z_{n,l}$ for $x_{n,l}(a,b,r)$, $y_{n,l}(a,b,r)$, $z_{n,l}(a,b,r)$
respectively. Thus, the recurrence
relations which define the sequences $(x_{n,l}(a,b,r))$,
$(y_{n,l}(a,b,r))$ and $(z_{n,l}(a,b,r))$ can be rewritten as 
($n\geq 0$ and $l\geq 1$):
$$
\begin{array}{l}
x_{n+1,l}=x_{n,l}C_{l-1,2}(x_{n,l},y_{n,l},z_{n,l})+y_{n,l}C_{l,2}(x_{n,l},y_{n,l},z_{n,l})
\\
y_{n+1,l}=y_{n,l}C_{l-1,2}(x_{n,l},y_{n,l},z_{n,l})+(x_{n,l}+y_{n,l}+z_{n,l})
C_{l,2}(x_{n,l},y_{n,l},z_{n,l})
\\
r_{n+1,l}=z_{n,l}(C_{l-1,2}(x_{n,l},y_{n,l},z_{n,l})+C_{l,2}(x_{n,l},y_{n,l},z_{n,l})).
\end{array}
$$

\begin{proposition}\label{p30}
Let $n\geq 0$ be a natural number and
let $l$ be a non-zero positive integer. We have
\begin{align}
C_{k,2}(x_{n+1,l},y_{n+1,l},z_{n+1,l})
 &= C_{l,2}(x_{n,l},y_{n,l},z_{n,l})C_{k+1,2}(x_{n,l},y_{n,l},z_{n,l})\nonumber\\
 & \quad +C_{l-1,2}(x_{n,l},y_{n,l},z_{n,l})C_{k,2}(x_{n,l},y_{n,l},z_{n,l}).\nonumber
\end{align}
\end{proposition}

\begin{proof}
This proposition is a direct
consequence of Definition \ref{d4}, Definition \ref{d29} and Theorem \ref{t22}.
\end{proof}

\begin{proposition}\label{p31}
Let $n\geq 0$ be a natural number and
let $l$ be a non-zero positive integer. We have ($n\geq 0$ and $l\geq 1$)
$$
(y_{n,l}-x_{n,l})(z_{n,l}y_{n+1,l}-y_{n,l}z_{n+1,l})
=(x_{n,l}+z_{n,l})(z_{n,l}x_{n+1,l}-x_{n,l}z_{n+1,l}).
$$
or equivalently ($n\geq 0$ and $l\geq 1$)
$$
z_{n,l}(x_{n,l}+z_{n,l})x_{n+1,l}+z_{n,l}(x_{n,l}-y_{n,l})y_{n+1,l}
-(x_{n,l}(x_{n,l}+y_{n,l}+z_{n,l})-y^2_{n,l})z_{n+1,l}=0.
$$
\end{proposition}

\begin{proof}
In the following, $n$ denotes a natural
number ($n\geq 0$) and $l$ denotes a non-zero positive integer ($l\geq 1$).
From Definition \ref{d29}, we have ($n\geq 0$ and $l\geq 1$)
$$
z_{n,l}x_{n+1,l}-x_{n,l}z_{n+1,l}=x_{n,l}z_{n,l}C_{l-1,2}(x_{n,l},y_{n,l},z_{n,l})
+y_{n,l}z_{n,l}C_{l,2}(x_{n,l},y_{n,l},z_{n,l})
$$
$$
-x_{n,l}z_{n,l}C_{l-1,2}(x_{n,l},y_{n,l},z_{n,l})
-x_{n,l}z_{n,l}C_{l,2}(x_{n,l},y_{n,l},z_{n,l}).
$$
So
\be
\label{E14}
z_{n,l}x_{n+1,l}-x_{n,l}z_{n+1,l}=z_{n,l}(y_{n,l}-x_{n,l})C_{l,2}(x_{n,l},y_{n,l},z_{n,l}),
\,\,\,\forall\,n\geq 0,\,\,\,\forall\,l\geq 1
\ee
Moreover, we have ($n\geq 0$ and $l\geq 1$):
$$
z_{n,l}y_{n+1,l}-y_{n,l}z_{n+1,l}
=y_{n,l}z_{n,l}C_{l-1,2}(x_{n,l},y_{n,l},z_{n,l})
+z_{n,l}(x_{n,l}+y_{n,l}+z_{n,l})C_{l,2}(x_{n,l},y_{n,l},z_{n,l})
$$
$$
-y_{n,l}z_{n,l}C_{l-1,2}(x_{n,l},y_{n,l},z_{n,l})
-y_{n,l}z_{n,l}C_{l,2}(x_{n,l},y_{n,l},z_{n,l}).
$$
So
\be
\label{E15}
z_{n,l}y_{n+1,l}-y_{n,l}z_{n+1,l}=z_{n,l}(x_{n,l}+z_{n,l})C_{l,2}(x_{n,l},y_{n,l},z_{n,l})
\,\,\,\forall\,n\geq 0,\,\,\,\forall\,l\geq 1
\ee
Taking $(x_{n,l}+z_{n,l})\,(\ref{E14})-(y_{n,l}-x_{n,l})\,(\ref{E15})$
side by side, we get
$$
(x_{n,l}+z_{n,l})(z_{n,l}x_{n+1,l}-x_{n,l}z_{n+1,l})
-(y_{n,l}-x_{n,l})(z_{n,l}y_{n+1,l}-y_{n,l}z_{n+1,l})=0
$$
and so
$$
(x_{n,l}+z_{n,l})(z_{n,l}x_{n+1,l}-x_{n,l}z_{n+1,l})
=(y_{n,l}-x_{n,l})(z_{n,l}y_{n+1,l}-y_{n,l}z_{n+1,l}).
$$
It proves the first part of Proposition \ref{p31}. The second part of
Proposition \ref{p31} follows from its first part. Indeed, from the
first part of Proposition \ref{p31}, we have ($n\geq 0$ and $l\geq 1$)
$$
(x_{n,l}+z_{n,l})z_{n,l}x_{n+1,l}-(x_{n,l}+z_{n,l})x_{n,l}z_{n+1,l}
=(y_{n,l}-x_{n,l})z_{n,l}y_{n+1,l}-(y_{n,l}-x_{n,l})y_{n,l}z_{n+1,l}
$$
$$
z_{n,l}(x_{n,l}+z_{n,l})x_{n+1,l}+z_{n,l}(x_{n,l}-y_{n,l})y_{n+1,l}
-((x_{n,l}+z_{n,l})x_{n,l}+(x_{n,l}-y_{n,l})y_{n,l})z_{n+1,l}=0
$$
$$
z_{n,l}(x_{n,l}+z_{n,l})x_{n+1,l}+z_{n,l}(x_{n,l}-y_{n,l})y_{n+1,l}
-(x^2_{n,l}+z_{n,l}x_{n,l}+x_{n,l}y_{n,l}-y^2_{n,l})z_{n+1,l}=0
$$
$$
z_{n,l}(x_{n,l}+z_{n,l})x_{n+1,l}+z_{n,l}(x_{n,l}-y_{n,l})y_{n+1,l}
-(x_{n,l}(x_{n,l}+z_{n,l}+y_{n,l})-y^2_{n,l})z_{n+1,l}=0
$$
It proves the second part of Proposition \ref{p31}.
\end{proof}

\begin{definition}\label{d32}
Let $\mathbb{K}$ be a field. Let $l$ be a non-zero positive integer
($l\geq 1$). The function $F_l$ is defined on 
$\mathbb{K}^3$ by ($l\geq 1$ and $(x,y,z)\in\mathbb{K}^3$)
\begin{align}
F_l(x,y,z) &= (xC_{l-1,2}(x,y,z)+yC_{l,2}(x,y,z),yC_{l-1,2}(x,y,z)\nonumber\\
 & \quad +(x+y+z)C_{l,2}(x,y,z),
z(C_{l-1,2}(x,y,z)+C_{l,2}(x,y,z))).\nonumber
\end{align}
\end{definition}

\begin{remark}\label{r33}
From Definition \ref{d29} and from Definition \ref{d32}, we have
($n\geq 0$ and $l\geq 1$)
$$
F_l(x_{n,l},y_{n,l},z_{n,l})=(x_{n+1,l},y_{n+1,l},z_{n+1,l})
$$
So, from Proposition \ref{p30}, we have
\begin{align}
C_{k,2}(F_l(x_{n,l},y_{n,l},z_{n,l})) & =C_{l,2}(x_{n,l},y_{n,l},z_{n,l})C_{k+1,2}(x_{n,l},y_{n,l},z_{n,l})\nonumber\\
 & \quad +C_{l-1,2}(x_{n,l},y_{n,l},z_{n,l})C_{k,2}(x_{n,l},y_{n,l},z_{n,l}).\nonumber
\end{align}
\end{remark}

\begin{proposition}\label{pp34}
Let $n\geq 0$ be a natural number and
let $l$ be a non-zero positive integer. We have ($n\geq 0$ and $l\geq 1$)
$$
x_{n,l}(1/2,1/2,-1/2)=y_{n,l}(1/2,1/2,-1/2)=1/2
$$
$$
z_{n,l}(1/2,1/2,-1/2)=-1/2.
$$
In other words, $(1/2,1/2,-1/2)$ is a fixed point of the function
$F_l$ for all $l\geq 1$.
\end{proposition}

\begin{proof}
Let $n\geq 0$ be a natural number and
let $l$ be a non-zero positive integer. Let us prove
Proposition \ref{pp34} by induction on the integer $n\geq 0$ for all
$l\geq 1$. Using Definition \ref{d4}, we have
$$
C_{0,2}(1/2,1/2,-1/2)=\frac{1}{2}-\frac{1}{2}-\left(-\frac{1}{2}\right)
=\frac{1}{2}.
$$

Moreover, from Proposition \ref{p8}, 
using the definition of the Fibonacci sequence, we have ($n\geq 2$)
$$
C_{n,2}(1/2,1/2,-1/2)=\frac{F_{n-2}}{2}+\frac{F_{n-1}}{2}-\frac{1}{2}(F_n-1)
=\frac{F_{n-2}+F_{n-1}-F_n}{2}+\frac{1}{2}=\frac{1}{2}.
$$
So
\be
\label{E16}
C_{n,2}(1/2,1/2,-1/2)=\frac{1}{2},\,\,\,\forall\,n\geq 0.
\ee
Using Definition \ref{d29} and using Equation (\ref{E16}), it gives ($l\geq 1$)
$$
\begin{array}{l}
x_{0,l}(1/2,1/2,-1/2)=\frac{1}{2}\times\frac{1}{2}+\frac{1}{2}\times\frac{1}{2}
=\frac{1}{2},
\\
y_{0,l}(1/2,1/2,-1/2)=\frac{1}{2}\times\frac{1}{2}
+\left(\frac{1}{2}+\frac{1}{2}-\frac{1}{2}\right)\frac{1}{2}
=\frac{1}{2},
\\
z_{0,l}(1/2,1/2,-1/2)=-\frac{1}{2}\left(\frac{1}{2}+\frac{1}{2}\right)
=-\frac{1}{2}.
\end{array}
$$
Hence, we verify that Proposition \ref{pp34} is true for $n=0$ and
for all $l\geq 1$. Let us assume that Proposition \ref{pp34} is true
up to an integer $n\geq 0$  and for all $l\geq 1$. Using again Definition
\ref{d29} and using Equation (\ref{E16}), we have ($n\geq 0$ and $l\geq 1$)
$$
\begin{array}{l}
x_{n+1,l}(1/2,1/2,-1/2)=\frac{1}{2}\times\frac{1}{2}+\frac{1}{2}\times\frac{1}{2}
=\frac{1}{2},
\\
y_{n+1,l}(1/2,1/2,-1/2)=\frac{1}{2}\times\frac{1}{2}
+\left(\frac{1}{2}+\frac{1}{2}-\frac{1}{2}\right)\frac{1}{2},
\\
z_{n+1,l}(1/2,1/2,-1/2)=-\frac{1}{2}\left(\frac{1}{2}+\frac{1}{2}\right)
=-\frac{1}{2}.
\end{array}
$$
Thus, if Proposition \ref{pp34} is true up to an integer $n\geq 0$,
then Proposition \ref{pp34} is true for $n+1$. Thus we have proved Proposition
\ref{p34} by induction on the integer $n\geq 0$ for all
$l\geq 1$. Using Remark \ref{r33}, we get ($l\geq 1$)
$$
F_l(1/2,1/2,-1/2)=(1/2,1/2,-1/2).
$$
Therefore, $(1/2,1/2,-1/2)$ is a fixed point of the function $F_l$ for
all $l\geq 1$.
\end{proof}

The results presented in this section can be further generalized to other class of sequences. For one such aspect, the reader can refer to \cite{hmsa}.

\section{Some results on Generalized Fibonacci polynomial sequences}

In this section, we introduce some generalized Fibonacci polynomial
sequences and we give some properties about these polynomial
sequences.

\begin{definition}\label{d6.1}

Let $k$ be an integer which is greater than $2$ and let
$a_0,\ldots, a_{k-1}$ be $k$ numbers.

The polynomial sequence $(F^{(1)}_{n,k}(a_0,\ldots,a_{k-1};x))$ in one
indeterminate $x$ is defined by ($k\geq 2$): 
$$
F^{(1)}_{n,k}(a_0,\ldots,a_{k-1};x)=F^{(1)}_{n-1,k}(a_0,\ldots,a_{k-1};x)
+xF^{(1)}_{n-k,k}(a_0,\ldots,a_{k-1};x),\,\,\,\forall\,n\geq k
$$
with:
$$
F^{(1)}_{i,k}(a_0,\ldots,a_{k-1};x)=a_i,\,\,\,\forall\,i\in\{0,\ldots,k-1\}
$$
The $k$-Fibonacci numbers sequence $(F_{n,k}(a_0,\ldots,a_{k-1}))$
with initial conditions 
$a_0,\ldots,a_{k-1}$ are obtained from this polynomial sequence
by substituting $x$ by $1$ in the sequence
$(F^{(1)}_{n,k}(a_0,\ldots,a_{k-1};x))$. This polynomial sequence is called
the $k$-Fibonacci polynomial sequence of the first kind with initial conditions
$a_0,\ldots,a_{k-1}$. 
\end{definition}

\textit{Case $k=2$}

Table of the first polynomial terms of sequence $(F^{(1)}_{n,2}(0,1;x))$
$$
\begin{array}{l}
F^{(1)}_{0,2}(0,1;x)=0\\
F^{(1)}_{1,2}(0,1;x)=1\\
F^{(1)}_{2,2}(0,1;x)=1\\
F^{(1)}_{3,2}(0,1;x)=1+x\\
F^{(1)}_{4,2}(0,1;x)=1+2x\\
F^{(1)}_{5,2}(0,1;x)=1+3x+x^2\\
F^{(1)}_{6,2}(0,1;x)=1+4x+3x^2\\
\end{array}
$$
Table of the first polynomial terms of sequence $(F^{(1)}_{n,2}(1,0;x))$
$$
\begin{array}{l}
F^{(1)}_{0,2}(1,0;x)=1\\
F^{(1)}_{1,2}(1,0;x)=0\\
F^{(1)}_{2,2}(1,0;x)=x\\
F^{(1)}_{3,2}(1,0;x)=x\\
F^{(1)}_{4,2}(1,0;x)=x(x+1)\\
F^{(1)}_{5,2}(1,0;x)=x(2x+1)\\
F^{(1)}_{6,2}(1,0;x)=x(x^2+3x+1)\\
\end{array}
$$

\begin{property}\label{p6.2}
Let $n$ be a non-zero positive integer. We have:
$$
F^{(1)}_{n,2}(0,1;x)={\displaystyle\sum^{\lfloor\frac{n-1}{2}\rfloor}_{k=0}}
\binom{n-k-1}{k}x^k
$$ 
$$
F^{(1)}_{n,2}(1,0;x)=xF^{(1)}_{n-1,2}(0,1;x)
$$
\end{property}

\begin{proof}
Let prove the first part of Property \ref{p6.2} by induction on the
integer $n>0$. We have
$$
F^{(1)}_{1,2}(0,1;x)=1={n-1\choose 0}x^0
={\displaystyle\sum^0_{k=0}}{n-k-1\choose k}x^k
$$ 
Thus, we verify that the first part of Property \ref{p6.2} is
true for $n=1$.
Let assume that Property \ref{p6.2} is true up to an integer
$n>0$. Using Definition \ref{d6.1}, we have
$$
F^{(1)}_{n+1,2}(0,1;x)=F^{(1)}_{n,2}(0,1;x)+xF^{(1)}_{n-1,2}(0,1;x)
$$
Using the assumption, it gives:
$$
F^{(1)}_{n+1,2}(0,1;x)={\displaystyle\sum^{\lfloor\frac{n-1}{2}\rfloor}_{k=0}}
{n-k-1\choose k}x^k+{\displaystyle\sum^{\lfloor\frac{n-2}{2}\rfloor}_{k=0}}
{n-k-2\choose k}x^{k+1}
$$
Taking the change of label $k\rightarrow m=k+1$ in the second sum of
the right hand side of the previous equation, after renaming $m$ by
$k$, we have (reccall that $\lfloor x+1\rfloor=\lfloor x\rfloor+1$,
$\forall x\in\mathbb{R}$)
$$
F^{(1)}_{n+1,2}(0,1;x)={\displaystyle\sum^{\lfloor\frac{n-1}{2}\rfloor}_{k=0}}
{n-k-1\choose k}x^k+{\displaystyle\sum^{\lfloor\frac{n}{2}\rfloor}_{k=1}}
{n-k-1\choose k-1}x^k
$$
Or
$$
\lfloor\frac{n}{2}\rfloor=\left\{\begin{array}{ccc}
\lfloor\frac{n-1}{2}\rfloor+1 & if & n\equiv 0\pmod 2
\\\\
\lfloor\frac{n-1}{2}\rfloor & if & n\equiv 1\pmod 2
\end{array}\right.
$$
If $n$ is odd, then
$\lfloor\frac{n}{2}\rfloor=\lfloor\frac{n-1}{2}\rfloor$ and we have
$$
F^{(1)}_{n+1,2}(0,1;x)=1+{\displaystyle\sum^{\lfloor\frac{n}{2}\rfloor}_{k=1}}
{n-k-1\choose k}x^k+{\displaystyle\sum^{\lfloor\frac{n}{2}\rfloor}_{k=1}}
{n-k-1\choose k-1}x^k
$$
Rearranging the different terms of this equation, it comes that ($n$ odd)
$$
F^{(1)}_{n+1,2}(0,1;x)=1+{\displaystyle\sum^{\lfloor\frac{n}{2}\rfloor}_{k=1}}
\left\{{n-k-1\choose k}+{n-k-1\choose k-1}\right\}x^k
$$
Using the combinatorial identity
$$
{n-k-1\choose k}+{n-k-1\choose k-1}={n-k\choose k}
$$
if $n$ is odd, then we have
$$
F^{(1)}_{n+1,2}(0,1;x)=1+{\displaystyle\sum^{\lfloor\frac{n}{2}\rfloor}_{k=1}}
{n-k\choose k}x^k
$$
$$
F^{(1)}_{n+1,2}(0,1;x)={\displaystyle\sum^{\lfloor\frac{n}{2}\rfloor}_{k=0}}
{n-k\choose k}x^k={\displaystyle\sum^{\lfloor\frac{n+1-1}{2}\rfloor}_{k=0}}
{n+1-k-1\choose k}x^k
$$
If $n$ is even, then
$\lfloor\frac{n}{2}\rfloor=\lfloor\frac{n-1}{2}\rfloor+1$ and we have
$$
F^{(1)}_{n+1,2}(0,1;x)=1+{\displaystyle\sum^{\lfloor\frac{n-1}{2}\rfloor}_{k=1}}
\left\{{n-k-1\choose k}+{n-k-1\choose k-1}\right\}x^k
+{n-\lfloor\frac{n}{2}\rfloor-1\choose \lfloor\frac{n}{2}\rfloor-1}
x^{\lfloor\frac{n}{2}\rfloor}
$$
Using again the combinatorial identity
$$
{n-k-1\choose k}+{n-k-1\choose k-1}={n-k\choose k}
$$
it gives
$$
F^{(1)}_{n+1,2}(0,1;x)=1+{\displaystyle\sum^{\lfloor\frac{n-1}{2}\rfloor}_{k=1}}
{n-k\choose k}x^k
+{n-\lfloor\frac{n}{2}\rfloor-1\choose \lfloor\frac{n}{2}\rfloor-1}
x^{\lfloor\frac{n}{2}\rfloor}
$$
Using the definition of binomial coefficients, it can be
shown that ($k>0$)
$$
{n-k\choose k}=\frac{n-k}{k}{n-k-1\choose k-1}
$$
Or
$$
n=\left\{\begin{array}{ccc}
2\lfloor\frac{n}{2}\rfloor & if & n\equiv 0\pmod 2
\\\\
2\lfloor\frac{n}{2}\rfloor+1 & if & n\equiv 1\pmod 2
\end{array}\right.
$$
In particular,  when $n$ is even, $n=2\lfloor\frac{n}{2}\rfloor$ and
so $n-\lfloor\frac{n}{2}\rfloor=\lfloor\frac{n}{2}\rfloor$. Accordingly, we have
$$
{n-\lfloor\frac{n}{2}\rfloor\choose \lfloor\frac{n}{2}\rfloor}
={n-\lfloor\frac{n}{2}\rfloor-1\choose \lfloor\frac{n}{2}\rfloor-1}
$$
If $n$ is even, then we have
$$
F^{(1)}_{n+1,2}(0,1;x)=1+{\displaystyle\sum^{\lfloor\frac{n-1}{2}\rfloor}_{k=1}}
{n-k\choose k}x^k
+{n-\lfloor\frac{n}{2}\rfloor\choose \lfloor\frac{n}{2}\rfloor}
x^{\lfloor\frac{n}{2}\rfloor}
$$
$$
F^{(1)}_{n+1,2}(0,1;x)={\displaystyle\sum^{\lfloor\frac{n}{2}\rfloor}_{k=0}}
{n-k\choose k}x^k={\displaystyle\sum^{\lfloor\frac{n+1-1}{2}\rfloor}_{k=0}}
{n+1-k-1\choose k}x^k
$$
So, the first part of Property \ref{p6.2} is proved by induction
on the integer $n>0$.

Afterwards, let prove the second part of Property \ref{p6.2} by
induction on the integer $n>0$. We have
$$
F^{(1)}_{1,2}(1,0;x)=0=xF^{(1)}_{0,2}(0,1;x)
$$ 
Thus, we verify that the second part of Property \ref{p6.2} is
true for $n=1$.
Let assume that the second part of Property \ref{p6.2} is true up
to an integer $n>0$. Using Definition \ref{d6.1}, we have
$$
F^{(1)}_{n+1,2}(1,0;x)=F^{(1)}_{n,2}(1,0;x)+xF^{(1)}_{n-1,2}(1,0;x)
$$
Using the assumption, it gives:
$$
F^{(1)}_{n+1,2}(1,0;x)=x(F^{(1)}_{n,2}(0,1;x)+xF^{(1)}_{n-1,2}(0,1;x))
=xF^{(1)}_{n+1,2}(0,1;x)
$$
So, the second part of Property \ref{p6.2} is proved by induction
on the integer $n>0$.
\end{proof}

\begin{property}\label{p6.3}
The generating function of the polynomials $F^{(1)}_{n,2}(0,1;x)$ is
given by
$$
{\mathcal F}^{(1)}_2(0,1;x,y)={\displaystyle\sum^{+\infty}_{n=0}}
F^{(1)}_{n,2}(0,1;x)y^n=\frac{y}{1-y-xy^2}
$$
where
$$
y\neq\left\{\begin{array}{ccc}
1 & if & x=0
\\
\frac{-1\pm\sqrt{1+4x}}{2x} & if & x\neq 0
\end{array}\right.
$$
\end{property}

\begin{proof}
The generating function of the polynomials $F^{(1)}_{n,2}(0,1;x)$ is
defined by
$$
{\mathcal F}^{(1)}_2(0,1;x,y)={\displaystyle\sum^{+\infty}_{n=0}}
F^{(1)}_{n,2}(0,1;x)y^n
$$
Since $F^{(1)}_{0,2}(0,1;x)=0$ and since $F^{(1)}_{1,2}(0,1;x)=1$, we have
$$
{\mathcal F}^{(1)}_2(0,1;x,y)={\displaystyle\sum^{+\infty}_{n=1}}
F^{(1)}_{n,2}(0,1;x)y^n={\displaystyle\sum^{+\infty}_{n=0}}
F^{(1)}_{n+1,2}(0,1;x)y^{n+1}=y+{\displaystyle\sum^{+\infty}_{n=1}}
F^{(1)}_{n+1,2}(0,1;x)y^{n+1}
$$
where in the sum over $n$, we performed the change of label
$n\rightarrow m=n-1$ and we renamed $m$ by $n$. 
Using Definition \ref{d6.1}, it gives
$$
{\mathcal F}^{(1)}_2(0,1;x,y)=y+{\displaystyle\sum^{+\infty}_{n=1}}
(F^{(1)}_{n,2}(0,1;x)+xF^{(1)}_{n-1,2}(0,1;x))y^{n+1}
$$
Expanding the sum over $n$ of the right hand side of the previous
equation, it comes that 
$$
{\mathcal F}^{(1)}_2(0,1;x,y)=y+{\displaystyle\sum^{+\infty}_{n=1}}
F^{(1)}_{n,2}(0,1;x)y^{n+1}+x{\displaystyle\sum^{+\infty}_{n=1}}F^{(1)}_{n-1,2}(0,1;x)
y^{n+1}
$$
Or, performing again the change of label $n\rightarrow m=n-1$ in the second
sum over $n$ of the right hand side of the previous equation, after renaming
$m$ by $n$, we have
$$
{\mathcal F}^{(1)}_2(0,1;x,y)=y+{\displaystyle\sum^{+\infty}_{n=1}}
F^{(1)}_{n,2}(0,1;x)y^{n+1}+x{\displaystyle\sum^{+\infty}_{n=1}}F^{(1)}_{n,2}(0,1;x)
y^{n+2}
$$
$$
{\mathcal F}^{(1)}_2(0,1;x,y)=y+y{\displaystyle\sum^{+\infty}_{n=1}}
F^{(1)}_{n,2}(0,1;x)y^n+xy^2{\displaystyle\sum^{+\infty}_{n=1}}F^{(1)}_{n,2}(0,1;x)
y^n
$$
Using the definition of the generating function ${\mathcal F}^{(1)}_2(0,1;x,y)$, 
we have: 
$$
{\mathcal F}^{(1)}_2(0,1;x,y)=y+y{\mathcal F}^{(1)}_2(0,1;x,y)
+xy^2{\mathcal F}^{(1)}_2(0,1;x,y)
$$
Therefore
$$
{\mathcal F}^{(1)}_2(0,1;x,y)=\frac{y}{1-y-xy^2}
$$
where
$$
y\neq\left\{\begin{array}{ccc}
1 & if & x=0
\\
\frac{-1\pm\sqrt{1+4x}}{2x} & if & x\neq 0
\end{array}\right.
$$
\end{proof}

\begin{property}\label{p6.4}
Let $a_0,a_1$ be two integers. We have:
$$
F^{(1)}_{n,2}(a_0,a_1;x)=a_0F^{(1)}_{n,2}(1,0;x)+a_1F^{(1)}_{n,2}(0,1;x),
\,\,\,\forall\,n\geq 0
$$
$$
F^{(1)}_{n,2}(a_0,a_1;x)=a_0xF^{(1)}_{n-1,2}(0,1;x)+a_1F^{(1)}_{n,2}(0,1;x),
\,\,\,\forall\,n\geq 1
$$
\end{property}

\begin{proof}
Let prove the first part of Property \ref{p6.4} by induction on the
integer $n\geq 0$. The second part of Property \ref{p6.4} follows from
Property \ref{p6.2}.
Since $F^{(1)}_{0,2}(0,1;x)=0$ and since $F^{(1)}_{0,2}(1,0;x)=1$, we have
$$
F^{(1)}_{0,2}(a_0,a_1;x)=a_0=a_0F^{(1)}_{0,2}(1,0;x)+a_1F^{(1)}_{0,2}(0,1;x)
$$
Thus, we verify that the first part of Property \ref{p6.4} is true for $n=0$.
Let assume that the first part of Property \ref{p6.4} is true up to an
integer $n\geq 0$. Using Definition \ref{d6.1}, we have ($n>0$)
$$
F^{(1)}_{n+1,2}(a_0,a_1;x)=F^{(1)}_{n,2}(a_0,a_1;x)+xF^{(1)}_{n-1,2}(a_0,a_1;x)
$$
Using the assumption, we have ($n>0$)
$$
F^{(1)}_{n+1,2}(a_0,a_1;x)=a_0F^{(1)}_{n,2}(1,0;x)+a_1F^{(1)}_{n,2}(0,1;x)
+x(a_0F^{(1)}_{n-1,2}(1,0;x)+a_1F^{(1)}_{n-1,2}(0,1;x))
$$
Rearranging the different terms of the right hand side of the previous
equation, it gives ($n>0$)
$$
F^{(1)}_{n+1,2}(a_0,a_1;x)=a_0(F^{(1)}_{n,2}(1,0;x)+xF^{(1)}_{n-1,2}(1,0;x))
+a_1(F^{(1)}_{n,2}(0,1;x)+xF^{(1)}_{n-1,2}(0,1;x))
$$
Using Definition \ref{d6.1}, we obtain ($n\geq 0$)
$$
F^{(1)}_{n+1,2}(a_0,a_1;x)=a_0F^{(1)}_{n+1,2}(1,0;x)+a_1F^{(1)}_{n+1,2}(0,1;x)
$$
So, the first part of Property \ref{p6.4} is proved by induction on
the integer $n$.
\end{proof}

\begin{remark}
In particular, if $a_0=a_1=1$, then using the recurrence relation of
the sequence $(F^{(1)}_{n,2}(0,1;x))$ (see Definition \ref{d6.1}), we obtain:
$$
F^{(1)}_{n,2}(1,1;x)=F^{(1)}_{n+1,2}(0,1;x),\,\,\,\forall n\geq 0
$$
\end{remark}

\begin{property}\label{p6.6}

Let $a_0$ and $a_1$ be two numbers and let $n$ be a positive
integer. We have ($n\geq 0$) 
$$
F^{(1)}_{n,2}(a_0,a_1;x)=\left\{\begin{array}{ccc}
\frac{a_0}{2^n}+\frac{n\left(a_1-\frac{a_0}{2}\right)}{2^{n-1}} & if &
x=-\frac{1}{4} 
\\\\
\left(\frac{a_0(\varphi(x)-1)+a_1}{2\varphi(x)-1}\right)\varphi(x)^n
+\left(\frac{a_0\varphi(x)-a_1}{2\varphi(x)-1}\right)(1-\varphi(x))^n
& if & x\neq-\frac{1}{4} 
\end{array}\right.
$$
In particular, we have
$$
F^{(1)}_{n,2}(0,1;x)=\left\{\begin{array}{ccc}
\frac{n}{2^{n-1}} & if & x=-\frac{1}{4}\\\\
\frac{\varphi(x)^n-(1-\varphi(x))^n}{2\varphi(x)-1} & if & x\neq-\frac{1}{4}
\end{array}\right.
$$
$$
F^{(1)}_{n,2}(1,0;x)=\left\{\begin{array}{ccc}
\frac{1-n}{2^n} & if & x=-\frac{1}{4}\\\\
x\left(\frac{\varphi(x)^{n-1}-(1-\varphi(x))^{n-1}}
{2\varphi(x)-1}\right) & if & x\neq-\frac{1}{4}
\end{array}\right.
$$
where
$$
\varphi(x)=\frac{1+\sqrt{1+4x}}{2}
$$
which verify
$$
\varphi^2(x)=\varphi(x)+x
$$
or
$$
\varphi(x)(\varphi(x)-1)=x
$$
\end{property}

\begin{proof}
Property \ref{p6.6} can be proved easily by induction or Property can be proved
in the same way as Property \ref{p11}.
\end{proof}

\begin{theorem}\label{t6.7}
Let $a_0$ and $a_1$ be two numbers and let $n$ and $m$ be two positive
integers. If $x\neq-\frac{1}{4}$, then we have
$$
F^{(1)}_{n,2}(a_0,a_1;x)F^{(1)}_{m+1,2}(a_0,a_1;x)
+xF^{(1)}_{n-1,2}(a_0,a_1;x)F^{(1)}_{m,2}(a_0,a_1;x)
$$
$$
=(a_0\varphi(x)-a_1)^2F^{(1)}_{m+n,2}(0,1;x)+a_0(2a_1-a_0)\varphi(x)^{m+n}
$$
Otherwise, we have
$$
F^{(1)}_{n,2}(a_0,a_1;-1/4)F^{(1)}_{m+1,2}(a_0,a_1;-1/4)
+xF^{(1)}_{n-1,2}(a_0,a_1;-1/4)F^{(1)}_{m,2}(a_0,a_1;-1/4)
$$
$$
=\frac{a^2_0(m+n-2)}{2^{m+n+1}}-\frac{a_0a_1(m+n-1)}{2^{m+n-1}}
+\frac{(m+n)a^2_1}{2^{m+n-1}}
$$
In particular, whatever $x$ is, we have
$$
F^{(1)}_{n,2}(0,1;x)F^{(1)}_{m+1,2}(0,1;x)
+xF^{(1)}_{n-1,2}(0,1;x)F^{(1)}_{m,2}(0,1;x)=F^{(1)}_{m+n,2}(0,1;x)
$$
\end{theorem}

\begin{proof}
Theorem \ref{t6.7} stems from Property \ref{p6.6}.
\end{proof}

\textit{Case $k=3$}

Table of the first polynomial terms of sequence $(F^{(1)}_{n,3}(0,0,1;x))$
$$
\begin{array}{l}
F^{(1)}_{0,3}(0,0,1;x)=0\\
F^{(1)}_{1,3}(0,0,1;x)=0\\
F^{(1)}_{2,3}(0,0,1;x)=1\\
F^{(1)}_{3,3}(0,0,1;x)=1\\
F^{(1)}_{4,3}(0,0,1;x)=1\\
F^{(1)}_{5,3}(0,0,1;x)=1+x\\
F^{(1)}_{6,3}(0,0,1;x)=1+2x\\
F^{(1)}_{7,3}(0,0,1;x)=1+3x\\
F^{(1)}_{8,3}(0,0,1;x)=1+4x+x^2
\end{array}
$$ 

\begin{property}\label{p6.8}
Let $n$ be a non-zero positive integer. For $n\geq 2$, we have
($n\geq 2$)
$$
F^{(1)}_{n,3}(0,0,1;x)={\displaystyle\sum^{\lfloor\frac{n-2}{3}\rfloor}_{k=0}}
{n-2k-2\choose k}x^k
$$
and ($n\geq 2$)
$$
F^{(1)}_{n,3}(0,1,0;x)=xF^{(1)}_{n-2,3}(0,0,1;x)
$$
Moreover, for $n\geq 1$, we have
$$
F^{(1)}_{n,3}(1,0,0;x)=xF^{(1)}_{n-1,3}(0,0,1;x)
$$
\end{property}

\begin{proof}
Property \ref{p6.8} can be proved in the same way as Property \ref{p6.2}.
\end{proof}

\begin{property}\label{p6.9}
The generating function of the polynomials $F^{(1)}_{n,3}(0,0,1;x)$ is
given by
$$
{\mathcal F}^{(1)}_3(0,0,1;x,y)={\displaystyle\sum^{+\infty}_{n=0}}
F^{(1)}_{n,3}(0,0,1;x)y^n=\frac{y^2}{1-y-xy^3}
$$
where
$$
1-y-xy^3\neq 0
$$
\end{property}

\begin{proof}
Property \ref{p6.9} can be proved in the same way as Property \ref{p6.3}.
\end{proof}

\begin{property}\label{p6.10}
Let $a_0,a_1,a_2$ be three integers. We have
$$
F^{(1)}_{n,3}(a_0,a_1,a_2;x)=a_0F^{(1)}_{n,3}(1,0,0;x)+a_1F^{(1)}_{n,3}(0,1,0;x)
+a_2F^{(1)}_{n,3}(0,0,1;x),
\,\,\,\forall\,n\geq 0
$$
\end{property}

\begin{proof}
Property \ref{p6.10} can be proved in the same way as Property \ref{p6.4}.
\end{proof}

\begin{theorem}\label{t6.11}
Let $m$ be an integer which is greater than $2$ and let $n$ be a non-zero
positive integer. For $n\geq m-1$, we have
$$
F^{(1)}_{n,m}(0,\ldots,0,1;x)={\displaystyle\sum^{\lfloor\frac{n-m+1}{m}\rfloor}_{k=0}}
{n-(m-1)(k+1)\choose k}x^k
$$
Moreover, for $n\geq 1$, we have
$$
F^{(1)}_{n,m}(1,0,\ldots,0)=xF^{(1)}_{n-1,m}(0,\ldots,0,1;x)
$$
and for $n\geq i$ with $i\in\{2,\ldots,m-1\}$ when $m>2$, we have
$$
F^{(1)}_{n,m}(0,\ldots,0_{i-2},1_{i-1},0_i,\ldots,0;x)=xF^{(1)}_{n-i,m}(0,\ldots,0,1;x)
$$
where $0_l$ means $a_l=0$ with $l\in\{i-2,i\}$ and $1_{i-1}$ means $a_{i-1}=1$ in 
$$
F^{(1)}_{n,m}(a_0,\ldots,a_{i-2},a_{i-1},a_i,\ldots,a_{m-1};x)
$$ 
\end{theorem}

\begin{proof}
Theorem \ref{t6.11} can be proved in the same way as Property \ref{p6.2}.
\end{proof}

\begin{theorem}\label{t6.12}
Let $m$ be an integer which is greater than $2$. 
The generating function of the polynomials $F^{(1)}_{n,m}(0,\ldots,0,1;x)$ is
given by
$$
{\mathcal F}^{(1)}_m(0,\ldots,0,1;x,y)={\displaystyle\sum^{+\infty}_{n=0}}
F^{(1)}_{n,m}(0,\ldots,0,1;x)y^n=\frac{y^{m-1}}{1-y-xy^m}
$$
where
$$
1-y-xy^m\neq 0
$$
\end{theorem}

\begin{proof}
Theorem \ref{t6.12} can be proved in the same way as Property \ref{p6.3}.
\end{proof}

\begin{theorem}\label{t6.13}
Let $m$ be an integer which is greater than $2$ and let
$a_0,a_1,\ldots,a_{m-1}$ be $m$ integers. We have 
$$
F^{(1)}_{n,m}(a_0,a_1,\ldots,a_{m-1};x)=a_0F^{(1)}_{n,m}(1,0,\ldots,0;x)
+a_1F^{(1)}_{n,m}(0,1,0,\ldots,0;x)
$$
$$
+\ldots+a_{m-1}F^{(1)}_{n,m}(0,\ldots,0,1;x),
\,\,\,\forall\,n\geq 0
$$
\end{theorem}

\begin{proof}
Theorem \ref{t6.13} can be proved in the same way as Property \ref{p6.4}.
\end{proof}

\begin{definition}\label{d6.14}

Let $k$ be an integer which is greater than $2$ and let
$a_0,\ldots, a_{k-1}$ be $k$ numbers.

The polynomial sequence $(F^{(2)}_{n,k}(a_0,\ldots,a_{k-1};x))$ in one
indeterminate $x$ is defined by ($k\geq 2$)
$$
F^{(2)}_{n,k}(a_0,\ldots,a_{k-1};x)=xF^{(2)}_{n-1,k}(a_0,\ldots,a_{k-1};x)
+F^{(2)}_{n-k,k}(a_0,\ldots,a_{k-1};x),\,\,\,\forall\,n\geq k
$$
with
$$
F^{(2)}_{i,k}(a_0,\ldots,a_{k-1};x)=a_i,\,\,\,\forall\,i\in\{0,\ldots,k-1\}
$$
The $k$-Fibonacci numbers sequence $(F_{n,k}(a_0,\ldots,a_{k-1}))$
with initial conditions 
$a_0,\ldots,a_{k-1}$ are obtained from this polynomial sequence
by substituting $x$ by $1$ in the sequence
$(F^{(2)}_{n,k}(a_0,\ldots,a_{k-1};x))$. This polynomial sequence is called
the $k$-Fibonacci polynomial sequence of the second kind with initial conditions
$a_0,\ldots,a_{k-1}$. 

\end{definition}

\textit{Case $k=2$}

Table of the first polynomial terms of sequence $(F^{(2)}_{n,2}(0,1;x))$
$$
\begin{array}{l}
F^{(2)}_{0,2}(0,1;x)=0\\
F^{(2)}_{1,2}(0,1;x)=1\\
F^{(2)}_{2,2}(0,1;x)=x\\
F^{(2)}_{3,2}(0,1;x)=x^2+1\\
F^{(2)}_{4,2}(0,1;x)=x^3+2x=x(x^2+2)\\
F^{(2)}_{5,2}(0,1;x)=x^4+3x^2+1\\
F^{(2)}_{6,2}(0,1;x)=x^5+4x^3+3x\\
\end{array}
$$
Table of the first polynomial terms of sequence $(F^{(2)}_{n,2}(1,0;x))$
$$
\begin{array}{l}
F^{(2)}_{0,2}(1,0;x)=1\\
F^{(2)}_{1,2}(1,0;x)=0\\
F^{(2)}_{2,2}(1,0;x)=1\\
F^{(2)}_{3,2}(1,0;x)=x\\
F^{(2)}_{4,2}(1,0;x)=x^2+1\\
F^{(2)}_{5,2}(1,0;x)=x^3+2x=x(x^2+2)\\
F^{(2)}_{6,2}(1,0;x)=x^4+3x^2+1\\
\end{array}
$$

\begin{property}\label{p6.15}
Let $n$ be an integer which is greater than $2$. We have ($n\geq 2$):
$$
F^{(2)}_{n,2}(1,0;x)={\displaystyle\sum^{\lfloor\frac{n-2}{2}\rfloor}_{k=0}}
{n-k-2\choose k}x^{n-2k-2}
$$
and ($n\geq 0$) 
$$
F^{(2)}_{n,2}(0,1;x)=F^{(2)}_{n+1,2}(1,0;x)
$$
or ($n\geq 1$)
$$
F^{(2)}_{n-1,2}(0,1;x)=F^{(2)}_{n,2}(1,0;x)
$$
\end{property}

\begin{proof}
Let prove the first part of Property \ref{p6.15} by induction on the
integer $n\geq 2$. We have
$$
F^{(2)}_{2,2}(1,0;x)=1={\displaystyle\sum^0_{k=0}}
{2-k-2\choose k}x^{2-2k-2}
$$ 
Thus, we verify that Property \ref{p6.15} is true for $n=2$. Let assume
that Property \ref{p6.15} is true up to an integer $n\geq 2$. Using
Definition \ref{d6.14}, we have ($n\geq 1$)
$$
F^{(2)}_{n+1,2}(1,0;x)=xF^{(2)}_{n,2}(1,0;x)+F^{(2)}_{n-1,2}(1,0;x)
$$
So, using the assumption, it comes that
$$
F^{(2)}_{n+1,2}(1,0;x)={\displaystyle\sum^{\lfloor\frac{n-2}{2}\rfloor}_{k=0}}
{n-k-2\choose k}x^{n-2k-1}+{\displaystyle\sum^{\lfloor\frac{n-3}{2}\rfloor}_{k=0}}
{n-k-3\choose k}x^{n-2k-3}
$$
Performing the change of label $k\rightarrow m=k+1$, after renaming
$m$ by $k$, we have
$$
F^{(2)}_{n+1,2}(1,0;x)={\displaystyle\sum^{\lfloor\frac{n-2}{2}\rfloor}_{k=0}}
{n-k-2\choose k}x^{n-2k-1}+{\displaystyle\sum^{\lfloor\frac{n-1}{2}\rfloor}_{k=1}}
{n-k-2\choose k-1}x^{n-2k-1}
$$
$$
F^{(2)}_{n+1,2}(1,0;x)=x^{n-1}+{\displaystyle\sum^{\lfloor\frac{n-2}{2}\rfloor}_{k=1}}
{n-k-2\choose k}x^{n-2k-1}+{\displaystyle\sum^{\lfloor\frac{n-1}{2}\rfloor}_{k=1}}
{n-k-2\choose k-1}x^{n-2k-1}
$$
Or
$$
\lfloor\frac{n-1}{2}\rfloor=\left\{\begin{array}{ccc}
\lfloor\frac{n-2}{2}\rfloor & if & n\equiv 0\pmod 2
\\\\
\lfloor\frac{n-2}{2}\rfloor+1 & if & n\equiv 1\pmod 2
\end{array}\right.
$$
If $n$ is even, then we have
$$
F^{(2)}_{n+1,2}(1,0;x)=x^{n-1}+{\displaystyle\sum^{\lfloor\frac{n-1}{2}\rfloor}_{k=1}}
\left\{{n-k-2\choose k}+{n-k-2\choose k-1}\right\}x^{n-2k-1}
$$
Using the combinatorial identity
$$
{n-k-2\choose k}+{n-k-2\choose k-1}={n-k-1\choose k}
$$
we obtain ($n$ even)
$$
F^{(2)}_{n+1,2}(1,0;x)=x^{n-1}+{\displaystyle\sum^{\lfloor\frac{n-1}{2}\rfloor}_{k=1}}
{n-k-1\choose k}x^{n-2k-1}
$$
$$
F^{(2)}_{n+1,2}(1,0;x)={\displaystyle\sum^{\lfloor\frac{n-1}{2}\rfloor}_{k=0}}
{n-k-1\choose k}x^{n-2k-1}={\displaystyle\sum^{\lfloor\frac{n+1-2}{2}\rfloor}_{k=0}}
{n+1-k-2\choose k}x^{n+1-2k-2}
$$
If $n$ is odd, then we have
$$
F^{(2)}_{n+1,2}(1,0;x)=x^{n-1}+{\displaystyle\sum^{\lfloor\frac{n-2}{2}\rfloor}_{k=1}}
\left\{{n-k-2\choose k}+{n-k-2\choose k-1}\right\}x^{n-2k-1}
$$
$$
+{n-\lfloor\frac{n-1}{2}\rfloor-2\choose \lfloor\frac{n-1}{2}\rfloor-1}
x^{n-2\lfloor\frac{n-1}{2}\rfloor-1}
$$
Using again the combinatorial identity
$$
{n-k-2\choose k}+{n-k-2\choose k-1}={n-k-1\choose k}
$$
it gives ($n$ odd)
$$
F^{(2)}_{n+1,2}(1,0;x)=x^{n-1}+{\displaystyle\sum^{\lfloor\frac{n-2}{2}\rfloor}_{k=1}}
{n-k-1\choose k}x^{n-2k-1}
+{n-\lfloor\frac{n-1}{2}\rfloor-2\choose \lfloor\frac{n-1}{2}\rfloor-1}
x^{n-2\lfloor\frac{n-1}{2}\rfloor-1}
$$
Or ($k>0$)
$$
{n-k-1\choose k}=\frac{n-k-1}{k}{n-k-2\choose k-1}
$$
and 
$$
n-1=\left\{\begin{array}{ccc}
2\lfloor\frac{n-1}{2}\rfloor+1 & if & n\equiv 0\pmod 2
\\\\
2\lfloor\frac{n-1}{2}\rfloor & if & n\equiv 1\pmod 2
\end{array}\right.
$$
In particular, when $n$ is odd, we have
$n-1=2\lfloor\frac{n-1}{2}\rfloor$ and so 
$n-\lfloor\frac{n-1}{2}\rfloor-1=\lfloor\frac{n-1}{2}\rfloor$.
Accordingly, we have
$$
{n-\lfloor\frac{n-1}{2}\rfloor-2\choose \lfloor\frac{n-1}{2}\rfloor-1}
={n-\lfloor\frac{n-1}{2}\rfloor-1\choose \lfloor\frac{n-1}{2}\rfloor}
$$
So, if $n$ is odd ($n>2$), then we have
$$
F^{(2)}_{n+1,2}(1,0;x)=x^{n-1}+{\displaystyle\sum^{\lfloor\frac{n-2}{2}\rfloor}_{k=1}}
{n-k-1\choose k}x^{n-2k-1}
+{n-\lfloor\frac{n-1}{2}\rfloor-1\choose \lfloor\frac{n-1}{2}\rfloor}
x^{n-2\lfloor\frac{n-1}{2}\rfloor-1}
$$
$$
F^{(2)}_{n+1,2}(1,0;x)=x^{n-1}+{\displaystyle\sum^{\lfloor\frac{n-1}{2}\rfloor}_{k=1}}
{n-k-1\choose k}x^{n-2k-1}
$$
$$
F^{(2)}_{n+1,2}(1,0;x)={\displaystyle\sum^{\lfloor\frac{n-1}{2}\rfloor}_{k=0}}
{n-k-1\choose k}x^{n-2k-1}={\displaystyle\sum^{\lfloor\frac{n+1-2}{2}\rfloor}_{k=0}}
{n+1-k-2\choose k}x^{n+1-2k-2}
$$
So, the first part of Property \ref{p6.15} is proved by induction on
the integer $n\geq 2$.

Afterwards, let prove the second part of Property \ref{p6.15} by
induction on the integer $n\geq 0$. We have
$$
F^{(2)}_{0,2}(0,1;x)=0=F^{(2)}_{1,2}(1,0;x)
$$
Thus, we verify that the second part of Property \ref{p6.15} is true
for $n=0$. Let assume that Property \ref{p6.15} is true up to an
integer $n\geq 0$. Using Definition \ref{d6.14}, we have ($n\geq 1$)
$$
F^{(2)}_{n+1,2}(0,1;x)=xF^{(2)}_{n,2}(0,1;x)+F^{(2)}_{n-1,2}(0,1;x)
$$
Using the assumption, it gives ($n\geq 0$)
$$
F^{(2)}_{n+1,2}(0,1;x)=xF^{(2)}_{n+1,2}(1,0;x)+F^{(2)}_{n,2}(1,0;x)
$$
Using again Definition \ref{d6.14}, we get ($n\geq 0$)
$$
F^{(2)}_{n+1,2}(0,1;x)=F^{(2)}_{n+2,2}(1,0;x)
$$
So, the second part of Property \ref{p6.15} is proved by induction on
the integer $n\geq 0$.
\end{proof}

\begin{property}\label{p6.16}
The generating function of the polynomials $F^{(2)}_{n,2}(1,0;x)$ is
given by
$$
{\mathcal F}^{(2)}_2(1,0;x,y)={\displaystyle\sum^{+\infty}_{n=0}}
F^{(2)}_{n,2}(1,0;x)y^n=\frac{1-xy}{1-xy-y^2}
$$
where
$$
y\neq\frac{-x\pm\sqrt{x^2+4}}{2}
$$
\end{property}

\begin{proof}
The generating function of the polynomials $F^{(2)}_{n,2}(1,0;x)$ is
defined by:
$$
{\mathcal F}^{(2)}_2(1,0;x,y)={\displaystyle\sum^{+\infty}_{n=0}}
F^{(2)}_{n,2}(1,0;x)y^n
$$
Since $F^{(2)}_{0,2}(1,0;x)=1$ and since $F^{(2)}_{1,2}(1,0;x)=0$, we
have
$$
{\mathcal F}^{(2)}_2(1,0;x,y)=1+{\displaystyle\sum^{+\infty}_{n=2}}
F^{(2)}_{n,2}(1,0;x)y^n
$$
Using Definition \ref{d6.14}, we have
$$
{\mathcal F}^{(2)}_2(1,0;x,y)=1+{\displaystyle\sum^{+\infty}_{n=2}}
(xF^{(2)}_{n-1,2}(1,0;x)+F^{(2)}_{n-2,2}(1,0;x))y^n
$$
$$
{\mathcal F}^{(2)}_2(1,0;x,y)=1+x{\displaystyle\sum^{+\infty}_{n=2}}
F^{(2)}_{n-1,2}(1,0;x)y^n+{\displaystyle\sum^{+\infty}_{n=2}}F^{(2)}_{n-2,2}(1,0;x)y^n
$$
Or
$$
{\displaystyle\sum^{+\infty}_{n=2}}F^{(2)}_{n-1,2}(1,0;x)y^n
={\displaystyle\sum^{+\infty}_{n=1}}F^{(2)}_{n,2}(1,0;x)y^{n+1}
$$
where we performed the change of label $n\rightarrow m=n-1$ and after
we renamed $m$ by $n$.
Moreover, we have
$$
{\displaystyle\sum^{+\infty}_{n=2}}F^{(2)}_{n-2,2}(1,0;x)y^n
={\displaystyle\sum^{+\infty}_{n=0}}F^{(2)}_{n,2}(1,0;x)y^{n+2}
$$
where we performed the change of label $n\rightarrow l=n-2$ and after
we renamed $l$ by $n$. It results that
$$
{\mathcal F}^{(2)}_2(1,0;x,y)=1+xy{\displaystyle\sum^{+\infty}_{n=1}}
F^{(2)}_{n,2}(1,0;x)y^n+y^2{\displaystyle\sum^{+\infty}_{n=0}}F^{(2)}_{n,2}(1,0;x)y^n
$$
$$
{\mathcal F}^{(2)}_2(1,0;x,y)=1+xy({\mathcal F}^{(2)}_2(1,0;x,y)-1)
+y^2{\mathcal F}^{(2)}_2(1,0;x,y)
$$
$$
(1-xy-y^2){\mathcal F}^{(2)}_2(1,0;x,y)=1-xy
$$
Therefore
$$
{\mathcal F}^{(2)}_2(1,0;x,y)=\frac{1-xy}{1-xy-y^2}
$$
where
$$
y\neq\frac{-x\pm\sqrt{x^2+4}}{2}
$$
\end{proof}

\begin{property}\label{p6.17}
Let $a_0,a_1$ be two integers. We have
$$
F^{(2)}_{n,2}(a_0,a_1;x)=a_0F^{(2)}_{n,2}(1,0;x)+a_1F^{(2)}_{n,2}(0,1;x),
\,\,\,\forall\,n\geq 0
$$
$$
F^{(2)}_{n,2}(a_0,a_1;x)=a_0F^{(2)}_{n,2}(1,0;x)+a_1F^{(2)}_{n+1,2}(1,0;x),
\,\,\,\forall\,n\geq 0
$$
\end{property}

\begin{proof}
Let prove the first part of Property \ref{p6.17} by induction on the
integer $n\geq 0$. The second part of Property \ref{p6.17} follows from
Property \ref{p6.15}. Since $F^{(2)}_{0,2}(1,0;x)=1$ and since
$F^{(2)}_{0,2}(0,1;x)=0$, we have 
$$
F^{(2)}_{0,2}(a_0,a_1;x)=a_0=a_0F^{(2)}_{0,2}(1,0;x)+a_1F^{(2)}_{0,2}(0,1;x)
$$
Thus, we verify that the first part of Property \ref{p6.17} is true for
$n=0$. Let assume that the first part of Property \ref{p6.17} is true
up to an integer $n\geq 0$. Using Definition \ref{d6.14}, we have
$$
F^{(2)}_{n+1,2}(a_0,a_1;x)=xF^{(2)}_{n,2}(a_0,a_1;x)+F^{(2)}_{n-1,2}(a_0,a_1;x)
$$
Using the assumption, we have
$$
F^{(2)}_{n+1,2}(a_0,a_1;x)=x(a_0F^{(2)}_{n,2}(1,0;x)+a_1F^{(2)}_{n,2}(0,1;x))
+a_0F^{(2)}_{n-1,2}(1,0;x)+a_1F^{(2)}_{n-1,2}(0,1;x)
$$
Rearranging the different terms in the right hand side of the previous
equation, it gives
$$
F^{(2)}_{n+1,2}(a_0,a_1;x)=a_0(xF^{(2)}_{n,2}(1,0;x)+F^{(2)}_{n-1,2}(1,0;x))
+a_1(xF^{(2)}_{n,2}(0,1;x)+F^{(2)}_{n-1,2}(0,1;x))
$$
Using again Definition \ref{d6.14}, we get
$$
F^{(2)}_{n+1,2}(a_0,a_1;x)=a_0F^{(2)}_{n+1,2}(1,0;x)+a_1F^{(2)}_{n+1,2}(0,1;x)
$$
So, the first part of Property \ref{p6.17} is proved by induction on
the integer $n\geq 0$.
\end{proof}

\begin{property}\label{p6.18}

Let $a_0$ and $a_1$ be two numbers and let $n$ be a positive
integer. We have ($n\geq 0$) 
$$
F^{(2)}_{n,2}(a_0,a_1;x)=\left\{\begin{array}{ccc}
na_1\left(\frac{x}{2}\right)^{n-1}-(n-1)a_0\left(\frac{x}{2}\right)^n
& if & x=\pm 2i 
\\\\
\left(\frac{a_0(g(x)-x)+a_1}{2g(x)-x}\right)g(x)^n
+\left(\frac{a_0g(x)-a_1}{2g(x)-x}\right)(x-g(x))^n
& if & x\neq\pm 2i 
\end{array}\right.
$$
In particular, we have
$$
F^{(2)}_{n,2}(0,1;x)=\left\{\begin{array}{ccc}
n\left(\frac{x}{2}\right)^{n-1} & if & x=\pm 2i\\\\
\frac{g(x)^n-(x-g(x))^n}{2g(x)-x} & if & x\neq\pm 2i
\end{array}\right.
$$
$$
F^{(2)}_{n,2}(1,0;x)=\left\{\begin{array}{ccc}
(1-n)\left(\frac{x}{2}\right)^n & if & x=\pm 2i\\\\
\frac{g(x)^{n-1}-(x-g(x))^{n-1}}
{2g(x)-x} & if & x\neq\pm 2i
\end{array}\right.
$$
where
$$
g(x)=\frac{x+\sqrt{x^2+4}}{2}
$$
which verify
$$
g(x)^2=xg(x)+1\qquad g(x)(g(x)-x)=1
$$
We have also
$$
g(x)^2+1=g(x)(2g(x)-x)\qquad (x-g(x))^2+1=-(x-g(x))(2g(x)-x)
$$
\end{property}

\begin{proof}
Property \ref{p6.18} can be proved easily by induction or Property can be proved
in the same way as Property \ref{p11}.
\end{proof}

\begin{theorem}\label{t6.19}
Let $a_0$ and $a_1$ be two numbers and let $n$ and $m$ be two positive
integers. If $x\neq\pm 2i$, then we have
$$
F^{(2)}_{n,2}(a_0,a_1;x)F^{(2)}_{m+1,2}(a_0,a_1;x)
+F^{(2)}_{n-1,2}(a_0,a_1;x)F^{(2)}_{m,2}(a_0,a_1;x)
$$
$$
=(a_0g(x)-a_1)^2F^{(2)}_{m+n,2}(0,1;x)+a_0(2a_1-a_0x)g(x)^{m+n}
$$
Otherwise, we have
$$
F^{(2)}_{n,2}(a_0,a_1;x)F^{(1)}_{m+1,2}(a_0,a_1;x)
+F^{(2)}_{n-1,2}(a_0,a_1;x)F^{(2)}_{m,2}(a_0,a_1;x)
$$
$$
=(m+n-2)a^2_0\left(\frac{x}{2}\right)^{m+n+1}-2(m+n-1)a_0a_1
\left(\frac{x}{2}\right)^{m+n}
+(m+n)a^2_1\left(\frac{x}{2}\right)^{m+n-1}
$$
In particular, whatever $x$ is, we have
$$
F^{(2)}_{n,2}(0,1;x)F^{(2)}_{m+1,2}(0,1;x)
+xF^{(2)}_{n-1,2}(0,1;x)F^{(2)}_{m,2}(0,1;x)=F^{(2)}_{m+n,2}(0,1;x)
$$
\end{theorem}

\begin{proof}
Theorem \ref{t6.19} stems from Property \ref{p6.18}.
\end{proof}

\textit{Case $k=3$}

Table of the first polynomial terms of sequence $(F^{(2)}_{n,3}(1,0,0;x))$
$$
\begin{array}{l}
F^{(2)}_{0,3}(1,0,0;x)=1\\
F^{(2)}_{1,3}(1,0,0;x)=0\\
F^{(2)}_{2,3}(1,0,0;x)=0\\
F^{(2)}_{3,3}(1,0,0;x)=1\\
F^{(2)}_{4,3}(1,0,0;x)=x\\
F^{(2)}_{5,3}(1,0,0;x)=x^2\\
F^{(2)}_{6,3}(1,0,0;x)=1+x^3\\
F^{(2)}_{7,3}(1,0,0;x)=2x+x^4=x(2+x^3)\\
F^{(2)}_{8,3}(1,0,0;x)=3x^2+x^5=x^2(3+x^3)
\end{array}
$$ 

\begin{property}\label{p6.20}
Let $n$ be an integer which is greater than $2$. We have ($n\geq 3$):
$$
F^{(2)}_{n,3}(1,0,0;x)={\displaystyle\sum^{\lfloor\frac{n-3}{3}\rfloor}_{k=0}}
{n-2k-3\choose k}x^{n-3k-3}
$$
and ($n\geq 0$) 
$$
F^{(2)}_{n,3}(0,0,1;x)=F^{(2)}_{n+1,3}(1,0,0;x)=F^{(2)}_{n+2,3}(0,1,0;x)
$$
\end{property}

\begin{proof}
Property \ref{p6.20} can be proved in the same way as Property \ref{p6.15}.
\end{proof}

\begin{property}\label{p6.21}
The generating function of the polynomials $F^{(2)}_{n,3}(1,0,0;x)$ is
given by:
$$
{\mathcal F}^{(2)}_3(1,0,0;x,y)={\displaystyle\sum^{+\infty}_{n=0}}
F^{(2)}_{n,3}(1,0,0;x)y^n=\frac{1-xy}{1-xy-y^3}
$$
where
$$
1-xy-y^3\neq 0
$$
\end{property}

\begin{proof}
Property \ref{p6.21} can be proved in the same way as Property \ref{p6.16}.
\end{proof}

\begin{property}\label{p6.22}
Let $a_0,a_1,a_2$ be three integers. We have
$$
F^{(2)}_{n,3}(a_0,a_1,a_2;x)=a_0F^{(2)}_{n,3}(1,0,0;x)+a_1F^{(2)}_{n,3}(0,1,0;x)
+a_2F^{(2)}_{n,2}(0,0,1;x),
\,\,\,\forall\,n\geq 0
$$
\end{property}

\begin{theorem}\label{t6.23}
Let $m$ be an integer which is greater than $2$ and let $n$ be a non-zero
positive integer. For $n\geq m$, we have
$$
F^{(2)}_{n,m}(1,0,\ldots,0;x)={\displaystyle\sum^{\lfloor\frac{n-m}{m}\rfloor}_{k=0}}
{n-(m-1)k-m\choose k}x^{n-m(k+1)}
$$
Moreover, for $n\geq 0$, we have
$$
F^{(2)}_{n+1,m}(1,0,\ldots,0)=F^{(2)}_{n,m}(0,\ldots,0,1;x)
$$
and for $n\geq 0$ with $i\in\{2,\ldots,m-1\}$ when $m>2$, we have
$$
F^{(2)}_{n+i,m}(0,\ldots,0_{i-2},1_{i-1},0_i,\ldots,0;x)=F^{(2)}_{n,m}(0,\ldots,0,1;x)
$$
where $0_l$ means $a_l=0$ with $l\in\{i-2,i\}$ and $1_{i-1}$ means $a_{i-1}=1$ in 
$$
F^{(2)}_{n,m}(a_0,\ldots,a_{i-2},a_{i-1},a_i,\ldots,a_{m-1};x)
$$ 
\end{theorem}

\begin{proof}
Theorem \ref{t6.23} can be proved in the same way as Property \ref{p6.15}.
\end{proof}

\begin{theorem}\label{t6.24}
Let $m$ be an integer which is greater than $2$. 
The generating function of the polynomials $F^{(2)}_{n,m}(1,0,\ldots,0;x)$ is
given by
$$
{\mathcal F}^{(2)}_m(1,0,\ldots,0;x,y)={\displaystyle\sum^{+\infty}_{n=0}}
F^{(2)}_{n,m}(1,0,\ldots,0;x)y^n=\frac{1-xy}{1-xy-y^m}
$$
where
$$
1-xy-y^m\neq 0
$$
\end{theorem}

\begin{proof}
Theorem \ref{t6.24} can be proved in the same way as Property \ref{p6.3}.
\end{proof}

\begin{theorem}\label{t6.25}
Let $m$ be an integer which is greater than $2$ and let
$a_0,a_1,\ldots,a_{m-1}$ be $m$ integers. We have 
$$
F^{(2)}_{n,m}(a_0,a_1,\ldots,a_{m-1};x)=a_0F^{(2)}_{n,m}(1,0,\ldots,0;x)
+a_1F^{(2)}_{n,m}(0,1,0,\ldots,0;x)
$$
$$
+\ldots+a_{m-1}F^{(2)}_{n,m}(0,\ldots,0,1;x),
\,\,\,\forall\,n\geq 0
$$
\end{theorem}

\begin{proof}
Theorem \ref{t6.25} can be proved in the same way as Property \ref{p6.4}.
\end{proof}

The results presented in this section can be related to
other class of sequences as in \cite{hmsb}. 

\section{{Acknowledgements}}

The authors would like to thank Zakir Ahmed, Sudarson Nath and Ankush
Goswami for helpful comments related to the first two sections. The authors are grateful to the anonymous referee for helpful comments.

\end{document}